\documentclass[10pt, leqno]{amsart}
\usepackage{amsfonts,amssymb}

\usepackage{amsmath}
\usepackage{amsthm}
\usepackage{amsrefs}
\usepackage{qsymbols}
\usepackage{latexsym}
\usepackage{chngcntr}
\usepackage[noadjust]{cite}
\usepackage{paralist}
\usepackage{esint}

\newtheorem{theorem}{Theorem}[section]

\newtheorem{lemma}[theorem]{Lemma}
\newtheorem{proposition}[theorem]{Proposition}
\newtheorem{corollary}[theorem]{Corollary}
\theoremstyle{definition}

\newtheorem{remark}[theorem]{Remark}
\newtheorem{problem}[theorem]{Problem}

\newcommand{\IR}{\mathbb{R}}
\newcommand{\IC}{\mathbb{C}}
\newcommand{\IN}{\mathbb{N}}

\newcommand{\IP}{\mathbb{P}}
\newcommand{\IQ}{\mathbb{Q}}

\newcommand{\II}{\mathbb{I}}

\newcommand{\cW}{\mathcal{W}}

\newcommand{\cH}{\mathcal{H}}

\newcommand{\cM}{\mathcal{M}}
\newcommand{\cN}{\mathcal{N}}

\newcommand{\cQ}{\mathcal{Q}}

\newcommand{\cR}{\mathcal{R}}
\newcommand{\cA}{\mathcal{A}}
\newcommand{\cB}{\mathcal{B}}
\newcommand{\cC}{\mathcal{C}}
\newcommand{\cL}{\mathcal{L}}

\newcommand{\cV}{\mathcal{V}}

\newcommand{\fg}{\mathfrak{g}}
\renewcommand{\L}{\mathrm{L}}
\newcommand{\C}{\mathrm{C}}

\renewcommand{\H}{\mathrm{H}}

\renewcommand{\S}{\mathrm{S}}
\newcommand{\W}{\mathrm{W}}

\newcommand{\T}{\mathrm{T}}

\newcommand{\fa}{\mathfrak{a}}
\newcommand{\fb}{\mathfrak{b}}

\newcommand{\e}{\mathrm{e}}
\newcommand{\ii}{\mathrm{i}}
\renewcommand{\d}{\mathrm{d}}
\newcommand{\eps}{\varepsilon}

\renewcommand\Re{\operatorname{Re}}

\newcommand{\Lop}{\mathcal{L}}

\newcommand{\divergence}{\operatorname{div}}

\DeclareMathOperator{\dist}{dist}
\DeclareMathOperator{\tr}{tr}
\DeclareMathOperator{\diam}{diam}

\DeclareMathOperator{\Id}{Id}
\DeclareMathOperator{\Rg}{\mathcal{R}}

\DeclareMathOperator{\dom}{\mathcal{D}}

\numberwithin{equation}{section}

\title[Remarks on the pressure and the Stokes problem in convex domains]{The Stokes resolvent problem: Optimal pressure estimates and remarks on resolvent estimates in convex domains}
\author{Patrick Tolksdorf}
\address{Institut f\"ur Mathematik, Johannes Gutenberg-Universit\"at Mainz, Staudingerweg 9, 55099 Mainz, Germany}
\email{tolksdorf@uni-mainz.de}

\date{\today}
\thanks{The author was supported by the  project  ANR INFAMIE (ANR-15-CE40-0011)}

\begin{document}
\begin{abstract}
The Stokes resolvent problem $\lambda u - \Delta u + \nabla \phi = f$ with $\divergence(u) = 0$ subject to homogeneous Dirichlet or homogeneous Neumann-type boundary conditions is investigated. In the first part of the paper we show that for Neumann-type boundary conditions the operator norm of $\L^2_{\sigma} (\Omega) \ni f \mapsto \pi \in \L^2 (\Omega)$ decays like $\lvert \lambda \rvert^{- 1 / 2}$ which agrees exactly with the scaling of the equation. In comparison to that, we show that the operator norm of this mapping under Dirichlet boundary conditions decays like $\lvert \lambda \rvert^{- \alpha}$ for $0 \leq \alpha < 1 / 4$ and we show that this decay rate cannot be improved to any exponent $\alpha > 1 / 4$, thereby, violating the natural scaling of the equation. In the second part of this article, we investigate the Stokes resolvent problem subject to homogeneous Neumann-type boundary conditions if the underlying domain $\Omega$ is convex. Invoking a famous result of Grisvard~\cite{Grisvard}, we show that weak solutions $u$ with right-hand side $f \in \L^2 (\Omega ; \IC^d)$ admit $\H^2$-regularity and further prove localized $\H^2$-estimates for the Stokes resolvent problem. We prove a generalized version of Shen's $\L^p$-extrapolation theorem~\cite{Shen-Riesz_transform} which can be seen as a version suitable for subspaces of $\L^p$ and combine this result with the localized $\H^2$-estimates to establish optimal resolvent estimates and gradient estimates in $\L^p (\Omega ; \IC^d)$ for $2d / (d + 2) < p < 2d / (d - 2)$ (with $1 < p < \infty$ if $d = 2$). This interval is larger than the known interval for resolvent estimates subject to Dirichlet boundary conditions~\cite{Shen} on general Lipschitz domains and is to the best knowledge of the author the first result that provides $\L^p$-estimates for the Stokes resolvent subject to Neumann-type boundary conditions on general convex domains.
\end{abstract}
\maketitle

\section{Introduction}

\noindent The main object under investigation is the Stokes resolvent problem in a bounded domain $\Omega \subset \IR^d$
\begin{align} \tag{Res} \label{Res}
 \left\{ \begin{aligned}
          \lambda u - \Delta u + \nabla \phi &= f && \text{in } \Omega \\
          \divergence(u) &= 0 && \text{in } \Omega.
         \end{aligned} \right.
\end{align}
The resolvent parameter $\lambda$ is supposed to be contained in a sector $\S_{\theta}$, $\theta \in [0 , \pi)$, in the complex plane, i.e., $\S_{\theta} := \{ z \in \IC \setminus \{ 0 \} : \lvert \arg(z) \rvert < \theta \}$ if $\theta \in (0 , \pi)$ and $\S_0 := (0 , \infty)$. In this article, this system is complemented with two different types of boundary conditions. There is the Dirichlet boundary condition
\begin{align} \tag{Dir} \label{Dir}
\begin{aligned}
 u &= 0 && \text{on } \partial \Omega
\end{aligned}
\end{align}
and there is a family of Neumann-type boundary conditions which reads
\begin{align} \tag{Neu} \label{Neu}
\begin{aligned}
 \{ D u + \mu [D u]^{\top} \} n - \phi n &= 0 && \text{on } \partial \Omega.
\end{aligned}
\end{align}
Here $\mu \in (-1 , 1]$ is a parameter, $n$ denotes the outward unit normal to $\Omega$, and $D u$ the Jacobi-matrix of $u$. There is a tremendous literature on these equations on different types of domains, see, e.g.,~\cite{Abels, Borchers_Sohr, Bolkart_Giga_Miura_Suzuki_Tsutsui, Farwig_Sohr, Geissert_Hess_Hieber_Schwarz_Stavrakidis, Giga, Mitrea_Monniaux, Mitrea_Monniaux_Wright, Shen, Tolksdorf, Tolksdorf_Watanabe} to mention only a few. Notice that the Neumann-type boundary condition with $\mu = 1$ plays an eminent role in the study of problems involving a free boundary~\cite{Abels_free, Beale, Grubb_Solonnikov, Saito, Solonnikov_free} and that the condition with $\mu = 0$ is central in the study of inhomogeneous boundary value problems involving the Stokes equations~\cite{Fabes_Kenig_Verchota, Mitrea_Wright, Shen}. In this article, we investigate two different questions: 
\subsection*{Question 1:}
The first question deals with the behavior of the operator norm of the mapping $f \mapsto \phi$ with respect to $\lambda$, i.e., we seek an inequality of the form
\begin{align*}
 \| \phi \|_{\L^2 (\Omega)} \leq C (\lambda) \| f \|_{\L^2 (\Omega ; \IC^d)} \qquad (f \in \L^2_{\sigma} (\Omega))
\end{align*}
and we would like to know what the \textit{exact} behavior of the constant $C(\lambda)$ is with respect to $\lambda$. Notice that in the case of homogeneous Dirichlet boundary conditions the pressure $\phi$ is unique up to an additive constant so that we assume its mean value to be zero. Notice further that the space of solenoidal $\L^2$-integrable functions differs depending on whether~\eqref{Res} is considered with condition~\eqref{Dir} or~\eqref{Neu}, cf.\@ Section~\ref{Sec: The Stokes operator}. However, during this introduction, we will only use one notation having in mind the difference of the two spaces. \par
Pressure estimates of the Stokes resolvent problem are studied in the engineering literature~\cite{Luhar_Sharma_McKeon} and they also appear in the study of the Stokes operator~\cite{Geissert_Heck_Hieber, Geissert_Hess_Hieber_Schwarz_Stavrakidis, Noll_Saal, Tolksdorf_Watanabe}. Another interesting application can be found in the analysis of the discrete Stokes resolvent problem, comparable to the Poisson case in \cite{Leykekhman_Vexler, Thomee}, where the pressure appears in the derivation of weighted norm estimates. \par
To obtain an idea of what the right behavior of $C(\lambda)$ with respect to $\lambda$ would be, set for a moment $\Omega = \IR^d$. In this case, the solutions $u$ and $\phi$ satisfy the following scaling property: Let $r > 0$ and assume that $u$ and $\phi$ solve~\eqref{Res} for some resolvent parameter $\lambda$ and right-hand side $f$. Then, $u_r := u (r \cdot)$ and $\phi_r := r \phi (r \cdot)$ solve~\eqref{Res} for the resolvent parameter $r^2 \lambda$ and right-hand side $f_r := r^2 f(r \cdot)$. Put $r := \lvert \lambda \rvert^{- 1 / 2}$ so that $\lvert r^2 \lambda \rvert = 1$. If there \textit{would} be a constant $C > 0$ (which on the whole space certainly does not have to be true) such that
\begin{align*}
 \| \phi_r \|_{\L^2 (\IR^d)} \leq C \| f_r \|_{\L^2 (\IR^d ; \IC^d)}
\end{align*}
holds, then the substitution rule ensures the estimate
\begin{align}
\label{Eq: Scaling estimate}
\| \phi \|_{\L^2 (\IR^d)} \leq C \lvert \lambda \rvert^{- 1 / 2} \| f \|_{\L^2 (\IR^d ; \IC^d)}.
\end{align}
We will show in Section~\ref{On uniform pressure estimates} that this behavior of $C(\lambda)$ is \textit{false} on bounded $\C^4$-domains and if homogeneous Dirichlet boundary conditions are imposed. More precisely, it is known~\cite{Noll_Saal, Tolksdorf_Watanabe} that $C(\lambda)$ satisfies for each $0 \leq \alpha < 1 / 4$ and some constant $C > 0$ independent of $\lambda$
\begin{align}
\label{Eq: C(lambda) inequality}
 C(\lambda) \leq C \lvert \lambda \rvert^{- \alpha},
\end{align}
see also Proposition~\ref{Prop: Pressure estimate Dirichlet}. In Proposition~\ref{Prop: Sharpness on L2} we show that the condition $\alpha < 1 / 4$ is sharp in the sense that for no $\alpha > 1 / 4$ there exists a constant $C > 0$ independent of $\lambda$ such that~\eqref{Eq: C(lambda) inequality} is valid. This shows, that the presence of a boundary causes the pressure to behave differently than its natural scaling would dictate. \par
In contrast to that, under boundary condition~\eqref{Neu}, then on each domain $\Omega$ with a sufficiently nice boundary, e.g., bounded $\C^{1 , 1}$-domains or bounded convex domains, we show that $C(\lambda)$ satisfies~\eqref{Eq: C(lambda) inequality} with $\alpha = 1 / 2$, see Proposition~\ref{Prop: Pressure estimate Neumann}. Thus, depending on the particular boundary condition at stake, the behavior of the pressure with respect to $\lambda$ might differ. \par
For both boundary conditions, we perform a similar analysis in which the $\L^2$-norm of $f$ on the right-hand side is replaced by the $\H^{-1}$-norm of $f$, see Propositions~\ref{Prop: Pressure estimate in H-1} and~\ref{Prop: Failure of pressure estimate}. For simplicity, we considered only $\L^2$-based spaces. An extension to the $\L^p$-situation should be straightforward. Notice that the exponent $\alpha$ for which the pressure estimates in $\L^p$ are valid satisfies the relation $\alpha < 1 / 2 - 1 / (2 p)$, see~\cite{Noll_Saal}, so that the decay estimate with exponent $\alpha > 1 / 2 - 1 / (2 p)$ should fail.

\subsection*{Question 2:}
If $\Omega \subset \IR^d$, $d \geq 3$, is a bounded Lipschitz domain the resolvent estimate 
\begin{align}
\label{Eq: Resolvent estimate}
 \lvert \lambda \rvert \| u \|_{\L^p (\Omega ; \IC^d)} \leq C \| f \|_{\L^p (\Omega ; \IC^d)} \qquad (f \in \L^p_{\sigma} (\Omega))
\end{align}
was proven for solutions to~\eqref{Res} subject to the boundary condition~\eqref{Dir} in the seminal paper of Shen~\cite{Shen}. Here, $p$ satisfies
\begin{align}
\label{Eq: p condition Lipschitz}
 \Big\lvert \frac{1}{p} - \frac{1}{2} \Big\rvert < \frac{1}{2 d} + \eps
\end{align}
for some $\eps > 0$ depending only on $d$, $\theta$, and the Lipschitz geometry. A special class of bounded Lipschitz domains are bounded convex domains and one might wonder, whether the condition~\eqref{Eq: p condition Lipschitz} on $p$ improves if convexity of $\Omega$ is imposed. It was for example proven by Geng and Shen~\cite{Geng_Shen} that on bounded and convex domains the Helmholtz projection gives rise to a bounded projection on $\L^p (\Omega ; \IC^d)$ for all $1 < p < \infty$. Moreover, a work of Geissert, Heck, Hieber, and Sawada~\cite{Geissert_Heck_Hieber_Sawada} formalizes the philosophy that the boundedness of the Helmholtz projection implies functional analytic properties of the Stokes operator like~\eqref{Eq: Resolvent estimate} at least under the condition that $\Omega$ is a (not necessarily bounded) uniform $\C^3$-domain. Combining the result of~\cite{Geng_Shen} with this philosophy leads to the conjecture that the resolvent estimate~\eqref{Eq: Resolvent estimate} should be valid \textit{for all} $1 < p < \infty$ if $\Omega$ is convex. This is a question that was raised by Maz'ya in~\cite[Prob.~66]{Mazya}. \par
We give first results in this direction for the Stokes resolvent problem~\eqref{Res} subject to the Neumann-type boundary condition~\eqref{Neu} but we restrict the interval of parameters $\mu$ to be $(-1 , \sqrt{2} - 1)$. This still includes the case $\mu = 0$ but unfortunately excludes the physically important case $\mu = 1$. The corresponding results are explained as follows. \par 
By virtue of a famous formula of integration by parts by Grisvard~\cite[Thm.~3.1.1.1]{Grisvard} we establish the estimate
\begin{align}
\label{Eq: H^2 resolvent estimate}
 \lvert \lambda \rvert \int_{\Omega} \lvert \nabla u \rvert^2 \; \d x + \int_{\Omega} \lvert \nabla^2 u \rvert^2 \; \d x + \int_{\Omega} \lvert \nabla \phi \rvert^2 \; \d x \leq C \bigg( \int_{\Omega} \lvert f \rvert^2 \; \d x + \lvert \lambda \rvert^2 \int_{\Omega} \lvert u \rvert^2 \; \d x \bigg)
\end{align}
for some constant $C > 0$ depending only on $d$, $\theta$, and $\mu$, see Theorem~\ref{Thm: H2 regularity on convex domains}. In particular, this implies that solutions $u$ and $\phi$ to $- \Delta u + \nabla \phi = f$ and $\divergence(u) = 0$ for some $f \in \L^2_{\sigma} (\Omega)$ and subject to the boundary condition~\eqref{Neu} satisfy $u \in \H^2 (\Omega ; \IC^d)$ and $\phi \in \H^1 (\Omega)$. This should be compared with the results of Kellogg and Osborn~\cite{Kellogg_Osborn}, Dauge~\cite{Dauge}, and Maz'ya and Rossmann~\cite{Mazya_Rossmann} in the case of the boundary condition~\eqref{Dir} and convex polygonal/polyhedral domains. For general bounded and convex domains this higher regularity property in the case of homogeneous Dirichlet boundary conditions is unknown. \par
We continue by establishing a localized version of~\eqref{Eq: H^2 resolvent estimate} which can be found in Proposition~\ref{Prop: Reverse regularity estimates}. Combining this with a Caccioppoli type estimate, see Lemma~\ref{Lem: Caccioppoli}, and Sobolev's embedding yields the validity of a weak reverse H\"older estimate of the form
\begin{align}
\label{Eq: Reverse Holder introduction}
 \bigg( \frac{1}{r^d} \int_{\Omega \cap Q (x_0 , r)} \big\{ \lvert \lambda \rvert \lvert u \rvert &+ \lvert \lambda \rvert^{1 / 2} \lvert \nabla u \rvert + \lvert \lambda \rvert^{1 / 2} \lvert \phi \rvert \big\}^p \; \d x \bigg)^{1 / p} \\
 &\leq C \bigg( \frac{1}{r^d} \int_{\Omega \cap Q (x_0 , 2 r)} \big\{ \lvert \lambda \rvert \lvert u \rvert + \lvert \lambda \rvert^{1 / 2} \lvert \nabla u \rvert + \lvert \lambda \rvert^{1 / 2} \lvert \phi \rvert \big\}^2 \; \d x \bigg)^{1 / 2},
\end{align}
where $Q (x_0 , r)$ is a cube in $\IR^d$ with midpoint $x_0$ and diameter $r > 0$, where $p$ satisfies $2 < p < \infty$ if $d = 2$ and $p = 2d / (d - 2)$ if $d \geq 3$, and where $u$ and $\phi$ solve the Stokes resolvent problem with a right-hand side $f \in \L^2 (\Omega ; \IC^d)$ that vanishes on $\Omega \cap Q (x_0 , 2 r)$. One could now conclude by an $\L^p$-extrapolation theorem of Shen~\cite{Shen-Riesz_transform} that the family of (sublinear) operators
\begin{align*}
 \L^q (\Omega ; \IC^d) \ni f \mapsto \lvert \lambda \rvert \lvert u \rvert &+ \lvert \lambda \rvert^{1 / 2} \lvert \nabla u \rvert + \lvert \lambda \rvert^{1 / 2} \lvert \phi \rvert
\end{align*}
is uniformly bounded with respect to $\lambda$ on $\L^q$, where $2 < q < 2d / (d - 2)$ if(!) the family of operators
\begin{align}
\label{Eq: Wrong uniform boundedness}
 T_{\lambda} : \L^2 (\Omega ; \IC^d) \to \L^2 (\Omega), \quad f \mapsto \lvert \lambda \rvert^{1 / 2} \phi
\end{align}
is uniformly bounded on $\L^2$. This gives a connection to Question~1 discussed above. Unfortunately, only the restriction of $T_{\lambda}$ to solenoidal vector fields satisfies this uniform bound, whereas the operators on all of $\L^2 (\Omega ; \IC^d)$ grow like $\lvert \lambda \rvert^{1 / 2}$. This fact can be easily seen by noting that the pressure $\phi$ solving~\eqref{Res} for general $f \in \L^2 (\Omega ; \IC^d)$ is the sum of the pressure associated to~\eqref{Res} but with the right-hand side $\IQ f \in \L^2_{\sigma} (\Omega)$ and the function $g$ which satisfies $(\IQ - \Id) f = \nabla g$. Here, $\IQ$ denotes the Helmholtz projection on $\L^2 (\Omega ; \IC^d)$. Notice that the function $g$ does not depend on $\lambda$ at all, which explains that the family defined in~\eqref{Eq: Wrong uniform boundedness} cannot be uniformly bounded on $\L^2$. \par
To circumvent this problem, we discuss in Section~\ref{An Lp-extrapolation theorem suitable for subspaces of Lp} a version of Shen's $\L^p$-extrapolation theorem, which is valid for \textit{subspaces} of $\L^p$. This allows us to employ the uniform boundedness of the restriction of the operators $T_{\lambda}$ to solenoidal spaces and delivers the following theorem which is proven in Section~\ref{Sec: Resolvent estimates on convex domains}.

\begin{theorem}
\label{Thm: Main}
Let $\Omega \subset \IR^d$, $d \geq 2$, be a bounded and convex domain and $r_0 > 0$ be such that $B(0 , r_0) \subset \tfrac{1}{2} [\Omega - \{x_0\}]$ for some $x_0 \in \Omega$. Let further $\theta \in [0 , \pi)$, $\mu \in (-1 , \sqrt{2} - 1)$, and let
\begin{align*}
 \Big\lvert \frac{1}{p} - \frac{1}{2} \Big\rvert < \frac{1}{d}.
\end{align*}
Then there exists a constant $C > 0$ such that for all $\lambda \in \S_{\theta}$ and all $f \in \L^2 (\Omega ; \IC^d) \cap \L^p (\Omega ; \IC^d)$ satisfying $\divergence(f) = 0$ in the sense of distributions the solutions $u \in \H^1 (\Omega ; \IC^d)$ and $\phi \in \L^2 (\Omega)$ to
\begin{align*}
 \left\{ \begin{aligned}
  \lambda u - \Delta u + \nabla \phi &= f && \text{in } \Omega \\
  \divergence(u) &= 0 && \text{in } \Omega \\
  \{ D u + \mu [D u]^{\top} \} n - \phi n &= 0 && \text{on } \partial \Omega
 \end{aligned} \right.
\end{align*}
satisfy
\begin{align*}
 \lvert \lambda \rvert \| u \|_{\L^p (\Omega ; \IC^d)} + \lvert \lambda \rvert^{1 / 2} \| \nabla u \|_{\L^p (\Omega ; \IC^{d^2})} \leq C \| f \|_{\L^p (\Omega ; \IC^d)}.
\end{align*}
If $p \geq 2$ it additionally holds
\begin{align*}
 \lvert \lambda \rvert^{1 / 2} \| \phi \|_{\L^p (\Omega)} \leq C \| f \|_{\L^p (\Omega ; \IC^d)}.
\end{align*}
The constant $C > 0$ depends only on $d$, $\theta$, $\mu$, $\diam (\Omega)$, and $r_0$. \par
Furthermore, there exists a constant $C > 0$ such that for all $\lambda \in \S_{\theta}$ and all $F \in \L^2 (\Omega ; \IC^{d \times d}) \cap \L^p (\Omega ; \IC^{d \times d})$ the solutions $u \in \H^1 (\Omega ; \IC^d)$ and $\phi \in \L^2 (\Omega)$ to
\begin{align*}
 \left\{ \begin{aligned}
  \lambda u - \Delta u + \nabla \phi &= \divergence(F) && \text{in } \Omega \\
  \divergence(u) &= 0 && \text{in } \Omega \\
  \{ D u + \mu [D u]^{\top} \} n - \phi n &= 0 && \text{on } \partial \Omega
 \end{aligned} \right.
\end{align*}
satisfy
\begin{align*}
 \lvert \lambda \rvert^{1 / 2} \| u \|_{\L^p (\Omega ; \IC^d)} + \| \nabla u \|_{\L^p (\Omega ; \IC^{d^2})} + \| \phi \|_{\L^p (\Omega)} \leq C \| F \|_{\L^p (\Omega ; \IC^{d \times d})}.
\end{align*}
If $p \geq 2$ it additionally holds
\begin{align*}
 \| \phi \|_{\L^p (\Omega)} \leq C \| F \|_{\L^p (\Omega ; \IC^{d \times d})}.
\end{align*}
Again, the constant $C > 0$ depends only on $d$, $\theta$, $\mu$, $\diam (\Omega)$, and $r_0$.
\end{theorem}

\subsection*{Acknowledgements}
The author likes to thank Niklas Behringer for initiating a discussion on the sharpness of the pressure decay estimates presented in Section~\ref{On uniform pressure estimates} and pointing out the references~\cite{Leykekhman_Vexler, Thomee}. Moreover, he likes to thank Zhongwei Shen for interesting discussions on the Stokes equations in convex domains.

\section{The Stokes operator on $\L^2_{\sigma} (\Omega)$ and $\H^{-1}_{\sigma} (\Omega)$}
\label{Sec: The Stokes operator}

\noindent In the following, we will assume that $\Omega \subset \IR^d$, $d \geq 2$, is a bounded and open domain whose boundary is at least Lipschitz regular, i.e., locally represented as the graph of a Lipschitz continuous function. This section is devoted to present results concerning the Stokes resolvent problem~\eqref{Res} subject to no-slip boundary conditions~\eqref{Dir} and subject to Neumann-type boundary conditions~\eqref{Neu}.

\subsection{Function spaces}

We define the space of compactly supported smooth and solenoidal vector fields in $\Omega$ as 
\begin{align*}
 \C_{c , \sigma}^{\infty} (\Omega) := \{ \varphi \in \C_c^{\infty} (\Omega ; \IC^d) : \divergence(\varphi) = 0  \}
\end{align*}
and the space of solenoidal vector fields that are smooth up to the boundary as
\begin{align*}
 \C_{\sigma}^{\infty} (\overline{\Omega}) := \{ \varphi|_{\Omega} : \varphi \in \C_{c , \sigma}^{\infty} (\IR^d) \}.
\end{align*}
As usual, we define for $1 < p < \infty$
\begin{align*}
 \L^p_{\sigma} (\Omega) := \overline{\C_{c , \sigma}^{\infty} (\Omega)}^{\L^p} \qquad \text{and} \qquad \W^{1 , p}_{0 , \sigma} (\Omega) := \overline{\C_{c , \sigma}^{\infty} (\Omega)}^{\W^{1 , p}}
\end{align*}
endowed with their natural norms. These spaces are usually introduced if the Stokes equations subject to no-slip boundary conditions are studied, e.g., see~\cite{Sohr, Galdi, Mitrea_Monniaux}. If one is interested in Neumann-type boundary conditions, then one defines the spaces
\begin{align*}
 \cL^p_{\sigma} (\Omega) := \{ u \in \L^p (\Omega ; \IC^d) : \divergence(u) = 0 \} \qquad \text{and} \qquad \cW^{1 , p}_{\sigma} (\Omega) := \{ u \in \W^{1 , p} (\Omega ; \IC^d) : \divergence(u) = 0 \}
\end{align*}
endowed with their natural norms, see, e.g.,~\cite{Solonnikov, Mitrea_Monniaux_Wright}. If $p = 2$ we write $\H^1_{0 , \sigma} (\Omega) := \W^{1 , 2}_{0 , \sigma} (\Omega)$ and $\cH^1_{\sigma} (\Omega) := \cW^{1 , 2}_{\sigma} (\Omega)$ henceforth. Notice that $\L^2_{\sigma} (\Omega)$ and $\cL^2_{\sigma} (\Omega)$ do in general not coincide. Indeed, elements $u \in \L^2_{\sigma} (\Omega)$ satisfy $n \cdot u = 0$ on $\partial \Omega$ whereas the mean value of $n \cdot u$ on $\partial \Omega$ vanishes for elements $u$ in $\cL^2_{\sigma} (\Omega)$. Furthermore, notice that $\C^{\infty}_{\sigma} (\overline{\Omega})$ embeds densely into $\cL^2_{\sigma} (\Omega)$, by~\cite[Lem.~2.1, Rem.~2.2]{Mitrea_Monniaux_Wright}. We define the antidual spaces
\begin{align*}
 \H^{-1}_{\sigma} (\Omega) := \H^1_{0 , \sigma} (\Omega)^* \qquad \text{and} \qquad \cH^{-1}_{0 , \sigma} (\Omega) := \cH^1_{\sigma} (\Omega)^*,
\end{align*}
where we consider antilinear functionals instead of linear functionals, i.e., they satisfy $f(\alpha \cdot) = \overline{\alpha} f(\cdot)$ for $\alpha \in \IC$ instead of the usual homogeneity condition. We further define $\H^{-1} (\Omega ; \IC^d) := \H^1_0 (\Omega ; \IC^d)^*$ and $\H^{-1}_0 (\Omega ; \IC^d) := \H^1 (\Omega ; \IC^d)^*$. Notice that the embeddings
\begin{align*}
 \H^1_{0 , \sigma} (\Omega) \hookrightarrow \H^1_0 (\Omega ; \IC^d) \qquad \text{and} \qquad \cH^1_{\sigma} (\Omega) \hookrightarrow \H^1 (\Omega ; \IC^d)
\end{align*}
result in the following embeddings for their dual spaces
\begin{align*}
 \H^{-1} (\Omega ; \IC^d) \hookrightarrow \H^{-1}_{\sigma} (\Omega) \qquad \text{and} \qquad \H^{-1}_0 (\Omega ; \IC^d) \hookrightarrow \cH^{-1}_{0 , \sigma} (\Omega).
\end{align*}
Notice further, that an element $u \in \L^2 (\Omega ; \IC^d)$ can be considered as an element in the spaces of negative order by identifying $u$ with the functional
\begin{align*}
 \Phi (u) (v) := \int_{\Omega} u \cdot \overline{v} \; \d x
\end{align*}
endowed with the respective domain of definition. \par
For $0 < s < 1$ we consider as intermediate spaces the scale of $\L^2$-based Bessel potential spaces $\H^s (\Omega) = \H^{s , 2} (\Omega)$ which are defined as the restriction spaces of Bessel potential spaces on the whole space. The solenoidal counterparts are denoted by $\H^s_{\sigma} (\Omega)$ and are defined to be $\H^s (\Omega ; \IC^d) \cap \L^2_{\sigma} (\Omega)$. If $s > 1 / 2$ we also define the corresponding spaces with vanishing trace, i.e., $\H^s_{0 , \sigma} (\Omega) := \H^s_0 (\Omega ; \IC^d) \cap \L^2_{\sigma} (\Omega)$. In the case of negative indices, we define for $0 < s < 1 / 2$ the space $\H^{- s}_{\sigma} (\Omega) := \H^s_{\sigma} (\Omega)^*$. \par
Having introduced all required function spaces, we are going to introduce the Stokes operators subject to no-slip and Neumann boundary conditions following~\cite{Mitrea_Monniaux, Mitrea_Monniaux_Wright}.

\subsection{The Stokes operator subject to no-slip boundary conditions}

Define the sesquilinear form
\begin{align*}
 \fa : \H^1_{0 , \sigma} (\Omega) \times \H^1_{0 , \sigma} (\Omega) \to \IC, \quad (u , v) \mapsto \int_{\Omega} \nabla u \cdot \overline{\nabla v} \; \d x.
\end{align*}
The weak Stokes operator $\cA$ on $\H^{-1}_{\sigma} (\Omega)$ subject to no-slip boundary conditions is defined as
\begin{align*}
 \dom(\cA) &:= \H^1_{0 , \sigma} (\Omega),\\
 \langle \cA u , v \rangle_{\H^{-1}_{\sigma} , \H^1_{0 , \sigma}} &:= \fa (u , v) \quad \text{for} \quad u \in \dom(\cA) \text{ and } v \in \H^1_{0 , \sigma} (\Omega).
\end{align*}
The Stokes operator $A$ on $\L^2_{\sigma} (\Omega)$ subject to no-slip boundary conditions is then defined as the part of $\cA$ in $\L^2_{\sigma} (\Omega)$, i.e.,
\begin{align*}
 \dom(A) &:= \{ u \in \L^2_{\sigma} (\Omega) : u \in \dom(\cA) \text{ and } \cA u \in \L^2_{\sigma} (\Omega) \}, \\
 A u &:= \cA u \qquad (u \in \dom(A)).
\end{align*}
Elements $u \in \dom(A)$ satisfy no-slip boundary conditions. Notice that the symmetry of $\fa$ implies that $A$ is a self-adjoint operator on $\L^2_{\sigma} (\Omega)$, see~\cite[Thm.~VI.2.23]{Kato}. Furthermore, the definition of $A$ implies for the resolvent sets the inclusion $\rho(\cA) \subset \rho(A)$ and that for $\lambda \in \rho(- \cA)$ it holds
\begin{align}
\label{Eq: Resolvent restriction of weak resolvent}
 (\lambda + A)^{-1} = (\lambda + \cA)^{-1}|_{\L^2_{\sigma} (\Omega)}.
\end{align}

\subsection{The Stokes operator subject to Neumann-type boundary conditions}

Define for $\mu \in (-1 , 1]$ the coefficients $a_{j k}^{\alpha \beta} (\mu) := \delta_{j k} \delta_{\alpha \beta} + \mu \delta_{j \beta} \delta_{k \alpha}$, where $\delta_{\alpha \beta}$ denotes Kronecker's delta. Notice that the divergence form operator with coefficients $a_{j k}^{\alpha \beta} (\mu)$ is formally given by (here and below we sum over repeated indices)
\begin{align*}
 \partial_j a_{j k}^{\alpha \beta} (\mu) \partial_k u_{\beta} = \Delta u_{\alpha} + \mu \partial_{\alpha} \divergence(u).
\end{align*}
Hence, if this operator acts only on solenoidal functions, this in merely the Laplacian. Consequently, defining the sesquilinear form
\begin{align*}
 \fb_{\mu} : \cH^1_{\sigma} (\Omega) \times \cH^1_{\sigma} (\Omega) \to \IC, \quad (u , v) \mapsto \int_{\Omega} a_{j k}^{\alpha \beta} (\mu) \partial_k u_{\beta} \cdot \overline{\partial_j v_{\alpha}} \; \d x
\end{align*}
gives still rise to an operator associated to the Stokes equations. The weak Stokes operator $\cB_{\mu}$ on $\cH^{-1}_{0 , \sigma} (\Omega)$ subject to Neumann-type boundary conditions is defined as
\begin{align*}
 \dom(\cB_{\mu}) &:= \cH^1_{\sigma} (\Omega),\\
 \langle \cB_{\mu} u , v \rangle_{\cH^{-1}_{0 , \sigma} , \cH^1_{\sigma}} &:= \fb_{\mu} (u , v) \quad \text{for} \quad u \in \dom(\cB_{\mu}) \text{ and } v \in \cH^1_{\sigma} (\Omega).
\end{align*}
For the Stokes operator subject to Neumann-type boundary conditions on the negative scale one has to understand the boundary condition very carefully as right-hand sides in $\cH^{-1}_{0 , \sigma} (\Omega)$ could induce inhomogeneous boundary terms. For example the functional
\begin{align*}
 \cH^1_{\sigma} (\Omega) \ni v \mapsto \int_{\partial \Omega} f \cdot \overline{v} \; \d \sigma =: F(v)
\end{align*}
for a smooth function $f$ lies in $\cH^{-1}_{0 , \sigma} (\Omega)$. Thus, the solution to the problem $\cB_{\mu} u = F$ would satisfy an inhomogeneous boundary condition. \par
The Stokes operator $B_{\mu}$ on $\cL^2_{\sigma} (\Omega)$ subject to Neumann-type boundary conditions is then defined as the part of $\cB_{\mu}$ in $\cL^2_{\sigma} (\Omega)$, i.e.,
\begin{align*}
 \dom(B_{\mu}) &:= \{ u \in \cL^2_{\sigma} (\Omega) : u \in \dom(\cB_{\mu}) \text{ and } \cB_{\mu} u \in \cL^2_{\sigma} (\Omega) \}, \\
 B_{\mu} u &:= \cB_{\mu} u \qquad (u \in \dom(B_{\mu})).
\end{align*}
Elements $u \in \dom(B_{\mu})$ formally satisfy the boundary conditions stated in~\eqref{Neu}. Notice that the symmetry of $\fb_{\mu}$ implies that $B_{\mu}$ is a self-adjoint operator on $\cL^2_{\sigma} (\Omega)$, see~\cite[Thm.~VI.2.23]{Kato}. Furthermore, the definition of $B_{\mu}$ implies for the resolvent sets the inclusion $\rho(\cB_{\mu}) \subset \rho(B_{\mu})$ and that for $\lambda \in \rho(- \cB_{\mu})$ it holds
\begin{align*}
 (\lambda + B_{\mu})^{-1} = (\lambda + \cB_{\mu})^{-1}|_{\cL^2_{\sigma} (\Omega)}.
\end{align*}

\subsection{The Laplace operators} Similarly, we introduce the weak Laplace operators $- \Delta_D$ on $\H^{-1} (\Omega)$ and $- \Delta_N$ on $\H^{-1}_0 (\Omega)$ via the sesquilinear form
\begin{align*}
 \cV \times \cV \to \IC, \quad (u , v) \mapsto \int_{\Omega} \nabla u \cdot \overline{\nabla v} \; \d x.
\end{align*}
The domain of the sesquilinear form $\cV$ is taken to be $\H^1_0 (\Omega)$ for the Dirichlet Laplacian and $\H^1 (\Omega)$ for the Neumann Laplacian. Recall that by Poincar\'e's inequality $\Delta_D$ is invertible and that $\Delta_N$ is invertible if considered on the factor space $(\H^1 (\Omega) / \textrm{const})^*$. Finally, recall that if the boundary of $\Omega$ is $\C^{1 , 1}$-regular or if $\Omega$ is convex the operators $\Delta_D^{-1} : \L^2 (\Omega) \to \H^2 (\Omega)$ and $\Delta_N^{-1} : \L^2_0 (\Omega) \to \H^2 (\Omega)$ are bounded, see~\cite[Sec.~3.2.1]{Grisvard} for the particular statements on convex domains, see also the discussion at the beginning of Section~\ref{Sec: Estimates on convex domains}. Here $\L^2_0 (\Omega)$ denotes the $\L^2$-space of average free functions. In the following, we do not distinguish the notation between the weak Laplacians defined on negative spaces or the strong Laplacians defined on $\L^2 (\Omega)$.

\subsection{The Bogovski\u{\i} operator}

Let us consider the divergence problem
\begin{align*}
\left\{
\begin{aligned}
 \divergence(u) &= f && \text{in } \Omega \\
 u &= 0 && \text{on } \partial \Omega,
\end{aligned} \right.
\end{align*}
where $f \in \L^2_0 (\Omega)$ and $\Omega$ is a bounded Lipschitz domain. It is well-known, see, e.g.,~\cite[Ch.~III.3]{Galdi} and the references therein, that there exists a bounded and linear operator $B : \L^2_0 (\Omega) \to \H^1_0 (\Omega ; \IC^d)$ which satisfies $\divergence(B f) = f$. This means that $u := B f$ solves the divergence problem posed above. Clearly, $B$ is a highly non-unique operator as one can always add a function $v \in \H^1_{0 , \sigma} (\Omega)$ to the solution $u$ and still have a solution to the problem. Here and below, the operator $B$ is called the Bogovski\u{\i} operator.

\subsection{The Helmholtz projection}

The Helmholtz projection $\IP : \L^2 (\Omega ; \IC^d) \to \L^2 (\Omega ; \IC^d)$ is introduced as being the orthogonal projection of $\L^2 (\Omega ; \IC^d)$ onto $\L^2_{\sigma} (\Omega)$. Analogously, let $\IQ : \L^2 (\Omega ; \IC^d) \to \L^2 (\Omega ; \IC^d)$ denote the orthogonal projection of $\L^2 (\Omega ; \IC^d)$ onto $\cL^2_{\sigma} (\Omega)$. It is well-known, see~\cite[Sec.~11]{Fabes_Mendez_Mitrea}, that the range of $\Id - \IP$ is given by
\begin{align}
\label{Eq: Orthogonal range Helmholtz P}
 \Rg(\Id - \IP) = \nabla \H^1 (\Omega) := \{ \nabla \varphi : \varphi \in \H^1 (\Omega) \}
\end{align}
and that the range of $\Id - \IQ$ is given by
\begin{align}
\label{Eq: Orthogonal range Helmholtz Q}
 \Rg(\Id - \IQ) = \nabla \H^1_0 (\Omega) := \{ \nabla \varphi : \varphi \in \H^1_0 (\Omega) \}.
\end{align}
Notice that $\IP$ and $\IQ$ can be realized by employing the Neumann and the Dirichlet Laplacian as follows. Define a distribution
\begin{align*}
 \langle \widetilde{\divergence} (u) , v \rangle_{\H^{-1}_0 , \H^1} := - \int_{\Omega} u \cdot \overline{\nabla v} \; \d x \quad \text{for} \quad u \in \L^2 (\Omega ; \IC^d) \text{ and } v \in \H^1 (\Omega),
\end{align*}
which acts as the distribution generated by the divergence operator but ignores the boundary values that would arise due to the integration by parts. Furthermore, define the usual divergence as
\begin{align*}
 \langle \divergence (u) , v \rangle_{\H^{-1} , \H^1_0} := - \int_{\Omega} u \cdot \overline{\nabla v} \; \d x \quad \text{for} \quad u \in \L^2 (\Omega ; \IC^d) \text{ and } v \in \H^1_0 (\Omega).
\end{align*}
Then, $\IP$ and $\IQ$ can be represented as
\begin{align}
\label{Eq: Representation Helmholtz}
 \IP = \Id + \nabla (- \Delta_N)^{-1} \widetilde{\divergence} \qquad \text{and} \qquad  \IQ = \Id + \nabla (- \Delta_D)^{-1} \divergence.
\end{align}
A calculation verifying this identity for $\IP$ can be found in~\cite[Lem.~5.1.3]{Tolksdorf_Dissertation} and for $\IQ$ in the proof of~\cite[Lem.~2.1]{Mitrea_Monniaux_Wright}. We record the following lemma.

\begin{lemma}
Let $\Omega$ be a bounded Lipschitz domain such that $\Delta_D^{-1} : \L^2 (\Omega) \to \H^2 (\Omega)$ and $\Delta_N^{-1} : \L^2_0 (\Omega) \to \H^2 (\Omega)$ are bounded. Then $\IP : \H^1_0 (\Omega ; \IC^d) \to \H^1 (\Omega ; \IC^d)$ and $\IQ : \H^1 (\Omega ; \IC^d) \to \H^1 (\Omega ; \IC^d)$ are bounded operators. In particular, if $\Omega$ is convex, then these operator norms depend at most on the dimension $d$.
\end{lemma}

\begin{proof}
Notice that if $u \in \H^1_0 (\Omega ; \IC^d)$, then $\widetilde{\divergence} (u) = \divergence(u)$ and since $u \in \H^1_0 (\Omega)$ the average of $\divergence(u)$ on $\Omega$ is zero. Thus, the statements concerning the boundedness of $\IP$ and $\IQ$ directly follow by~\eqref{Eq: Representation Helmholtz} and the assumption that $\Delta_D^{-1} : \L^2 (\Omega) \to \H^2 (\Omega)$ and $\Delta_N^{-1} : \L^2_0 (\Omega) \to \H^2 (\Omega)$ are bounded. Concerning the dependence of the constants for $\Omega$ being convex, see~\cite[Eq.~(3.1.2.2), Eq.~(3.1.2.7)]{Grisvard}.
\end{proof}

By~\eqref{Eq: Representation Helmholtz} it should be clear that $\IP$ does not map $\H^1_0 (\Omega ; \IC^d)$ into $\H^1_0 (\Omega ; \IC^d)$, i.e., that it does not preserve zero boundary values (if the boundary is merely Lipschitz, then also the differentiability is not preserved). A formal proof on bounded $\C^4$-domains is given in the next lemma. The proof uses so-called Fermi coordinates. These coordinates are introduced following the exposition in~\cite[Sec.~2.3]{Dziuk_Elliott}. \par
Let $\delta (x)$ denote the oriented distance function, i.e., 
\begin{align*}
 \delta (x) := \begin{cases} \dist(x , \partial \Omega), & x \in \Omega \\ - \dist(x , \partial \Omega), & x \in \overline{\Omega}^c. \end{cases}
\end{align*}
If $\Omega$ has a $\C^k$-boundary with $k \in \IN$ with $k \geq 2$, one verifies by virtue of uniform inner and outer ball properties of $\partial \Omega$, that there exists $\eps > 0$ such that with
\begin{align*}
 U_{\eps} := \{ x \in \IR^d : \lvert \delta (x) \rvert < \eps \}
\end{align*}
one has $\delta \in \C^k (U_{\eps})$ and that for every point $x \in U_{\eps}$ there exists a unique point $a (x) \in \partial \Omega$ such that
\begin{align*}
 x = a (x) + \delta (x) n (a(x)),
\end{align*}
where $n$ denotes the exterior unit normal to $\partial \Omega$. Thus, in the neighborhood $U_{\eps}$, every point $x$ can be represented uniquely by the new coordinates $a(x)$ and $\delta(x)$. To proceed, we introduce some further geometric notions. A function $u \in \L^1 (\partial \Omega)$ is weakly differentiable if its composition with the coordinate chart is weakly differentiable in $\IR^{d - 1}$. For such functions one can define the tangential gradient $\nabla_{\T} u$ of $u$ (see, e.g., the exposition in~\cite[Sec.~1.3]{Tolksdorf_Dissertation}). The tangential gradient has the property, that for functions $u$ that are smooth enough and defined in a neighborhood of $\partial \Omega$ one has
\begin{align*}
 \nabla u (x) = \nabla_{\T} u (x) + (n(x) \cdot \nabla u(x)) n(x) \qquad (x \in \partial \Omega).
\end{align*}
Notice that $\nabla_{\T} u (x)$ for $x \in \partial \Omega$ always lies in the tangent space at $x$ if $\partial \Omega$ is smooth enough. Similarly, we define a vector $v \in \IC^d$ and $x \in \partial \Omega$ its tangential component $v_{\T}$ to satisfy
\begin{align}
\label{Eq: Tangential component}
 v_{\T} = v - (n(x) \cdot v) n(x).
\end{align}
This will be used in Section~\ref{Sec: Estimates on convex domains}. \par
Given a function $g \in \C^1 (\partial \Omega)$, then $g$ can be extended to a function $G$ on $U_{\eps}$ by setting
\begin{align*}
 G (x) := g(a (x)) \qquad (x \in U_{\eps})
\end{align*}
and~\cite[Eq.~(2.14)]{Dziuk_Elliott} shows that
\begin{align}
\label{Eq: Gradient of extended function}
 \nabla G (x) = (1 - \delta (x) \cH (x)) \nabla_{\T} g (a (x)).
\end{align}
Here, $\cH$ denotes the extended Weingarten map, which is given by
\begin{align*}
 \cH (x) = (\cH_{i , j} (x))_{i , j = 1}^d := \e_i \cdot \nabla_{\T} n_j (x)
\end{align*}
and which is $\C^2$-regular if $\Omega$ has a $\C^4$-boundary. Notice that $\e_i$ denotes the $i$th standard basis vector of $\IR^d$.

\begin{lemma}
\label{Lem: Trace of Helmholtz}
Let $\Omega$ be a bounded domain with $\C^4$-boundary. Then there exists $u \in \H^1_0 (\Omega ; \IC^d)$ such that the trace of $\IP u$ to $\partial \Omega$ is not zero.
\end{lemma}

\begin{proof}
Notice that $\divergence : \H^1_0 (\Omega ; \IC^d) \to \L^2_0 (\Omega)$ is surjective by~\cite[Thm.~III.3.1]{Galdi}. Consequently, the range of $\Delta_N^{-1} \divergence$ is given by
\begin{align*}
 \Rg(\Delta_N^{-1} \divergence) = \Rg(\Delta_N^{-1}|_{\L^2_0 (\Omega)}) = \{ u \in \H^2 (\Omega) : n \cdot \nabla u = 0 \text{ on } \partial \Omega \}.
\end{align*}
Now, if there exists $u \in \H^2 (\Omega)$ with $n \cdot \nabla u = 0$ on $\partial \Omega$ and $\nabla u \neq 0$ on $\partial \Omega$, set $f := B \Delta_N u$ (with $B$ being the Bogovski\u{\i} operator) which lies in $\H^1_0 (\Omega ; \IC^d)$ and by virtue of~\eqref{Eq: Representation Helmholtz} $\IP f$ satisfies $\tr(\IP f) \neq 0$. This would conclude the proof. To construct such a function $u$, let $g : \partial \Omega \to \IC$ be a non-constant and smooth function and let $\eta \in \C_c^{\infty} (U_{\eps})$ with $\eta = 1$ in a neighborhood of $\partial \Omega$. Extend $g$ to $U_{\eps}$ by setting
\begin{align*}
 u (x) := g(a (x)) \eta (x).
\end{align*}
The gradient of $u$ is calculated by virtue of~\eqref{Eq: Gradient of extended function}, leading to
\begin{align*}
 \nabla u (x) = \eta (x) (1 - \delta (x) \cH(x)) \nabla_{\T} g (a (x)) + g(a (x)) \nabla \eta (x).
\end{align*}
Clearly, the normal derivative of $u$ vanishes on $\partial \Omega$ while its full gradient is non-zero since $g$ is non-constant on $\partial \Omega$. Since $\eta$, $\delta$, and $\cH$ are at least $\C^2$-regular, it follows that $u \in \Rg(\Delta_N^{-1} \divergence)$ with $\nabla u$ not being constantly zero on $\partial \Omega$.
\end{proof}

\subsection{Resolvent estimates}

We continue by discussing some classical resolvent estimates in $\L^2_{\sigma} (\Omega)$ and $\H^{-1}_{\sigma} (\Omega)$ for the operators $A$ and $\cA$, and for $B_{\mu}$ and $\cB_{\mu}$ on $\cL^2_{\sigma} (\Omega)$ and $\cH^{- 1}_{0 , \sigma} (\Omega)$. In the case of Neumann-type boundary conditions, we restrict ourselves to parameters that satisfy $\mu \in (-1 , 1]$. This is due to the fact, that this ensures a certain coercivity of the sesquilinear form $\fb_{\mu}$. Indeed, by~\cite[Prop.~4.1.2]{Mitrea_Wright}, for $\lvert \mu \rvert < 1$ there exists $\kappa_{\mu} > 0$ such that
\begin{align}
\label{Eq: Ellipticity Neumann}
 \Re(a_{j k}^{\alpha \beta} \xi_{\beta k} \overline{\xi_{\alpha j}}) \geq \kappa_{\mu} \lvert \xi \rvert^2 \qquad (\xi \in \IC^{d \times d}).
\end{align}
Moreover, the same result ensures that in the case $\mu = 1$ there exists $\kappa_1 > 0$ such that
\begin{align}
\label{Eq: Symmetric gradient}
 \Re(a_{j k}^{\alpha \beta} \xi_{\beta k} \overline{\xi_{\alpha j}}) \geq \kappa_1 \lvert \xi + \xi^{\top} \rvert^2 \qquad (\xi \in \IC^{d \times d}).
\end{align}
To proceed, define for some angle $\theta \in [0 , \pi)$ the sector in the complex plane
\begin{align*}
 \S_{\theta} := \begin{cases} (0 , \infty), &\text{if } \theta = 0 \\ \{ z \in \IC \setminus \{ 0 \} : \lvert \arg(z) \rvert  < \theta \}, &\text{if } \theta \in (0 , \pi). \end{cases}
\end{align*}
Notice, that by elementary trigonometry, one can prove that there exists $C_{\theta} > 0$ depending only on $\theta$, such that
\begin{align}
\label{Eq: Inverse triangle inequality}
 \lvert z \rvert + \alpha \leq C_{\theta} \lvert z + \alpha \rvert \qquad (z \in \overline{\S_{\theta}} , \alpha \geq 0).
\end{align}

\begin{proposition}
\label{Prop: Resolvent}
Let $\Omega \subset \IR^d$, $d \geq 2$, be a bounded Lipschitz domain and $\theta \in [0 , \pi)$.
\begin{enumerate}
 \item \label{Prop: Dirichlet} It holds $\{ 0 \} \cup \S_{\theta} \subset \rho(- A) \cap \rho(- \cA)$. Moreover, there exists $C > 0$ depending only on $d$ and $\theta$ such that for all $f \in \L^2_{\sigma} (\Omega)$, all $F \in \H^{-1}_{\sigma} (\Omega)$, and all $\lambda \in \S_{\theta}$ it holds
\begin{align*}
 \lvert \lambda \rvert \| (\lambda + A)^{-1} f \|_{\L^2_{\sigma} (\Omega)} + \lvert \lambda \rvert^{1 / 2} \| \nabla (\lambda + A)^{-1} f \|_{\L^2 (\Omega ; \IC^{d^2})} \leq C \| f \|_{\L^2_{\sigma} (\Omega)}
\end{align*}
and
\begin{align*}
 \lvert \lambda \rvert \| (\lambda + \cA)^{-1} F \|_{\H^{-1}_{\sigma} (\Omega)} &+ \lvert \lambda \rvert^{1 / 2} \| (\lambda + \cA)^{-1} F \|_{\L^2_{\sigma} (\Omega)} \\
 &+ \| \nabla (\lambda + \cA)^{-1} F \|_{\L^2 (\Omega ; \IC^{d^2})} \leq C \| F \|_{\H^{-1}_{\sigma} (\Omega)}.
\end{align*}
 \item \label{Prop: Neumann} For all $\mu \in (-1 , 1)$ it holds $\S_{\theta} \subset \rho(- B_{\mu}) \cap \rho(- \cB_{\mu})$. Moreover, there exists $C > 0$ depending only on $d$, $\theta$, and $\mu$ such that for all $f \in \cL^2_{\sigma} (\Omega)$, all $F \in \cH^{-1}_{0 , \sigma} (\Omega)$, and all $\lambda \in \S_{\theta}$ it holds
\begin{align*}
 \lvert \lambda \rvert \| (\lambda + B_{\mu})^{-1} f \|_{\cL^2_{\sigma} (\Omega)} + \lvert \lambda \rvert^{1 / 2} \| \nabla (\lambda + B_{\mu})^{-1} f \|_{\L^2 (\Omega ; \IC^{d^2})} \leq C \| f \|_{\cL^2_{\sigma} (\Omega)}
\end{align*}
and
\begin{align*}
 \lvert \lambda \rvert \| (\lambda + \cB_{\mu})^{-1} F \|_{\cH^{-1}_{0 , \sigma} (\Omega)} &+ \lvert \lambda \rvert^{1 / 2} \| (\lambda + \cB_{\mu})^{-1} F \|_{\cL^2_{\sigma} (\Omega)} \\
 &+ \| \nabla (\lambda + \cB_{\mu})^{-1} F \|_{\L^2 (\Omega ; \IC^{d^2})} \leq C \| F \|_{\cH^{-1}_{0 , \sigma} (\Omega)}.
\end{align*}
 \item \label{Prop: Free} For $\mu = 1$ it holds $\S_{\theta} \subset \rho(- B_1) \cap \rho(- \cB_1)$. Moreover, there exists $C > 0$ depending only on $d$, $\theta$, the Lipschitz character of $\Omega$, and $\diam (\Omega)$ such that for all $f \in \cL^2_{\sigma} (\Omega)$, all $F \in \cH^{-1}_{0 , \sigma} (\Omega)$, and all $\lambda \in \S_{\theta}$ it holds
\begin{align*}
 \lvert \lambda \rvert \| (\lambda + B_1)^{-1} f \|_{\cL^2_{\sigma} (\Omega)} + \lvert \lambda \rvert^{1 / 2} \| (D + D^{\top}) (\lambda + B_1)^{-1} f \|_{\L^2 (\Omega ; \IC^{d \times d})} \leq C \| f \|_{\cL^2_{\sigma} (\Omega)}
\end{align*}
and
\begin{align*}
 \lvert \lambda \rvert \| (\lambda + \cB_1)^{-1} F \|_{\cH^{-1}_{0 , \sigma} (\Omega)} &+ \lvert \lambda \rvert^{1 / 2} \| (\lambda + \cB_1)^{-1} F \|_{\cL^2_{\sigma} (\Omega)} \\
 &+ \| (D + D^{\top}) (\lambda + \cB_1)^{-1} F \|_{\L^2 (\Omega ; \IC^{d \times d})} \leq C \| F \|_{\cH^{-1}_{0 , \sigma} (\Omega)}.
\end{align*}
Recall that $D u := (\partial_i u_j)_{i , j = 1}^d$ denotes the Jacobian matrix of some function $u$.
\end{enumerate}
\end{proposition}

\begin{proof}
The statements on the resolvent set follow by the Lemma of Lax--Milgram. Indeed, for $\lambda \in \S_{\theta}$ one defines new sesquilinear forms
\begin{align*}
 \fa_{\lambda} (u , v) := \lambda \int_{\Omega} u \cdot \overline{v} \; \d x + \fa (u , v)
\end{align*}
and analogously one defines $\fb_{\mu , \lambda}$. By~\eqref{Eq: Inverse triangle inequality}, $\fa_{\lambda}$ becomes coercive. If $\lvert \mu \rvert < 1$, then~\eqref{Eq: Inverse triangle inequality} together with~\eqref{Eq: Ellipticity Neumann} implies the coercivity of $\fb_{\mu , \lambda}$. Finally, in the case $\mu = 1$, one uses a Korn-type inequality proved in~\cite[Prop.~11.4.2]{Mitrea_Wright}, to define an equivalent norm on $\H^1 (\Omega ; \IC^d)$, which is given by
\begin{align*}
 \| u \| := \| u \|_{\L^2 (\Omega ; \IC^d)} + \| Du + [Du]^{\top} \|_{\L^2 (\Omega ; \IC^{d \times d})}.
\end{align*}
Notice that the constants implicit in the equivalence of the norms depend on the Lipschitz character of $\Omega$ and its diameter. In this case, the coercivity of $\fb_{1, \lambda}$ follows from~\eqref{Eq: Inverse triangle inequality} and~\eqref{Eq: Symmetric gradient}. The first inequalities of~\eqref{Prop: Dirichlet},~\eqref{Prop: Neumann}, and~\eqref{Prop: Free} follow as usual by testing the resolvent equation by the solution $u$ and by employing~\eqref{Eq: Inverse triangle inequality}, see, e.g.,~\cite[Prop.~5.2.5]{Tolksdorf_Dissertation}. Also the estimates on the second and third terms in the second inequalities of~\eqref{Prop: Dirichlet},~\eqref{Prop: Neumann}, and~\eqref{Prop: Free} follow by testing the resolvent equations by the solution $u$. Finally, the $\H^{-1}_{\sigma} (\Omega)$-estimate on $\lvert \lambda \rvert (\lambda + \cA)^{-1} F =: \lvert \lambda \rvert u$ follows by virtue of the resolvent equation and the estimates that were already established before by
\begin{align*}
 \sup_{\substack{v \in \H^1_{0 , \sigma} (\Omega) \\ \| v \|_{\H^1_0 \leq 1}}} \Big\lvert \lambda \int_{\Omega} u \cdot \overline{v} \; \d x \Big\rvert = \sup_{\substack{v \in \H^1_{0 , \sigma} (\Omega) \\ \| v \|_{\H^1_0 \leq 1}}} \Big\lvert \langle F , v \rangle_{\H^{-1}_{\sigma} , \H^1_{0 , \sigma}} - \int_{\Omega} \nabla u \cdot \overline{\nabla v} \; \d x \Big\rvert \leq C \| F \|_{\H^{-1}_{\sigma} (\Omega)}.
\end{align*}
In~\eqref{Prop: Neumann} and~\eqref{Prop: Free} the remaining estimates follow analogously.
\end{proof}

\begin{remark}
\label{Rem: Higher regularity}
Notice that if $\Omega$ has a $\C^{\infty}$-boundary and if $f \in \C_{\sigma}^{\infty} (\overline{\Omega})$ one shows by the method of difference quotients (by using~\eqref{Eq: Ellipticity Neumann}) and localization that for $\mu \in (-1 , 1)$ it holds $u \in \C^{\infty} (\overline{\Omega} ; \IC^d)$ and $\phi \in \C^{\infty} (\overline{\Omega})$.
\end{remark}

\subsection{Analytic semigroups and fractional powers}

It is well-known, see, e.g.,~\cite{Engel_Nagel}, that the resolvent estimates presented in Proposition~\ref{Prop: Resolvent}~\eqref{Prop: Dirichlet} imply that $-A$ and $- \cA$ generate bounded analytic semigroups $(\e^{- t A})_{t \geq 0}$ and $(\e^{- t \cA})_{t \geq 0}$ on $\L^2_{\sigma} (\Omega)$ or $\H^{-1}_{\sigma} (\Omega)$, respectively. By the real characterization of analytic semigroups~\cite[Thm.~II.4.6]{Engel_Nagel} it further holds
\begin{align}
\label{Eq: Real characterization}
 \sup_{t > 0} \| t A \e^{- t A} \|_{\Lop(\L^2_{\sigma} (\Omega))} < \infty \qquad \text{and} \qquad \sup_{t > 0} \| t \cA \e^{- t \cA} \|_{\Lop(\H^{-1}_{\sigma} (\Omega))} < \infty.
\end{align}
\indent For $\vartheta \in (0 , \pi / 2)$ and $r > 0$ let $\gamma_{\vartheta , r}$ denote the path that parameterizes $B(0 , r) \cup \S_{\vartheta}$ in the counterclockwise direction. For $t > 0$ the operators $\e^{- t A}$ and $\e^{- t \cA}$ are then given by the contour integrals
\begin{align}
\label{Eq: Cauchy integrals}
 \e^{- t A} = \frac{1}{2 \pi \ii} \int_{\gamma_{\vartheta , 1 / t}} \e^{- t \lambda} (\lambda - A)^{-1} \; \d \lambda \qquad \text{and} \qquad \e^{- t \cA} = \frac{1}{2 \pi \ii} \int_{\gamma_{\vartheta , 1 / t}} \e^{- t \lambda} (\lambda - \cA)^{-1} \; \d \lambda.
\end{align}
These integrals converge in $\Lop(\L^2_{\sigma} (\Omega))$ and $\Lop(\H^{-1}_{\sigma} (\Omega))$, respectively, due to Proposition~\ref{Prop: Resolvent}~\eqref{Prop: Dirichlet}. Using this representation also gradient estimates of the resolvent, cf.\@ Proposition~\ref{Prop: Resolvent}~\eqref{Prop: Dirichlet}, translate into gradient estimates of the corresponding semigroups, i.e., following for example~\cite[Prop.~3.7]{Tolksdorf_Watanabe} there exists $C > 0$ depending only on $d$ and $\theta$ such that for all $t > 0$
\begin{align}
\label{Eq: Gradient estimates semigroup}
 t^{1 / 2} \| \nabla \e^{- t A} \|_{\Lop(\L^2_{\sigma} (\Omega) , \L^2 (\Omega ; \IC^{d^2}))} + t^{1 / 2} \| \e^{- t \cA} \|_{\Lop(\H^{-1}_{\sigma} (\Omega) , \L^2_{\sigma} (\Omega))} \leq C.
\end{align}

\indent Besides analytic semigroups one can define for $\alpha > 0$ the fractional powers $A^{\alpha}$ and $\cA^{\alpha}$. There is a counterpart of~\eqref{Eq: Real characterization} for fractional powers reading for $0 < \alpha < 1$ as
\begin{align}
\label{Eq: Fractional parabolic smoothing}
 \sup_{t > 0} \| (t A)^{\alpha} \e^{- t A} \|_{\Lop(\L^2_{\sigma} (\Omega))} < \infty \qquad \text{and} \qquad \sup_{t > 0} \| (t \cA)^{\alpha} \e^{- t \cA} \|_{\Lop(\H^{-1}_{\sigma} (\Omega))} < \infty,
\end{align}
which follows for example by~\eqref{Eq: Real characterization} combined with the moment inequality~\cite[Prop.~6.6.4]{Haase}. Since the sesquilinear form that is associated to $A$ is symmetric~\cite[Thm.~VI.2.23]{Kato} yields that
\begin{align}
\label{Eq: Kato square root}
 \dom(A^{1 / 2}) = \H^1_{0 , \sigma} (\Omega).
\end{align}
Moreover,~\cite[Thm.~5.1]{Mitrea_Monniaux} implies that
\begin{align}
\label{Eq: Fractional power domains on L2}
 \dom(A^{\alpha}) = \H^{2 \alpha}_{0 , \sigma} (\Omega) \quad \text{if} \quad \tfrac{1}{4} < \alpha < \tfrac{1}{2} \qquad \text{and} \qquad \dom(A^{\alpha}) = \H^{2 \alpha}_{\sigma} (\Omega) \quad \text{if} \quad 0 < \alpha < \tfrac{1}{4}.
\end{align}
To determine the fractional power domains of $\cA^{\alpha}$ for $1 / 4 < \alpha \leq 1 / 2$ one can argue as follows: As $A$ is bijective, it follows that $A^{1 / 2}$ is an isomorphism from $\H^1_{0 , \sigma} (\Omega)$ onto $\L^2_{\sigma} (\Omega)$ and by duality, $(A^{1 / 2})^*$ is an isomorphism from $\L^2_{\sigma} (\Omega)$ onto $\H^{-1}_{\sigma} (\Omega)$. A quick calculation, cf.\@ \cite[Lem.~5.1]{Choudhury_Hussein_Tolksdorf} for the case $d = 3$, reveals that
\begin{align*}
 \cA = (A^{1 / 2})^* \circ A \circ (A^{- 1 / 2})^*.
\end{align*}
In other words, $A$ and $\cA$ are similar with respect to the isomorphism $(A^{1 / 2})^*$. Now, $\dom(\cA^{\alpha})$ is given by definition by $\Rg(\cA^{- \alpha})$. The similarity implies that
\begin{align*}
 \cA^{- \alpha} = (A^{1 / 2})^* \circ A^{- \alpha} \circ (A^{- 1 / 2})^* \qquad (\tfrac{1}{4} < \alpha \leq \tfrac{1}{2}).
\end{align*}
Thus, since $(A^{- 1 / 2})^*$ is an isomorphism from $\H^{-1}_{\sigma} (\Omega)$ onto $\L^2_{\sigma} (\Omega)$,~\eqref{Eq: Kato square root} and~\eqref{Eq: Fractional power domains on L2} imply that
\begin{align*}
 \Rg(\cA^{- \alpha}) = (A^{1 / 2})^* \H^{2 \alpha}_{0 , \sigma} (\Omega).
\end{align*}
Thus, $v \in \dom(\cA^{\alpha})$ if and only if there exists $u \in \H^{2 \alpha}_{0 , \sigma} (\Omega)$ such that $v = (A^{1 / 2})^* u$. To characterize these functionals in terms of Sobolev regularity, notice that by the self-adjointness of $A$ on $\L^2_{\sigma} (\Omega)$, $v$ is the functional
\begin{align*}
 \langle v , w \rangle_{\H^{-1}_{\sigma} , \H^1_{0 , \sigma}} = \langle u , A^{1 / 2} w \rangle_{\L^2_{\sigma} , \L^2_{\sigma}} = \langle A^{\alpha} u , A^{1 / 2 - \alpha} w \rangle_{\L^2_{\sigma} , \L^2_{\sigma}} \qquad (w \in \H^1_{0 , \sigma} (\Omega)).
\end{align*}
By~\eqref{Eq: Fractional power domains on L2} it now follows that
\begin{align*}
 \lvert \langle v , w \rangle_{\H^{-1}_{\sigma} , \H^1_{0 , \sigma}} \rvert \leq C \| A^{\alpha} u \|_{\L^2_{\sigma} (\Omega)} \| w \|_{\H^{1 - 2 \alpha}_{\sigma} (\Omega)}.
\end{align*}
This implies that $v \in \H^{2 \alpha - 1}_{\sigma} (\Omega)$. Finally, if $v \in \H^{2 \alpha - 1}_{\sigma} (\Omega)$, define $u := (A^{- 1 / 2})^* v$ and conclude that $u \in \H^{2 \alpha}_{0 , \sigma} (\Omega)$ by an interpolation argument (the case $\alpha = 1 / 2$ is clear so that we assume $1 / 4 < \alpha < 1 / 2$). Indeed, $(A^{- 1 / 2})^*$ is bounded from $\H^{-1}_{\sigma} (\Omega)$ onto $\L^2_{\sigma} (\Omega)$ and its restriction to $\L^2_{\sigma} (\Omega)$ (this restriction is the operator $A^{- 1 / 2}$) is bounded from $\L^2_{\sigma} (\Omega)$ onto $\H^1_{0 , \sigma} (\Omega)$. The complex interpolation space $[\H^{-1}_{\sigma} (\Omega) , \L^2_{\sigma} (\Omega)]_{2 \alpha}$ is calculated by the duality rule (notice that we identify $\L^2_{\sigma} (\Omega) \simeq \L^2_{\sigma} (\Omega)^*$) and~\cite[Thm.~2.12]{Mitrea_Monniaux} by
\begin{align*}
 \big[ \H^{-1}_{\sigma} (\Omega) , \L^2_{\sigma} (\Omega) \big]_{2 \alpha} = \big[ \L^2_{\sigma} (\Omega) , \H^1_{0 , \sigma} (\Omega) \big]_{1 - 2 \alpha}^* = \H^{2 \alpha - 1}_{\sigma} (\Omega).
\end{align*}
Moreover, employing~\cite[Thm.~2.12]{Mitrea_Monniaux} again yields
\begin{align*}
 \big[ \L^2_{\sigma} (\Omega) , \H^1_{0 , \sigma} (\Omega) \big]_{2 \alpha} = \H^{2 \alpha}_{0 , \sigma} (\Omega).
\end{align*}
It follows that $(A^{- 1 / 2})^*$ is bounded from $\H^{2 \alpha - 1}_{\sigma} (\Omega)$ onto $\H^{2 \alpha}_{0 , \sigma} (\Omega)$ and thus that $u \in \H^{2 \alpha}_{0 , \sigma} (\Omega)$. As a consequence, this reveals
\begin{align}
\label{Eq: Fractional power domains on H-1}
 \dom(\cA^{\alpha}) = \H^{2 \alpha - 1}_{\sigma} (\Omega) \quad \text{if} \quad \tfrac{1}{4} < \alpha \leq \tfrac{1}{2},
\end{align}
where $\H^0_{\sigma} (\Omega)$ is identified with $\L^2_{\sigma} (\Omega)$.

\section{On uniform pressure estimates}
\label{On uniform pressure estimates}

\noindent Having the theory on the Stokes operator from Section~\ref{Sec: The Stokes operator} at hand, one associates a pressure function $\phi$ to a solution $u$ as follows. Assume that $F \in \H^{-1} (\Omega ; \IC^d) \subset \H^{-1}_{\sigma} (\Omega)$ and let $\lambda \in \S_{\theta}$ for some $\theta \in [0 , \pi)$. By Proposition~\ref{Prop: Resolvent} there exists a unique $u := (\lambda + \cA)^{-1} F$ such that
\begin{align*}
 \langle G , v \rangle_{\H^{-1} , \H^1_0} := \langle F , v \rangle_{\H^{-1} , \H^1_0}  - \lambda \int_{\Omega} u \cdot \overline{v} \; \d x - \int_{\Omega} \nabla u \cdot \overline{\nabla v} \; \d x = 0 \quad (v \in \H^1_{0 , \sigma} (\Omega)).
\end{align*}
Then $G$ is a functional in $\H^{-1} (\Omega)$ which vanishes on $\C_{c , \sigma}^{\infty} (\Omega)$ so that $G$ must in fact be a gradient. Indeed, by~\cite[Lem.~II.2.2.2]{Sohr} there exists $\phi \in \L^2 (\Omega)$ with mean value zero such that
\begin{align}
\label{Eq: Variational Dirichlet}
 \lambda \int_{\Omega} u \cdot \overline{v} \; \d x + \int_{\Omega} \nabla u \cdot \overline{\nabla v} \; \d x - \int_{\Omega} \phi \, \overline{\divergence (v)} \; \d x = \langle F , v \rangle_{\H^{-1} , \H^1_0} \quad (v \in \H^1_0 (\Omega ; \IC^d)).
\end{align}

\indent In the case of Neumann-type boundary conditions, one proceeds similarly. As above, for $F \in \H^{-1}_0 (\Omega ; \IC^d)$ and $u := (\lambda + \cB_{\mu})^{-1} F$ one finds $\vartheta \in \L^2 (\Omega)$ such that
\begin{align*}
 \lambda \int_{\Omega} u \cdot \overline{v} \; \d x + \int_{\Omega} a_{j k}^{\alpha \beta} (\mu) \partial_k u_{\beta} \overline{\partial_j v_{\alpha}} \; \d x - \int_{\Omega} \vartheta \, \overline{\divergence (v)} \; \d x = \langle F , v \rangle_{\H^{-1}_0 , \H^1} \quad (v \in \H^1_0 (\Omega ; \IC^d)).
\end{align*}
However, it would be desirable to lift this identity to hold for all $v \in \H^1 (\Omega ; \IC^d)$. As one can expect, by the boundary condition given in~\eqref{Neu}, the pressure function is unique (in the case of no-slip boundary conditions the pressure is unique up to an additive constant). Thus, one must find a constant $c \in \IC$ such that the identity above with $\vartheta$ replaced by $\phi := \vartheta + c$ holds for all $v \in \H^1 (\Omega ; \IC^d)$. In fact, a way of how to construct this constant $c$ is described in the proof of~\cite[Thm.~6.8]{Mitrea_Monniaux_Wright}. Thus, we record that there exists $\phi \in \L^2 (\Omega)$ such that
\begin{align}
\label{Eq: Variational Neumann}
 \lambda \int_{\Omega} u \cdot \overline{v} \; \d x + \int_{\Omega} a_{j k}^{\alpha \beta} (\mu) \partial_k u_{\beta} \overline{\partial_j v_{\alpha}} \; \d x - \int \phi \, \overline{\divergence (v)} \; \d x = \langle F , v \rangle_{\H^{-1}_0 , \H^1} \quad (v \in \H^1 (\Omega ; \IC^d)).
\end{align}
Finally, notice that $\L^2 (\Omega ; \IC^d)$ naturally embeds into $\H^{-1} (\Omega ; \IC^d)$ and $\H^{-1}_0 (\Omega ; \IC^d)$. If $F_1$ denotes the functional in $\H^{-1} (\Omega ; \IC^d)$ identified with $f \in \L^2 (\Omega ; \IC^d)$ and if $F_2$ denotes its identification with an element in $\H^{-1}_0 (\Omega ; \IC^d)$, we find by virtue of~\eqref{Eq: Orthogonal range Helmholtz P} and~\eqref{Eq: Orthogonal range Helmholtz Q}
\begin{align*}
 \langle F_1 , v \rangle_{\H^{-1} , \H^1_0} = \int_{\Omega} f \cdot \overline{v} \; \d x = \int_{\Omega} \IP f \cdot \overline{v} \; \d x - \int_{\Omega} g_1 \, \overline{\divergence(v)} \; \d x \quad (v \in \H^1_0 (\Omega ; \IC^d))
\end{align*}
and
\begin{align}
\label{Eq: Neumann general f}
 \langle F_2 , v \rangle_{\H^{-1}_0 , \H^1} = \int_{\Omega} f \cdot \overline{v} \; \d x = \int_{\Omega} \IQ f \cdot \overline{v} \; \d x - \int_{\Omega} g_2 \, \overline{\divergence(v)} \; \d x \quad (v \in \H^1 (\Omega ; \IC^d))
\end{align}
for functions $g_1 \in \H^1 (\Omega)$ and $g_2 \in \H^1_0 (\Omega)$. Absorbing the functions $g_1$ and $g_2$, respectively, into the pressure functions, one finds the identities (by abusing the notation we write $f$ instead of $F_1$ and $F_2$)
\begin{align*}
 (\lambda + \cA)^{-1} f = (\lambda + A)^{-1} \IP f \qquad \text{and} \qquad (\lambda + \cB_{\mu})^{-1} f = (\lambda + B_{\mu})^{-1} \IQ f.
\end{align*}
Consequently, solving the Stokes resolvent problem with a right-hand side $f \in \L^2 (\Omega ; \IC^d)$ is the same as solving the Stokes resolvent problem with right-hand side $\IP f$ (or $\IQ f$, respectively) and one only changes the pressure by the gradient part inherent in $f$. \par
Given $F \in \H^{-1}_{\sigma} (\Omega)$, we say that $\phi$ is the associated pressure to~\eqref{Res} subject to~\eqref{Dir} with right-hand side $F$ if $\phi \in \L^2_0 (\Omega)$ and if $u := (\lambda + \cA)^{-1} F$ and $\phi$ satisfy~\eqref{Eq: Variational Dirichlet}. Analogously, we proceed for Neumann-type boundary conditions but with the relation~\eqref{Eq: Variational Neumann} and without the requirement on the mean value. \par
For Neumann-type boundary conditions we have the following estimates on the pressure.

\begin{proposition}
\label{Prop: Pressure estimate Neumann}
Let $\Omega$ be a bounded Lipschitz domain such that $\Delta_D^{-1} : \L^2 (\Omega) \to \H^2 (\Omega)$ is bounded. Let $\theta \in (0 , \pi]$ and $\mu \in (-1 , 1]$. There exists a constant $C > 0$ such that for all $f \in \cL^2_{\sigma} (\Omega)$ and $\lambda \in \S_{\theta}$ the associated pressure $\phi \in \L^2 (\Omega)$ to~\eqref{Res} subject to~\eqref{Neu} with right-hand side $f$ satisfies
\begin{align}
\label{Eq: Neumann pressure L2}
 \lvert \lambda \rvert^{1 / 2} \| \phi \|_{\L^2 (\Omega)} \leq C \| f \|_{\cL^2_{\sigma} (\Omega)}.
\end{align}
Furthermore, there exists a constant $C > 0$ such that for all $F \in \H^{-1}_0 (\Omega ; \IC^d)$ and $\lambda \in \S_{\theta}$ the associated pressure $\phi \in \L^2 (\Omega)$ to~\eqref{Res} subject to~\eqref{Neu} and right-hand side $F$ satisfies
\begin{align}
\label{Eq: Neumann pressure negative}
 \| \phi \|_{\L^2 (\Omega)} \leq C \| F \|_{\H^{-1}_0 (\Omega ; \IC^d)}.
\end{align}
If $\Omega$ is bounded and convex and $\lvert \mu \rvert < 1$, the constants $C > 0$ depend at most on $d$, $\theta$, and $\mu$. If $\mu = 1$, the constants further depend on the Lipschitz character of $\Omega$ and its diameter.
\end{proposition}

\begin{proof}
To prove~\eqref{Eq: Neumann pressure L2} consider the test function $v := \nabla (- \Delta_D)^{-1} \phi$, which lies in the orthogonal space to $\cL^2_{\sigma} (\Omega)$ by~\eqref{Eq: Orthogonal range Helmholtz Q}. Thus, by~\eqref{Eq: Variational Neumann} and the boundedness property of the Laplacian, we infer in the case $\lvert \mu \rvert < 1$
\begin{align*}
 \int_{\Omega} \lvert \phi \rvert^2 \; \d x = \int_{\Omega} \phi \, \overline{\divergence(v)} \; \d x = \int_{\Omega} a_{j k}^{\alpha \beta} (\mu) \partial_k u_{\beta} \overline{\partial_j v_{\alpha}} \; \d x \leq C \| \nabla u \|_{\L^2 (\Omega ; \IC^{d^2})} \| \phi \|_{\L^2 (\Omega)}.
\end{align*}
If $\mu = 1$ one obtains the same but with $\| \nabla u \|_{\L^2 (\Omega ; \IC^{d^2})}$ replaced by $\| Du + [Du]^{\top} \|_{\L^2 (\Omega ; \IC^{d \times d})}$. The estimate is concluded by dividing by $\| \phi \|_{\L^2 (\Omega)}$ and by employing Proposition~\ref{Prop: Resolvent}~\eqref{Prop: Neumann} or~\eqref{Prop: Free}. \par
To establish~\eqref{Eq: Neumann pressure negative} use the same test function. The only difference to the calculation above is the behavior in $\lambda$ of the terms $\| \nabla u \|_{\L^2 (\Omega ; \IC^{d^2})}$ and $\| Du + [Du]^{\top} \|_{\L^2 (\Omega ; \IC^{d \times d})}$ and the fact, that $\langle F , v \rangle_{\H^{-1}_0 , \H^1}$ does not vanish. However, it is estimated by the boundedness assumption of the Laplacian as
\begin{align*}
 \lvert \langle F , v \rangle_{\H^{-1}_0 , \H^1} \rvert \leq \| F \|_{\H^{-1}_0 (\Omega ; \IC^d)} \| v \|_{\H^1 (\Omega ; \IC^d)} \leq C \| F \|_{\H^{-1}_0 (\Omega ; \IC^d)} \| \phi \|_{\L^2 (\Omega)}
\end{align*}
and the term $\| \phi \|_{\L^2 (\Omega)}$ is handled again by division. \par
Concerning the dependence of $C$ on the quantities $d$, $\theta$, and $\mu$, notice that the only critical quantity is the operator norm of $\nabla^2 \Delta_D^{-1}$ on $\L^2$. That this is bounded by a constant depending only on $d$ follows by~\cite[Eq.~(3.1.2.2)]{Grisvard}.
\end{proof}

\begin{remark}
\label{Rem: No pressure decay}
\begin{enumerate}
 \item Notice that~\eqref{Eq: Neumann pressure L2} cannot hold if $f \in \L^2 (\Omega ; \IC^d) \setminus \cL^2_{\sigma} (\Omega)$ since in this case the pressure part $g_2$ defined in~\eqref{Eq: Neumann general f} does not vanish. This gives a contribution that does not even depend on $\lambda$.
 \item Since $\nabla u$ and the pressure are connected via the imposed boundary condition in~\eqref{Neu}, it seems natural that the pressure and $\nabla u$ both have the same behavior in the resolvent parameter $\lambda$.
\end{enumerate}
\end{remark}

To find out how the corresponding estimates for the Stokes resolvent problem subject to no-slip boundary conditions look like will occupy the rest of this section. Notice that the following proposition was proven (also in the $\L^p$-situation) on bounded and smooth domains in~\cite[Lem.~13]{Noll_Saal} and on bounded Lipschitz domains in~\cite[Prop.~4.3]{Tolksdorf_Watanabe}.

\begin{proposition}
\label{Prop: Pressure estimate Dirichlet}
Let $\Omega$ be a bounded Lipschitz domain and $\theta \in (0 , \pi]$. For all $0 \leq \alpha < 1 / 4$, there exists a constant $C > 0$ such that for all $f \in \L^2_{\sigma} (\Omega)$ and $\lambda \in \S_{\theta}$ the associated pressure $\phi \in \L^2_0 (\Omega)$ to~\eqref{Res} subject to~\eqref{Dir} with right-hand side $f$ satisfies
\begin{align*}
 \min\{ 1 , \lvert \lambda \rvert^{\alpha} \} \| \phi \|_{\L^2 (\Omega)} \leq C \| f \|_{\L^2_{\sigma} (\Omega)}.
\end{align*}
\end{proposition}

\begin{proof}
Let $f \in \L^2_{\sigma} (\Omega)$ and $u \in \dom(A)$ with $\lambda u + A u = f$. Notice that by~\eqref{Eq: Variational Dirichlet} it holds $- \Delta u + \nabla \phi = A u$ in the sense of distributions. Let $B : \L^2_0 (\Omega) \to \H^1_0 (\Omega ; \IC^d)$ denote the Bogovski\u{\i} operator and define the test function $v := B \phi$. Then
\begin{align*}
 \int_{\Omega} \lvert \phi \rvert^2 \; \d x = \int_{\Omega} \phi \, \overline{\divergence(B \phi)} \; \d x = \langle \nabla u , \nabla B \phi \rangle_{\L^2 , \L^2} - \langle A u , B \phi \rangle_{\L^2 , \L^2}.
\end{align*}
The first term on the right-hand side is estimated by the boundedness of $B$ and by Proposition~\ref{Prop: Resolvent}~\eqref{Prop: Dirichlet} as
\begin{align*}
 \lvert \langle \nabla u , \nabla B \phi \rangle_{\L^2 , \L^2} \rvert \leq C \lvert \lambda \rvert^{- 1 / 2} \| f \|_{\L^2_{\sigma} (\Omega)} \| \phi \|_{\L^2 (\Omega)}.
\end{align*}
To bound the second term, use that $A u = \IP A u$, that $\IP$ is self-adjoint, and that $\IP$ maps the Bessel potential space $\H^{2 \alpha} (\Omega ; \IC^d)$ boundedly into $\dom(A^{\alpha})$ whenever $0 \leq \alpha < 1 / 4$ (this follows by combining~\cite[Prop.~2.16]{Mitrea_Monniaux} with~\cite[Thm.~5.1]{Mitrea_Monniaux}). Thus, for $0 \leq \alpha < 1 / 4$ it holds
\begin{align*}
 \lvert \langle A u , B \phi \rangle_{\L^2 , \L^2} \rvert = \lvert \langle A^{1 - \alpha} u , A^{\alpha} \IP B \phi \rangle_{\L^2 , \L^2} \rvert \leq C \lvert \lambda \rvert^{- \alpha} \| f \|_{\L^2_{\sigma} (\Omega)} \| \phi \|_{\L^2 (\Omega)}.
\end{align*}
Notice that the estimate on $A^{1 - \alpha} u$ follows by writing $u = (\lambda + A)^{-1} f$ and by using the moment inequality~\cite[Prop.~6.6.4]{Haase}. \par
For the improved inequality for small $\lambda$, use the invertibility of the Stokes operator and estimate
\begin{align*}
  \langle \nabla A^{-1} A u , \nabla B \phi \rangle_{\L^2 , \L^2} - \langle A u , B \phi \rangle_{\L^2 , \L^2} \leq C \| A u \|_{\L^2_{\sigma} (\Omega)} \| \phi \|_{\L^2 (\Omega)} \leq C \| f \|_{\L^2_{\sigma} (\Omega)} \| \phi \|_{\L^2 (\Omega)}. &\qedhere
\end{align*}
\end{proof}

Comparing this estimate with the corresponding estimate for Neumann-type boundary conditions, one sees that there is a lack of an exponent of $1 / 4$ in the decay rate as $\lvert \lambda \rvert \to \infty$. As the proof for the decay estimate for no-slip boundary conditions relied on the construction of an appropriate test function, one might wonder whether the test function was just a ``bad choice'' and whether one could do better by choosing a more subtle test function. The following proposition shows that this is not the case, i.e., that the decay rate above is optimal.

\begin{proposition}
\label{Prop: Sharpness on L2}
Let $\Omega$ be a bounded domain with $\C^4$-boundary, $\theta \in (\pi / 2 , \pi)$, and $\alpha > 1 / 4$. Then for all $n \in \IN$ there exist $f_n \in \L^2_{\sigma} (\Omega)$ and $\lambda_n \in \S_{\theta}$ with $\lvert \lambda_n \rvert \geq 1$ such that the to~\eqref{Res} subject to~\eqref{Dir} with right-hand side $f_n$ associated pressure $\phi_n \in \L^2_0 (\Omega)$ satisfies
\begin{align*}
 \lvert \lambda_n \rvert^{\alpha} \| \phi_n \|_{\L^2 (\Omega)} > n \| f_n \|_{\L^2_{\sigma} (\Omega)}.
\end{align*}
\end{proposition}

\begin{proof}
We argue by contradiction and assume without loss of generality that $1 / 4 < \alpha \leq 1 / 2$. Assume that there exists $C > 0$ such that for all $f \in \L^2_{\sigma} (\Omega)$, $\lambda \in \S_{\theta}$ with $\lvert \lambda \rvert \geq 1$, and the to~\eqref{Res} and~\eqref{Dir} associated pressure $\phi$ it holds
\begin{align}
\label{Eq: First contradiction estimate}
 \lvert \lambda \rvert^{\alpha} \| \phi \|_{\L^2 (\Omega)} \leq C \| f \|_{\L^2_{\sigma} (\Omega)}.
\end{align}
Let $u \in \dom(A)$ with $\lambda u + A u = f$ and notice by~\eqref{Eq: Variational Dirichlet} that $- \Delta u + \nabla \phi = A u$ holds in the sense of distributions. Employing Proposition~\ref{Prop: Resolvent}~\eqref{Prop: Dirichlet} and~\eqref{Eq: First contradiction estimate} it follows
\begin{align*}
 \lvert \lambda \rvert^{\alpha} \| A u \|_{\H^{-1} (\Omega ; \IC^d)} \leq \lvert \lambda \rvert^{\alpha} \| \nabla u \|_{\L^2 (\Omega ; \IC^{d^2})} + \lvert \lambda \rvert^{\alpha} \| \phi \|_{\L^2 (\Omega)} \leq C \| f \|_{\L^2_{\sigma} (\Omega)}.
\end{align*}
By duality, there exists $C > 0$ such that for all $g \in \H^1_0 (\Omega ; \IC^d)$ and $\lambda \in \S_{\theta}$ with $\lvert \lambda \rvert \geq 1$ it holds 
\begin{align}
\label{Eq: Decay resolvent estimate}
 \lvert \lambda \rvert^{\alpha} \| A (\lambda + A)^{-1} \IP g \|_{\L^2_{\sigma} (\Omega)} \leq C \| g \|_{\H^1_0 (\Omega ; \IC^d)}.
\end{align}
Similarly to~\cite[Prop.~3.7]{Tolksdorf_Watanabe}, use~\eqref{Eq: Decay resolvent estimate} and~\eqref{Eq: Cauchy integrals}, to deduce a semigroup estimate of the form
\begin{align}
\label{Eq: Decay semigroup estimate}
 t^{1 - \alpha} \| A \e^{- t A} \IP g \|_{\L^2_{\sigma} (\Omega)} \leq C \| g \|_{\H^1_0 (\Omega)} \qquad (0 < t \leq 1).
\end{align}
Next, we estimate for a natural number $n \in \IN$ the term $(t A)^n \e^{- t A} \IP g$. To this end, write
\begin{align}
\label{Eq: Semigroup decomposition}
 (t A)^n \e^{- t A} \IP g = t^{\alpha} \e^{- \frac{1}{n + 1} t A} (t A \e^{- \frac{1}{n + 1} t A})^{n - 1} t^{1 - \alpha} A \e^{- \frac{1}{n + 1} t A} \IP g.
\end{align}
The first semigroup term in the product on the right-hand side is estimated by a combination of the interpolation inequality $\| \cdot \|_{\H^{2 \alpha}} \leq C \| \cdot \|_{\L^2}^{1 - 2 \alpha} \| \nabla \cdot \|_{\L^2}^{2 \alpha}$ with the uniform boundedness of the semigroup $\e^{-t A}$ as a family on $\L^2_{\sigma} (\Omega)$ and the gradient estimate~\eqref{Eq: Gradient estimates semigroup} as
\begin{align}
\label{Eq: First term in product}
\begin{aligned}
 t^{\alpha} \| \e^{- \frac{1}{n + 1} t A} h \|_{\H^{2 \alpha} (\Omega ; \IC^d)} &\leq C \| \e^{- \frac{1}{n + 1} t A} h \|_{\L^2_{\sigma} (\Omega)}^{1 - 2 \alpha} \Big( t^{1 / 2} \| \nabla \e^{- \frac{1}{n + 1} t A} h \|_{\L^2 (\Omega ; \IC^{d^2})} \Big)^{2 \alpha} \\
 &\leq C (n + 1)^{\alpha} \| h \|_{\L^2_{\sigma} (\Omega)}.
\end{aligned}
\end{align}
This holds for all $h \in \L^2_{\sigma} (\Omega)$. The term in the center of the product on the right-hand side of~\eqref{Eq: Semigroup decomposition} is estimated by~\eqref{Eq: Real characterization} by
\begin{align}
\label{Eq: Second term in product}
 \| (t A \e^{- \frac{1}{n + 1} t A})^{n - 1} h \|_{\L^2_{\sigma}} \leq (C (n + 1))^{n - 1} \| h \|_{\L^2_{\sigma} (\Omega)} \qquad (h \in \L^2_{\sigma} (\Omega)).
\end{align}
Finally, the last term in~\eqref{Eq: Semigroup decomposition} is estimated by using~\eqref{Eq: Decay semigroup estimate} yielding
\begin{align}
\label{Eq: Third term in product}
 \| t^{1 - \alpha} A \e^{- \frac{1}{n + 1} t A} \IP g \|_{\L^2_{\sigma} (\Omega)} \leq C (n + 1)^{1 - \alpha} \| g \|_{\H^1_0 (\Omega ; \IC^d)}.
\end{align}
Combining~\eqref{Eq: Semigroup decomposition},~\eqref{Eq: First term in product},~\eqref{Eq: Second term in product}, and~\eqref{Eq: Third term in product} and using that $n^n \leq n ! \e^n$ (Stirling formula!) finally yields
\begin{align}
\label{Eq: Estimate of powers}
 \| (t A)^n \e^{- t A} \IP g \|_{\H^{2 \alpha} (\Omega ; \IC^d)} \leq (C (n + 1))^{n + 1} \| g \|_{\H^1_0 (\Omega ; \IC^d)} \leq (n + 1)! (C \e)^{n + 1} \| g \|_{\H^1_0 (\Omega ; \IC^d)}.
\end{align}
To proceed, let $0 < t \leq 1$ and $s \in \IR$ with $\lvert s \rvert$ being small enough. Since $\e^{- t A}$ is an analytic semigroup, it can be written by its Taylor expansion
\begin{align*}
 \e^{- (t + s) A} = \sum_{n = 0}^{\infty} \frac{(-1)^n s^n}{n !} A^n \e^{- t A}.
\end{align*}
Combining this with~\eqref{Eq: Estimate of powers} finally yields if $\lvert s \rvert < t / (4 C \e)$ by using $(n + 1) \leq 2^n$
\begin{align*}
 \| \e^{- (t + s) A} \IP g \|_{\H^{2 \alpha} (\Omega ; \IC^d)} \leq C \e \sum_{n = 0}^{\infty} \bigg(\frac{\lvert s \rvert C \e}{t}\bigg)^n (n + 1) \| g \|_{\H^1_0 (\Omega ; \IC^d)} < 2 C \e \| g \|_{\H^1_0 (\Omega ; \IC^d)}.
\end{align*}
Especially, if $s = 0$, this shows that the family of operators $(\e^{- t A} \IP)_{0 < t \leq 1}$ is uniformly bounded in the space $\Lop(\H^1_0 (\Omega ; \IC^d) , \H^{2 \alpha} (\Omega ; \IC^d))$. To conclude the argument, let $(t_n)_{n \in \IN} \subset (0 , 1]$ converge to zero. Notice that $\e^{- t A} \IP g \to \IP g$ in $\L^2_{\sigma} (\Omega)$ as $t \to 0$ by the strong continuity of the semigroup. Since $(\e^{- t_n A} \IP g)_{n \in \IN}$ is uniformly bounded in the space $\H^{2 \alpha} (\Omega ; \IC^d)$, for any $0 < \eps \leq 2 \alpha$ there exists a convergent subsequence in the space $\H^{2 \alpha - \eps} (\Omega ; \IC^d)$ by the Theorem of Rellich and Kondrachov. Denoting the subsequence again by $(t_n)_{n \in \IN}$ we have that $\e^{- t_n A} \IP g \to \IP g$ as $n \to \infty$ in $\H^{2 \alpha - \eps} (\Omega ; \IC^d)$. Notice that $2 \alpha > 1 / 2$ and choose $\eps$ small enough such that $2 \alpha - \eps > 1 / 2$ holds. Now, the trace operator $\tr$ is well-defined on the space $\H^{2 \alpha - \eps} (\Omega ; \IC^d)$ and it is continuous from $\H^{2 \alpha - \eps} (\Omega ; \IC^d)$ to $\L^2 (\partial \Omega ; \IC^d)$. Consequently,
\begin{align*}
 0 = \lim_{n \to \infty} \tr(\e^{- t A} \IP g) = \tr(\IP g).
\end{align*}
We thus proved that for any $g \in \H^1_0 (\Omega ; \IC^d)$ the trace of $\IP g$ to $\partial \Omega$ is zero. This contradicts Lemma~\ref{Lem: Trace of Helmholtz}.
\end{proof}

In the following, we do the same analysis for right-hand sides in $\H^{-1} (\Omega ; \IC^d)$. We start with the following lemma, relating an estimate on the $\L^2$-norm of $\phi$ to a corresponding estimate on the $\H^{-1}$-norm of $u$.

\begin{lemma}
\label{Lem: Equivalence}
Let $\theta \in [0 , \pi)$ and $0 \leq \alpha \leq 1 / 2$. Then the following are equivalent:
\begin{enumerate}
 \item \label{Lem: Pressure estimate} There exists a constant $C > 0$ such that for all $F \in \H^{-1} (\Omega ; \IC^d)$ and $\lambda \in \S_{\theta}$ with $\lvert \lambda \rvert \geq 1$ the associated pressure $\phi \in \L^2_0 (\Omega)$ to~\eqref{Res} subject to~\eqref{Dir} and right-hand side $F$ satisfies
\begin{align*}
 \| \phi \|_{\L^2 (\Omega)} \leq C \lvert \lambda \rvert^{\alpha} \| F \|_{\H^{-1} (\Omega ; \IC^d)}.
\end{align*}
 \item \label{Lem: Velocity estimate} There exists a constant $C > 0$ such that for all $F \in \H^{-1} (\Omega ; \IC^d)$ and $\lambda \in \S_{\theta}$ with $\lvert \lambda \rvert \geq 1$ the function $u := (\lambda + \cA)^{-1} F$ satisfies
\begin{align*}
 \| u \|_{\H^{-1} (\Omega ; \IC^d)} \leq C \lvert \lambda \rvert^{\alpha - 1} \| F \|_{\H^{-1} (\Omega ; \IC^d)}.
\end{align*}
\end{enumerate}
\end{lemma}

\begin{proof}
To prove~\eqref{Lem: Velocity estimate}~$\Rightarrow$~\eqref{Lem: Pressure estimate}, use~\eqref{Eq: Variational Dirichlet} and choose as a test function $v := B \phi$ with $B : \L^2_0 (\Omega) \to \H^1_0 (\Omega ; \IC^d)$ being the Bogovski\u{\i} operator. Indeed, this together with Proposition~\ref{Prop: Resolvent}~\eqref{Prop: Dirichlet} yields
\begin{align}
\label{Eq: Small lambda negative}
\begin{aligned}
 \int_{\Omega} \lvert \phi \rvert^2 \; \d x &= \lambda \int_{\Omega} u \cdot \overline{B \phi} \; \d x + \int_{\Omega} \nabla u \cdot \overline{\nabla B \phi} \; \d x - \langle F , B \phi \rangle_{\H^{-1} , \H^1_0} \\
 &\leq C \Big( \lvert \lambda \rvert \| u \|_{\H^{-1} (\Omega ; \IC^d)} + \| \nabla u \|_{\L^2 (\Omega ; \IC^{d^2})} + \| F \|_{\H^{-1} (\Omega ; \IC^d)} \Big) \| \phi \|_{\L^2 (\Omega)} \\
 &\leq C \lvert \lambda \rvert^{\alpha} \| F \|_{\H^{-1} (\Omega ; \IC^d)} \| \phi \|_{\L^2 (\Omega)}.
\end{aligned}
\end{align}
The estimate is concluded by dividing by $\| \phi \|_{\L^2 (\Omega)}$. \par
To prove~\eqref{Lem: Pressure estimate}~$\Rightarrow$~\eqref{Lem: Velocity estimate}, write by virtue of~\eqref{Eq: Variational Dirichlet}
\begin{align*}
 \lvert \lambda \rvert \sup_{\substack{v \in \H^1_0 (\Omega ; \IC^d) \\ \| v \|_{\H^1_0} \leq 1}} \Big\lvert \int_{\Omega} u \cdot \overline{v} \; \d x \Big\rvert &= \sup_{\substack{v \in \H^1_0 (\Omega ; \IC^d) \\ \| v \|_{\H^1_0} \leq 1}} \Big\lvert \int_{\Omega} \nabla u \cdot \overline{\nabla v} \; \d x - \int_{\Omega} \phi \; \overline{\divergence(v)} \; \d x - \langle F , v \rangle_{\H^{-1} , \H^1_0} \Big\rvert
\end{align*}
and conclude by means of H\"older's inequality, Proposition~\ref{Prop: Resolvent}~\eqref{Prop: Dirichlet}, and the presumed estimate on the pressure.
\end{proof}

We start by establishing of the actual estimates being valid and prove their sharpness afterwards.

\begin{proposition}
\label{Prop: Pressure estimate in H-1}
Let $\Omega$ be a bounded Lipschitz domain and $\theta \in (0 , \pi]$. For all $1 / 4 < \alpha \leq 1 / 2$, there exists a constant $C > 0$ such that for all $F \in \H^{-1} (\Omega ; \IC^d)$ and $\lambda \in \S_{\theta}$ the associated pressure $\phi \in \L^2_0 (\Omega)$ to~\eqref{Res} subject to~\eqref{Dir} and right-hand side $F$ satisfies
\begin{align*}
 \| \phi \|_{\L^2 (\Omega)} \leq C \max\{ 1 , \lvert \lambda \rvert^{\alpha} \} \| F \|_{\H^{-1} (\Omega ; \IC^d)}.
\end{align*}
\end{proposition}

\begin{proof}
First of all, notice that the calculation carried out in~\eqref{Eq: Small lambda negative} already gives the uniform boundedness of the constant for all $\lambda \in \S_{\theta}$ with $\lvert \lambda \rvert < 1$ and thus, leaving us with the task to prove estimates in the case $\lvert \lambda \rvert \geq 1$. In this case, Lemma~\ref{Lem: Equivalence} reduces the problem to bound the $\H^{-1}$-norm of $u$. \par
To this end, let $F \in \H^{-1} (\Omega ; \IC^d)$ and $u := (\lambda + \cA)^{-1} F$. Since $u \in \L^2_{\sigma} (\Omega)$ and $\IP$ is self-adjoint one finds
\begin{align*}
 \lambda \int_{\Omega} u \cdot \overline{v} \; \d x = \lambda \int_{\Omega} u \cdot \overline{\IP v} \; \d x.
\end{align*}
By~\cite[Prop.~2.16]{Mitrea_Monniaux}, $\IP$ maps $\H^{1 - 2 \alpha} (\Omega ; \IC^d)$ boundedly into $\H^{1 - 2 \alpha}_{\sigma} (\Omega)$, so that
\begin{align*}
 \Big\lvert \lambda \int_{\Omega} u \cdot \overline{\IP v} \; \d x \Big\rvert \leq C \lvert \lambda \rvert \| u \|_{\H^{2 \alpha - 1}_{\sigma} (\Omega)} \| v \|_{\H^1_0 (\Omega ; \IC^d)}.
\end{align*}
Since the space $\H^{2 \alpha - 1}_{\sigma} (\Omega)$ coincides with $\dom(\cA^{\alpha})$, compare~\eqref{Eq: Fractional power domains on H-1}, one finds
\begin{align*}
 \lvert \lambda \rvert \| u \|_{\H^{2 \alpha - 1}_{\sigma} (\Omega)} \leq C \lvert \lambda \rvert \| \cA^{\alpha} (\lambda + \cA)^{-1} F \|_{\H^{- 1}_{\sigma} (\Omega)} \leq C \lvert \lambda \rvert^{\alpha} \| F \|_{\H^{-1}_{\sigma} (\Omega)}.
\end{align*}
Now, the continuous inclusion $\H^{-1} (\Omega ; \IC^d) \hookrightarrow \H^{-1}_{\sigma} (\Omega)$ concludes the proof.
\end{proof}

Finally, we prove that this bound is in fact sharp.

\begin{proposition}
\label{Prop: Failure of pressure estimate}
Let $\Omega$ be a bounded domain with $\C^4$-boundary, $\theta \in (\pi / 2 , \pi)$, and $0 \leq \alpha < 1 / 4$. Then for all $n \in \IN$ there exist $F_n \in \H^{-1} (\Omega ; \IC^d)$ and $\lambda_n \in \S_{\theta}$ with $\lvert \lambda_n \rvert \geq 1$ such that the to~\eqref{Res} subject to~\eqref{Dir} with right-hand side $F_n$ associated pressure $\phi_n \in \L^2_0 (\Omega)$ satisfies
\begin{align*}
 \| \phi_n \|_{\L^2 (\Omega)} > n \lvert \lambda_n \rvert^{\alpha} \| F_n \|_{\H^{-1} (\Omega ; \IC^d)}.
\end{align*}
\end{proposition}

\begin{proof}
We argue by contradiction. Hence by virtue of Lemma~\ref{Lem: Equivalence}, we assume that there exists $0 \leq \alpha < 1 / 4$ and $C > 0$ such that for all $F \in \H^{-1}_{\sigma} (\Omega)$ and $\lambda \in \S_{\theta}$ with $\lvert \lambda \rvert \geq 1$ it holds
\begin{align*}
 \| (\lambda + \cA)^{-1} F \|_{\H^{-1} (\Omega ; \IC^d)} \leq C \lvert \lambda \rvert^{\alpha - 1} \| F \|_{\H^{-1} (\Omega ; \IC^d)}.
\end{align*}
By duality and~\eqref{Eq: Resolvent restriction of weak resolvent}, there exists $C > 0$ such that for all $\lambda \in \S_{\theta}$ with $\lvert \lambda \rvert \geq 1$ and all $g \in \H^1_0 (\Omega ; \IC^d)$ it holds
\begin{align}
\label{Eq: Contradictory resolvent bound}
 \lvert \lambda \rvert^{1 - \alpha} \| (\lambda + A)^{-1} \IP g \|_{\H^1_{0 , \sigma} (\Omega)} \leq C \| g \|_{\H^1_0 (\Omega ; \IC^d)}.
\end{align}
Following the proof of~\cite[Prop.~3.7]{Tolksdorf_Watanabe}, the estimate~\eqref{Eq: Contradictory resolvent bound} in combination with~\eqref{Eq: Cauchy integrals} lead to the semigroup estimate
\begin{align}
\label{Eq: Contradiction estimate}
 t^{\alpha} \| \e^{- t A} \IP g \|_{\H^1_{0 , \sigma} (\Omega)} \leq C \| g \|_{\H^1_0 (\Omega ; \IC^d)} \qquad (0 < t \leq 1).
\end{align}
Next, we are going to estimate as in the proof of Proposition~\ref{Prop: Sharpness on L2} for a natural number $n \in \IN$ and $0 < t \leq 1$ the term $(t A)^n \e^{- t A} \IP g$. To this end, write
\begin{align}
\label{Eq: Semigroup decomposition negative}
  (t A)^n \e^{- t A} \IP g = t^{1 - \alpha} A^{1 / 2} \e^{- \frac{1}{n + 1} t A} (t A \e^{- \frac{1}{n + 1} t A})^{n - 1} A^{1 / 2} t^{\alpha} \e^{- \frac{1}{n + 1} t A} \IP g.
\end{align}
The first term in the product on the right-hand side is estimated by means of the interpolation inequality $\| \cdot \|_{\H^{1 - 2 \alpha}} \leq C \| \cdot \|_{\L^2}^{2 \alpha} \| \nabla \cdot \|_{\L^2}^{1 - 2 \alpha}$, the uniform boundedness of the semigroup $\e^{-t A}$ as a family on $\L^2_{\sigma} (\Omega)$, the gradient estimate~\eqref{Eq: Gradient estimates semigroup}, and~\eqref{Eq: Fractional parabolic smoothing}. Indeed, for all $h \in \L^2_{\sigma} (\Omega)$, we have
\begin{align}
\label{Eq: First product estimate negative}
\begin{aligned}
 &t^{1 - \alpha} \| A^{1 / 2} \e^{- \frac{1}{n + 1} t A} h \|_{\H^{1 - 2 \alpha} (\Omega ; \IC^d)} \\
 &\leq C \Big(t^{1 / 2} \| \e^{- \frac{1}{2 (n + 1)} t A} A^{1 / 2} \e^{- \frac{1}{2 (n + 1)} t A} h \|_{\L^2_{\sigma} (\Omega)} \Big)^{2 \alpha} \Big( t \| \nabla \e^{- \frac{1}{2 (n + 1)} t A} A^{1 / 2} \e^{- \frac{1}{2 (n + 1)} t A} h \|_{\L^2 (\Omega ; \IC^{d^2})} \Big)^{1 - 2 \alpha} \\
 &\leq C (n + 1)^{1 / 2 - \alpha} \Big(t^{1 / 2} \| A^{1 / 2} \e^{- \frac{1}{2 (n + 1)} t A} h \|_{\L^2_{\sigma} (\Omega)} \Big)^{2 \alpha} \Big( t^{1 / 2} \| A^{1 / 2} \e^{- \frac{1}{2 (n + 1)} t A} h \|_{\L^2_{\sigma} (\Omega)} \Big)^{1 - 2 \alpha} \\
 &\leq C (n + 1)^{1 - \alpha} \| h \|_{\L^2_{\sigma} (\Omega)}.
\end{aligned}
\end{align}
The second term in the product in~\eqref{Eq: Semigroup decomposition negative} was already estimated in~\eqref{Eq: Second term in product}. The third term in the product in~\eqref{Eq: Semigroup decomposition negative} is finally estimated, by using~\eqref{Eq: Kato square root} and~\eqref{Eq: Contradiction estimate} by
\begin{align}
\label{Eq: Third product estimate negative}
 \| A^{1 / 2} t^{\alpha} \e^{- \frac{1}{n + 1} t A} \IP g \|_{\L^2_{\sigma} (\Omega)} \leq t^{\alpha} \| \e^{- \frac{1}{n + 1} t A} \IP g \|_{\H^1_{0 , \sigma} (\Omega)} \leq C (n + 1)^{\alpha} \| g \|_{\H^1_0 (\Omega ; \IC^d)}.
\end{align}
Combining~\eqref{Eq: Semigroup decomposition negative},~\eqref{Eq: First product estimate negative},~\eqref{Eq: Second term in product},~\eqref{Eq: Third product estimate negative}, and using $n^n \leq n ! \e^n$ (Stirling formula!) finally yields
\begin{align*}
  \| (t A)^n \e^{- t A} \IP g \|_{\H^{1 - 2 \alpha} (\Omega ; \IC^d)} \leq (C (n + 1))^{n + 1} \| g \|_{\H^1_0 (\Omega ; \IC^d)} \leq (n + 1)! (C \e)^{n + 1} \| g \|_{\H^1_0 (\Omega ; \IC^d)}.
\end{align*}
The rest of the contradiction argument follows exactly the lines below~\eqref{Eq: Estimate of powers} in the proof of Proposition~\ref{Prop: Sharpness on L2} and is thus omitted.
\end{proof}

Recall that in order to derive the estimates in the case of Neumann-type boundary conditions in Proposition~\ref{Prop: Pressure estimate Neumann} it was needed that solutions to the Poisson problem with right-hand side in $\L^2 (\Omega)$ admit $\H^2$-regularity. Thus, this proof cannot be carried out on general bounded Lipschitz domains. However, as all objects appearing in the estimate in Proposition~\ref{Prop: Pressure estimate Neumann} exist if the boundary of $\Omega$ is merely Lipschitz. Thus, one might wonder whether Proposition~\ref{Prop: Pressure estimate Neumann} is true on general Lipschitz domains. Unfortunately, one cannot deduce the validity of these estimates by approximating the Lipschitz domain by smooth domains as the constants in the respective estimate blow up. If one wants to prove Stokes resolvent estimates in $\L^p$ for Neumann-type boundary conditions on mere Lipschitz domains, it would be tempting to imitate Shen's proof~\cite{Shen} carried out for no-slip boundary conditions. As it was described in the introduction, a corresponding weak reverse H\"older estimate might look as~\eqref{Eq: Reverse Holder introduction} but on general Lipschitz domains with $p := 2d / (d - 1)$. It was further described in the introduction, that an estimate of the form presented in Proposition~\ref{Prop: Pressure estimate Neumann} would help to achieve these resolvent estimates. In view of this, it would be interesting to know the answer to the following problem.

\begin{problem}
Prove or disprove the validity of~\eqref{Eq: Neumann pressure L2} if $\Omega$ is a bounded Lipschitz domain.
\end{problem}

\section{Regularity estimates in convex domains}
\label{Sec: Estimates on convex domains}

\noindent If $\Omega$ is a bounded and convex domain, it is well-known that weak solutions to the Poisson problem with homogeneous Dirichlet or Neumann boundary conditions and right-hand side in $\L^2 (\Omega)$ admit $\H^2$-regularity. To understand a rough sketch of its proof, we need to introduce some notions from geometry. \par
If $\Omega \subset \IR^d$ is a bounded domain with $\C^2$-boundary (not necessarily convex), and if after a suitable translation and rotation of $\Omega$ the function $\varphi : \IR^{d - 1} \to \IR$ locally describes the boundary of $\Omega$ around the point $p = (0 , \varphi(0))$, then, if the rotation is chosen such that $\nabla \varphi (0) = 0$, the second fundamental form $\II_p$ at this boundary point is the sesquilinear form given by
\begin{align*}
 \II_p (\xi ; \eta) = \frac{\partial^2 \varphi(0)}{\partial_{x_j} \partial_{x_k}} \xi_j \overline{\eta_k} \qquad (\xi , \eta \in \IC^{d - 1}).
\end{align*}
Notice that $\II_p (\cdot ; \cdot)$ is conjugate symmetric and thus $\II_p (\xi ; \xi)$ is a real number for each $\xi \in \IC^{d - 1}$. If $\Omega$ is convex and if $\Omega$ locally lies below the graph of $\varphi$, then $- \varphi$ is convex and thus the second fundamental form is non-positive, which means that
\begin{align}
\label{Eq: Convexity 1}
 \II_p (\xi ; \xi) \leq 0 \qquad (\xi \in \IC^{d - 1}).
\end{align}
Furthermore, if $\II_p$ denotes the matrix associated to the sesquilinear form $\II_p (\cdot ; \cdot)$, then convexity of $\Omega$ implies that
\begin{align}
\label{Eq: Convexity 2}
 \tr(\II_p) \leq 0.
\end{align}
In the following, we skip the subscript $p$ and keep in mind, that the second fundamental form varies from boundary point to boundary point. \par
To understand why the domain of the Laplacian embeds into $\H^2$ in convex domains, the following formula of integration by parts due to Grisvard is eminent~\cite[Thm.~3.1.1.1]{Grisvard}. Notice that in~\cite[Thm.~3.1.1.1]{Grisvard} this formula is derived for real-valued functions, but that a short analysis of its proof reveals the following formulation for complex-valued functions. Here and below, $\sigma$ generically denotes the surface measure of a set with a Lipschitz boundary. Recall further the notation $v_{\T}$ for the tangential component of a vector $v$ introduced in~\eqref{Eq: Tangential component}.

\begin{theorem}
\label{Thm: Grisvard}
Let $\Omega \subset \IR^d$ be a bounded domain with $\C^2$-boundary and let $v \in \C^{\infty} (\overline{\Omega} ; \IC^d)$. Then,
\begin{align*}
 \int_{\Omega} \lvert \divergence(v) \rvert^2 \; \d x - \int_{\Omega} \partial_j v_i \overline{\partial_i v_j} \; \d x &= - \int_{\partial \Omega} 2 \Re( v_{\T} \cdot \overline{\nabla_{\T} (v \cdot n)}) \; \d \sigma \\
 &\qquad - \int_{\partial \Omega} \big( \II (v_{\T} ; v_{\T}) + (\tr \II) \lvert v \cdot n \rvert^2 \big) \; \d \sigma.
\end{align*}
\end{theorem}

There is also a counterpart of Theorem~\ref{Thm: Grisvard} for piecewise $\C^2$-domains, see~\cite[Thm.~3.1.1.2]{Grisvard} for real-valued functions. To state the theorem, we adopt the definition by Grisvard, that a bounded Lipschitz domain $\Omega$ is said to be piecewise $\C^2$-regular if there exist $\Gamma_0 , \Gamma_1 \subset \partial \Omega$ with $\partial \Omega = \Gamma_0 \cup \Gamma_1$ and where $\Gamma_0$ has surface measure zero and for each $x \in \Gamma_1$ the boundary of $\partial \Omega$ can be described as the graph of a $\C^2$-function in a neighborhood of $x$.

\begin{theorem}
\label{Thm: Piecewise Grisvard}
Let $\Omega \subset \IR^d$ be a bounded domain with a piecewise $\C^2$-boundary and let $v \in \C^{\infty} (\overline{\Omega} ; \IC^d)$. Then,
\begin{align*}
 \int_{\Omega} \lvert \divergence(v) \rvert^2 \; \d x - \int_{\Omega} \partial_j v_i \overline{\partial_i v_j} \; \d x &= \int_{\Gamma_1} \big(\divergence_{\T} ([v \cdot n] \overline{v_{\T}}) - 2 \Re( v_{\T} \cdot \overline{\nabla_{\T} (v \cdot n)})\big) \; \d \sigma \\
 &\qquad - \int_{\Gamma_1} \II (v_{\T} ; v_{\T}) + (\tr \II) \lvert v \cdot n \rvert^2 \; \d \sigma.
\end{align*}
\end{theorem}

To deduce that weak solutions to the equation $- \Delta u = f$ with Dirichlet or Neumann boundary conditions lie in $\H^2 (\Omega)$ if $\Omega$ is bounded and convex, let first $\Omega$ be a bounded, convex, and smooth domain. If $f \in \C^{\infty} (\overline{\Omega} ; \IR)$, then $u \in \C^{\infty} (\overline{\Omega} ; \IR)$ by higher regularity of the Laplacian. Take $v := \nabla u$ and apply Theorem~\ref{Thm: Grisvard} together with~\eqref{Eq: Convexity 1} and~\eqref{Eq: Convexity 2} to deduce
\begin{align*}
 \int_{\Omega} \lvert \divergence(v) \rvert^2 \; \d x \geq  \int_{\Omega} \partial_j v_i \partial_i v_j \; \d x - 2 \int_{\partial \Omega} v_{\T} \cdot \nabla_{\T} (v \cdot n) \; \d \sigma.
\end{align*}
A computation of the first term on the right-hand side yields
\begin{align*}
 \int_{\Omega} \partial_j v_i \partial_i v_j \; \d x = \sum_{i , j = 1}^d \int_{\Omega} \lvert \partial_i \partial_j u \rvert^2 \; \d x
\end{align*}
and since $\divergence(v) = -f$, it remains to understand what the boundary integral does. Here, the boundary conditions enter the game. If $u$ satisfies homogeneous Dirichlet boundary conditions, i.e., $u = 0$ on $\partial \Omega$, then $v_{\T} = \nabla_{\T} u = 0$ and if $u$ satisfies homogeneous Neumann boundary conditions, then $v \cdot n = n \cdot \nabla u = 0$. Hence, by Theorem~\ref{Thm: Grisvard}, we infer
\begin{align*}
 \int_{\Omega} \lvert f \rvert^2 \; \d x \geq \sum_{i , j = 1}^d \int_{\Omega} \lvert \partial_i \partial_j u \rvert^2 \; \d x.
\end{align*}
By density, one obtains this estimate for all $f \in \L^2(\Omega)$. Finally, since the constant in this inequality is one, in particular, it is independent of properties of the boundary, one can conclude the $\H^2$-regularity for general bounded convex domains by an approximation argument.

\begin{remark}
\label{Rem: Approximating sequence}
Let us explain the approximation of a bounded and convex domain by a sequence of smooth, bounded, and convex domains $(\Omega_k)_{k \in \IN}$ with $\Omega_k \subset \Omega_{k + 1}$ and $\bigcup_{k \in \IN} \Omega_k = \Omega$ in more detail. \par
Let $\Omega$ be a bounded and convex domain and assume without loss of generality that $0 \in \Omega$. For $k \in \IN$ let $K_k$ denote the closure of $(1 - 2^{- k}) \Omega$ and notice that $K_k \subset (1 - 2^{- (k + 1)}) \Omega$ and that $K_k$ is a compact and convex set. In this situation,~\cite[Lem.~2.3.2]{Hormander} provides us with a compact and convex set $C_k$ with smooth boundary that satisfies $K_k \subset C_k \subset (1 - 2^{- (k + 1)}) \Omega$. Now, let $\Omega_k$ be defined as the interior of $C_k$. \par
One could also ask, whether the sets are uniform in certain properties. For example, for all $k \in \IN$ it holds $\frac{1}{2} \Omega \subset \Omega_k \subset \Omega$ so that $\diam(\Omega) / 2 \leq \diam(\Omega_k) \leq \diam(\Omega)$. Another property is a uniform $d$-set property, which is the following: Let $r_0 > 0$ be such that $B := B(0 , r_0) \subset \frac{1}{2} \Omega$ so that $B \subset \Omega_k$ for all $k \in \IN$. Let $x_0 \in \partial \Omega_k$. Since $\Omega_k$ is convex, for all $t \in [0 , 1)$ and $x \in B$ the points $(1 - t) x + t x_0$ are contained in $\Omega_k$. This implies that $\Omega_k$ contains a cone with vertex at $x_0$, height $h = \lvert x_0 \rvert \geq r_0$, and opening angle $\omega = 2 \arctan (r_0 / \lvert x_0 \rvert)$. Since $\lvert x_0 \rvert \leq \diam(\Omega)$ we find $\omega \geq 2 \arctan(r_0 / \diam(\Omega))$. Thus, if $Q (x_0 , r)$ is a cube centered in $x_0$ and diameter $0 < r \leq 2 r_0$, then there exists a constant $C > 0$ depending only on $r_0$, $\diam(\Omega)$, and $d$ such that
\begin{align}
\label{Eq: d-set property}
 \lvert Q (x_0 , r) \cap \Omega_k \rvert \geq C r^d.
\end{align}
Notice that if $R_0 > r_0$, then for all $2 r_0 < r \leq 2 R_0$ it holds
\begin{align*}
 \lvert Q (x_0 , r) \cap \Omega_k \rvert \geq \lvert Q (x_0 , r_0) \cap \Omega_k \rvert \geq C r_0^d = \frac{C r_0^d}{(2 R_0)^d} r^d.
\end{align*}
Thus, we can assume that for all $R_0 > 0$ there exists a constant $C > 0$ depending only on $r_0$, $R_0$, $\diam(\Omega)$, and $d$ such that for all $k \in \IN$ and all $x_0 \in \partial \Omega_k$ the inequality~\eqref{Eq: d-set property} holds.
\end{remark}

Let $\Omega$ again be a bounded convex domain with smooth boundary. If $u$ and $\phi$ satisfy
\begin{align*}
\left\{ \begin{aligned}
 - \Delta u + \nabla \phi &= f &&\text{in } \Omega \\
 \divergence(u) &= 0 &&\text{in } \Omega \\
 u &= 0 &&\text{on } \partial \Omega,
\end{aligned} \right.
\end{align*}
with $f$ being smooth up to the boundary, one could try to imitate the calculations for the Laplacian above. To this end, there are at least two obvious choices for $v$. Fix $1 \leq \beta \leq d$. For the first choice, define $v_{\beta} := \nabla u_{\beta}$. Clearly, all boundary integrals as well as the integral involving the mixed product can be handled as above. However, $\divergence(v_{\beta}) = - f_{\beta} + \partial_{\beta} \phi$, so that the gradient of the pressure appears on the right-hand side of the inequality, which is an unfortunate situation. \par
Another choice for $v$ should incorporate that $\divergence(v) = - f_{\beta}$. For the $\beta$th component of the equation, this is accomplished by choosing $v_{\beta} := \nabla u_{\beta} - \phi \e_{\beta}$, where $\e_{\beta}$ denotes the $\beta$th unit basis vector. Moreover, convexity deals with the terms involving the second fundamental form, and one directly verifies that the mixed product (for $\beta$ fix) computes as
\begin{align*}
  \partial_j (v_{\beta})_i \partial_i (v_{\beta})_j = \partial_i \partial_j u_{\beta} \partial_i \partial_j u_{\beta}  + \lvert \partial_{\beta} \phi \rvert^2 - 2 \nabla \partial_{\beta} u_{\beta} \cdot \nabla \phi.
\end{align*}
Next, a summation over $\beta$ yields due to the solenoidality of $u$ (notice that we now sum over repeated indices as usual)
\begin{align*}
 \partial_j (v_{\beta})_i \partial_i (v_{\beta})_j &= \partial_i \partial_j u_{\beta} \partial_i \partial_j u_{\beta} + \lvert \nabla \phi \rvert^2.
\end{align*}
Altogether, we find
\begin{align*}
  \int_{\Omega} \lvert f \rvert^2 &\geq \sum_{i , j , \beta = 1}^d \int_{\Omega} \lvert \partial_i \partial_j u_{\beta} \rvert^2 \; \d x + \int_{\Omega} \lvert \nabla \phi \rvert^2 \; \d x - 2 \int_{\partial \Omega} (v_{\beta})_{\T} \cdot \nabla_{\T} (v_{\beta} \cdot n) \; \d \sigma.
\end{align*}
Unfortunately, one cannot simply conclude that the boundary integral vanishes as nothing is known about the trace of the pressure on the boundary of $\Omega$. However, imposing for example the Neumann-type boundary condition
\begin{align*}
 n \cdot \nabla u - \phi n = 0
\end{align*}
seems to be better suited for this approach as in this case the function $v_{\beta}$ turns out to have the additional property that $v_{\beta} \cdot n = 0$ on $\partial \Omega$. For more general Neumann-type boundary conditions and the resolvent problem this is made precise in the following theorem.

\begin{theorem}
\label{Thm: H2 regularity on convex domains}
Let $\Omega \subset \IR^d$, $d \geq 2$, be a bounded convex domain, $\mu \in (-1 , \sqrt{2} - 1)$, and $\theta \in (0 , \pi)$. Then for all $\lambda \in \S_{\theta}$ and all $f \in \L^2 (\Omega ; \IC^d)$ the weak solutions $u$ and $\phi$ to~\eqref{Res} subject to~\eqref{Neu} satisfy $u \in \H^2(\Omega ; \IC^d)$ and $\phi \in \H^1(\Omega)$. Moreover, there exists $C > 0$ depending only on $d$, $\mu$, and $\theta$ such that
\begin{align*}
  \lvert \lambda \rvert \int_{\Omega} \lvert \nabla u \rvert^2 \; \d x + \int_{\Omega} \lvert \nabla^2 u \rvert^2 \; \d x + \int_{\Omega} \lvert \nabla \phi \rvert^2 \; \d x \leq C \bigg( \int_{\Omega} \lvert f \rvert^2 \; \d x + \lvert \lambda \rvert^2 \int_{\Omega} \lvert u \rvert^2 \; \d x \bigg).
\end{align*}
\end{theorem}

\begin{proof}
Assume first that $\Omega$ has a $\C^{\infty}$-boundary and that $f \in \C^{\infty}_c (\Omega ; \IC^d)$. Then, by virtue of Remark~\ref{Rem: Higher regularity}, the functions $u$ and $\phi$ are smooth up to the boundary. Fix $1 \leq \beta \leq d$ and define
\begin{align}
\label{Eq: Grisvard testfunction}
 v_{\beta} := \big( \{ \delta_{l k} \delta_{\alpha \beta} + \mu \delta_{l \beta} \delta_{k \alpha} \} \partial_l u_{\alpha} - \delta_{k \beta} \phi \big)_{k = 1}^d.
\end{align}
Since $u$ and $\phi$ solve~\eqref{Res} one readily verifies that
\begin{align}
\label{Eq: Divergence of auxiliary function}
 \divergence(v_{\beta}) = \{ \delta_{l k} \delta_{\alpha \beta} + \mu \delta_{l \beta} \delta_{k \alpha} \} \partial_k \partial_l u_{\alpha} - \delta_{k \beta} \partial_k \phi = \partial_l \partial_l u_{\beta} + \mu \partial_{\beta} \partial_{\alpha} u_{\alpha} - \partial_{\beta} \phi = \lambda u_{\beta} - f_{\beta}.
\end{align}
Moreover,
\begin{align}
\label{Eq: Boundary condition of auxiliary function}
 n \cdot v_{\beta} = n_k \{ \delta_{l k} \delta_{\alpha \beta} + \mu \delta_{l \beta} \delta_{k \alpha} \} \partial_l u_{\alpha} - n_k \delta_{k \beta} \phi = n_k \partial_k u_{\beta} + \mu n_k \partial_{\beta} u_k - \phi n_{\beta},
\end{align}
which coincides with the $\beta$th component of
\begin{align*}
 \{ D u + \mu [D u]^{\top} \} n - \phi n
\end{align*}
and thus vanishes on the boundary. The mixed product is calculated as follows (note that we also sum over $\beta$ in this calculation so that in particular $\partial_{\beta} u_{\beta} = 0$)
\begin{align*}
 \partial_j (v_{\beta})_i \overline{\partial_i (v_{\beta})_j} &= \big\{ [\delta_{l i} \delta_{\alpha \beta} + \mu \delta_{l \beta} \delta_{i \alpha}] \partial_j \partial_l u_{\alpha} - \delta_{i \beta} \partial_j \phi \big\} 	\big\{ [\delta_{l^{\prime} j} \delta_{\alpha^{\prime} \beta} + \mu \delta_{l^{\prime} \beta} \delta_{j \alpha^{\prime}}] \overline{\partial_i \partial_{l^{\prime}} u_{\alpha^{\prime}}} - \delta_{j \beta} \overline{\partial_i \phi} \big\} \\
 &= \partial_j \partial_i u_{\beta} \overline{\partial_i \partial_j u_{\beta}} + 2 \mu \Re( \partial_j \partial_i u_{\beta} \overline{\partial_i \partial_{\beta} u_j}) + \mu^2 \partial_j \partial_{\beta} u_i \overline{\partial_i \partial_{\beta} u_j} - 2 \mu \Re( \partial_{\beta} \partial_{\beta} u_i \overline{\partial_i \phi}) \\
 &\qquad + \partial_{\beta} \phi \overline{\partial_{\beta} \phi}.
\end{align*}
Relabelling the index variables yields
\begin{align*}
 \partial_j \partial_{\beta} u_i \overline{\partial_i \partial_{\beta} u_j} = \frac{1}{2} \partial_j \partial_{\beta} u_i \overline{\partial_i \partial_{\beta} u_j} + \frac{1}{2} \partial_i \partial_{\beta} u_j \overline{\partial_j \partial_{\beta} u_i} = \Re(\partial_j \partial_{\beta} u_i \overline{\partial_i \partial_{\beta} u_j}).
\end{align*}
Next, use that $\partial_{\beta} \partial_{\beta} u_i = \lambda u_i - f_i + \partial_i \phi$ to deduce
\begin{align}
\label{Eq: Cross product of auxiliary function}
\begin{aligned}
 \partial_j (v_{\beta})_i \overline{\partial_i (v_{\beta})_j} &= \partial_j \partial_i u_{\beta} \overline{\partial_i \partial_j u_{\beta}} + (2 \mu + \mu^2) \Re(\partial_j \partial_i u_{\beta} \overline{\partial_i \partial_{\beta} u_j}) + (1 - 2 \mu) \partial_{\beta} \phi \overline{\partial_{\beta} \phi} \\
 &\qquad + 2 \mu \Re(f_i \overline{\partial_i \phi}) - 2 \mu \Re(\lambda u_i \overline{\partial_i \phi}).
\end{aligned}
\end{align}
Finally, Young's inequality implies
\begin{align}
\label{Eq: Cross product asymmetric term}
\begin{aligned}
 (2 \mu + \mu^2) \Re(\partial_j \partial_i u_{\beta} \overline{\partial_i \partial_{\beta} u_j}) &\geq - \frac{\lvert 2 \mu + \mu^2 \rvert}{2} \partial_j \partial_i u_{\beta} \overline{\partial_j \partial_i u_{\beta}} - \frac{\lvert 2 \mu + \mu^2 \rvert}{2} \partial_i \partial_{\beta} u_j \overline{\partial_i \partial_{\beta} u_j} \\
 &= - \lvert 2 \mu + \mu^2 \rvert \partial_j \partial_i u_{\beta} \overline{\partial_j \partial_i u_{\beta}}.
\end{aligned}
\end{align}
Use the two rightmost representations of $\divergence(v_{\beta})$ in~\eqref{Eq: Divergence of auxiliary function} and an integration by parts together with the fact that $n \cdot v_{\beta} = 0$ on $\partial \Omega$ (due to~\eqref{Eq: Boundary condition of auxiliary function} and the imposed boundary condition), the representation of $v_{\beta}$ in~\eqref{Eq: Grisvard testfunction}, and the solenoidality of $u$ to deduce
\begin{align}
\label{Eq: Divergence in Grisvard formula}
\begin{aligned}
 \sum_{\beta = 1}^d \int_{\Omega} \lvert \divergence(v_{\beta}) \rvert^2 \; \d x &= \int_{\Omega} \divergence(v_{\beta}) \{\overline{\lambda u_{\beta}} - \overline{f_{\beta}}\} \; \d x \\
 &= - \overline{\lambda} \int_{\Omega} (v_{\beta})_k \overline{\partial_k u_{\beta}} \; \d x - \int_{\Omega} \{ \partial_l \partial_l u_{\beta} + \mu \partial_{\beta} \partial_{\alpha} u_{\alpha} - \partial_{\beta} \phi \} \overline{f_{\beta}} \; \d x \\
 &= - \overline{\lambda} \int_{\Omega} \{ \partial_k u_{\beta} + \mu \partial_{\beta} u_k \} \overline{\partial_k u_{\beta}} \; \d x - \int_{\Omega} \{ \partial_l \partial_l u_{\beta} - \partial_{\beta} \phi \} \overline{f_{\beta}} \; \d x.
\end{aligned}
\end{align}
Finally, apply Theorem~\ref{Thm: Grisvard} with $v = v_{\beta}$ and sum over $\beta$. By~\eqref{Eq: Cross product of auxiliary function} and since the term in~\eqref{Eq: Boundary condition of auxiliary function} vanishes on the boundary one finds after rearranging terms
\begin{align}
\label{Eq: Application of Grisvard formula}
\begin{aligned}
 &\overline{\lambda} \int_{\Omega} \{\partial_k u_{\beta} + \mu \partial_{\beta} u_k \} \overline{\partial_k u_{\beta}} \; \d x + \int_{\Omega} \partial_i \partial_j u_{\beta} \overline{\partial_i \partial_j u_{\beta}} \; \d x + (2 \mu + \mu^2) \int_{\Omega} \Re(\partial_j \partial_i u_{\beta} \overline{\partial_i \partial_{\beta} u_j}) \; \d x \\
 &\qquad + (1 - 2 \mu) \int_{\Omega} \partial_{\beta} \phi \overline{\partial_{\beta} \phi} \; \d x - \int_{\partial \Omega} \II ((v_{\beta})_{\T} ; (v_{\beta})_{\T}) \; \d \sigma \\
 &= - \int_{\Omega} \{ \partial_l \partial_l u_{\beta} - \partial_{\beta} \phi \} \overline{f_{\beta}} \; \d x - 2 \mu \int_{\Omega} \Re(f_i \overline{\partial_i \phi}) \; \d x + 2 \mu \int_{\Omega} \Re(\lambda u_i \overline{\partial_i \phi}) \; \d x.
\end{aligned}
\end{align}
Now, notice the following facts: If $\lambda \in \S_{\theta}$, then $\overline{\lambda} \in \S_{\theta}$. If $\lvert \mu \rvert \leq 1$, then
\begin{align}
\label{Eq: Mixed derivative}
 \{\partial_k u_{\beta} + \mu \partial_{\beta} u_k \} \overline{\partial_k u_{\beta}} = \lvert \nabla u \rvert^2 + \mu \Re (\partial_{\beta} u_k \overline{\partial_k u_{\beta}}) \geq (1 - \lvert \mu \rvert) \lvert \nabla u \rvert^2 \geq 0.
\end{align}
If $\lvert 2 \mu + \mu^2 \rvert < 1$, then the sum of the second and third integrals on the left-hand side of~\eqref{Eq: Application of Grisvard formula} is non-negative due to~\eqref{Eq: Cross product asymmetric term}. If $1 - 2 \mu > 0$, then the fourth integral on the left-hand side of~\eqref{Eq: Application of Grisvard formula} is non-negative and finally, the convexity of $\Omega$ implies that the fifth integral is non-positive. This results in the condition $-1 < \mu < \sqrt{2} - 1$, which is the imposed condition on $\mu$. Thus, the left-hand side is of the form $z + \alpha$ for some $z \in \overline{\S_{\theta}}$ and $\alpha \geq 0$. Consequently, by~\eqref{Eq: Inverse triangle inequality} there exists a constant $C_{\theta} > 0$ depending only on $\theta$, such that
\begin{align*}
 &\lvert \lambda \rvert \int_{\Omega} \{\partial_k u_{\beta} + \mu \partial_{\beta} u_k \} \overline{\partial_k u_{\beta}} \; \d x + \int_{\Omega} \partial_i \partial_j u_{\beta} \overline{\partial_i \partial_j u_{\beta}} \; \d x + (2 \mu + \mu^2) \int_{\Omega} \Re(\partial_j \partial_i u_{\beta} \overline{\partial_i \partial_{\beta} u_j}) \; \d x \\
 &\qquad + (1 - 2 \mu) \int_{\Omega} \partial_{\beta} \phi \overline{\partial_{\beta} \phi} \; \d x - \int_{\partial \Omega} \II ((v_{\beta})_{\T} ; (v_{\beta})_{\T}) \; \d \sigma \\
 &\leq C_{\theta} \bigg( \int_{\Omega} (\lvert \Delta u \rvert + (1 + 2 \lvert \mu \rvert) \lvert \nabla \phi \rvert) \lvert f \rvert \; \d x + 2 \lvert \lambda \rvert \lvert \mu \rvert \int_{\Omega} \lvert u \rvert \lvert \nabla \phi \rvert \; \d x \bigg).
\end{align*}
By virtue of~\eqref{Eq: Mixed derivative},~\eqref{Eq: Cross product asymmetric term}, and the convexity of $\Omega$ one finds
\begin{align*}
 &\lvert \lambda \rvert (1 - \lvert \mu \rvert) \int_{\Omega} \lvert \nabla u \rvert^2 \; \d x + (1 - \lvert 2 \mu + \mu^2 \rvert) \int_{\Omega} \lvert \nabla^2 u \rvert^2 \; \d x + (1 - 2 \mu) \int_{\Omega} \lvert \nabla \phi \rvert^2 \; \d x \\
 &\leq C_{\theta} \bigg( \int_{\Omega} (\lvert \Delta u \rvert + (1 + 2 \lvert \mu \rvert) \lvert \nabla \phi \rvert) \lvert f \rvert \; \d x + 2 \lvert \lambda \rvert \lvert \mu \rvert \int_{\Omega} \lvert u \rvert \lvert \nabla \phi \rvert \; \d x \bigg).
\end{align*}
The desired inequality now follows for $f \in \C_c^{\infty} (\Omega ; \IC^d)$ by an application of Young's inequality and for $f \in \L^2 (\Omega ; \IC^d)$ by density. \par
To conclude the proof, we approximate an arbitrary bounded and convex domain $\Omega$ by smooth, bounded, and convex domains $\Omega_k$ as described in Remark~\ref{Rem: Approximating sequence}. Let $R_{\Omega_k}$ denote the restriction operator to $\Omega_k$, $\IQ_k$ be the Helmholtz projection on $\Omega_k$, and $B_{\mu , k}$ the Stokes operator subject to Neumann-type boundary conditions on $\Omega_k$. Define $f_k := R_{\Omega_k} f \in \L^2 (\Omega_k ; \IC^d)$, $u_k := (\lambda + B_{\mu , k})^{-1} \IQ_k f_k$, and define $u := (\lambda + B_{\mu})^{-1} \IQ f$. Then
\begin{align*}
 &\lambda \int_{\Omega_k} (u - u_k) \cdot \overline{(u - u_k)} \; \d x + \int_{\Omega_k} a^{\alpha \beta}_{j l} (\mu) \partial_l (u_{\beta} - (u_k)_{\beta}) \overline{\partial_j (u_{\alpha} - (u_k)_{\alpha})} \; \d x \\
 &= \lambda \int_{\Omega_k} u \cdot \overline{u} \; \d x + (\lambda - \overline{\lambda}) \int_{\Omega_k} u_k \cdot \overline{u_k} \; \d x + \overline{\lambda} \int_{\Omega_k} u_k \cdot \overline{u_k} \; \d x - \lambda \int_{\Omega_k} u_k \cdot \overline{u} \; \d x - \lambda \int_{\Omega_k} u \cdot \overline{u_k} \; \d x \\
 &\qquad + \int_{\Omega_k} a^{\alpha \beta}_{j l} (\mu) \partial_l u_{\beta} \overline{\partial_j u_{\alpha}} \; \d x + \int_{\Omega_k} a^{\alpha \beta}_{j l} (\mu) \partial_l (u_k)_{\beta} \overline{\partial_j (u_k)_{\alpha}} \; \d x \\
 &\qquad - \int_{\Omega_k} a^{\alpha \beta}_{j l} (\mu) \partial_l u_{\beta} \overline{\partial_j (u_k)_{\alpha}} \; \d x - \int_{\Omega_k} a^{\alpha \beta}_{j l} (\mu) \partial_l (u_k)_{\beta} \overline{\partial_j u_{\alpha}} \; \d x \\
 &= \int_{\Omega \setminus \Omega_k} f \cdot \overline{u} \; \d x - \bigg( \lambda \int_{\Omega \setminus \Omega_k} u \cdot \overline{u} \; \d x + \int_{\Omega \setminus \Omega_k} a^{\alpha \beta}_{j l} (\mu) \partial_l u_{\beta} \overline{\partial_j u_{\alpha}} \; \d x \bigg) \\
 &\qquad + (\lambda - \overline{\lambda}) \int_{\Omega_k} (u - u_k) \cdot \overline{(u - u_k)} \; \d x - (\lambda - \overline{\lambda}) \int_{\Omega_k} (u - u_k) \cdot \overline{u} \; \d x - \int_{\Omega_k} (u - u_k) \cdot \overline{f} \; \d x.
\end{align*}
Rearranging terms yields
\begin{align}
\label{Eq: Weak to local}
\begin{aligned}
\overline{\lambda} \int_{\Omega_k} \lvert u - u_k \rvert^2 \; \d x &+ \int_{\Omega_k} a^{\alpha \beta}_{j l} (\mu) \partial_l (u_{\beta} - (u_k)_{\beta}) \overline{\partial_j (u_{\alpha} - (u_k)_{\alpha})} \; \d x \\
 &= \int_{\Omega \setminus \Omega_k} f \cdot \overline{u} \; \d x - \bigg( \lambda \int_{\Omega \setminus \Omega_k} u \cdot \overline{u} \; \d x + \int_{\Omega \setminus \Omega_k} a^{\alpha \beta}_{j l} (\mu) \partial_l u_{\beta} \overline{\partial_j u_{\alpha}} \; \d x \bigg) \\
 &\qquad - (\lambda - \overline{\lambda}) \int_{\Omega_k} (u - u_k) \cdot \overline{u} \; \d x - \int_{\Omega_k} (u - u_k) \cdot \overline{f} \; \d x.
\end{aligned}
\end{align}
Since $u \in \H^1 (\Omega ; \IC^d)$ and $f \in \L^2 (\Omega ; \IC^d)$, we find by~\eqref{Eq: Ellipticity Neumann} and~\eqref{Eq: Inverse triangle inequality}, that $(u - u_k)_{k \in \IN}$ defines a bounded sequence in $\L^2 (\Omega ; \IC^d)$ and $(\nabla u - \nabla u_k)_{k \in \IN}$ defines a bounded sequence in $\L^2 (\Omega ; \IC^{d^2})$. Here, we regard $u - u_k$ and $\nabla u - \nabla u_k$ to be zero on $\Omega \setminus \Omega_k$. Thus, there exist subsequences (again denoted by the same indices) and weak limits $v \in \L^2 (\Omega ; \IC^d)$ and $w \in \L^2 (\Omega ; \IC^{d^2})$, such that $u - u_k \rightharpoonup v$ and $\nabla u - \nabla u_k \rightharpoonup w$ as $k \to \infty$. One directly verifies that $v$ is weakly differentiable with $\nabla v = w$ and that the distributional divergence of $v$ is zero. It follows that $v \in \cH^1_{\sigma} (\Omega)$. Now, for $\varphi \in \cH^1_{\sigma} (\Omega)$ one finds, since $u$ and $u_k$ solve their respective equations, that
\begin{align*}
 \lambda \int_{\Omega} v \cdot \overline{\varphi} \; \d x &+ \int_{\Omega} a_{j l}^{\alpha \beta} (\mu) \partial_l v_{\beta} \cdot \overline{\partial_j \varphi_{\alpha}} \; \d x \\
 &= \lambda \lim_{k \to \infty} \int_{\Omega_k} (u - u_k) \cdot \overline{\varphi} \; \d x + \lim_{k \to \infty} \int_{\Omega_k} a_{j l}^{\alpha \beta} (\mu) \partial_l (u_{\beta} - (u_k)_{\beta}) \cdot \overline{\partial_j \varphi_{\alpha}} \; \d x \\
 &= 0.
\end{align*}
In follows that $v$ is zero. Going back to~\eqref{Eq: Weak to local}, one even finds that $u_k \to u$ in $\H^1_{\mathrm{loc}} (\Omega ; \IC^d)$. Since due to the first part of the proof, also the sequence $(\nabla^2 u_k)_{k \in \IN}$ is bounded in $\L^2 (\Omega ; \IC^{d^3})$, where $\nabla^2 u_k$ is regarded to be zero in $\Omega \setminus \Omega_k$, we find again by picking a weakly convergent subsequence that $u$ is in $\H^2 (\Omega ; \IC^d)$ and that
\begin{align}
\label{Eq: Fatou}
 \| \nabla^2 u \|_{\L^2 (\Omega ; \IC^{d^3})} \leq \liminf_{k \to \infty} \| \nabla^2 u_k \|_{\L^2 (\Omega_k ; \IC^{d^3})}.
\end{align}
If $\phi_k$ denotes the pressure such that $\lambda u_k - \Delta u_k + \nabla \phi_k = f_k$ holds in $\Omega_k$ (and satisfies the appropriate boundary condition), then we find by virtue of~\eqref{Eq: Variational Neumann} with $\varphi_k := \nabla \Delta_D^{-1}\chi_{\Omega_k} (\phi - \phi_k)$ where $\Delta_D$ denotes the Dirichlet Laplacian on $\Omega$ that
\begin{align*}
 &\int_{\Omega_k} \lvert \phi - \phi_k \rvert^2 \; \d x \\
 &= \int_{\Omega_k} (\phi - \phi_k) \, \overline{\divergence \varphi_k} \; \d x \\
 &= - \int_{\Omega \setminus \Omega_k} \phi \, \overline{\divergence \varphi_k} \; \d x + \lambda \int_{\Omega} u \cdot \overline{\varphi_k} \; \d x + \int_{\Omega} a^{\alpha \beta}_{j l} (\mu) \partial_l u_{\beta} \overline{\partial_j (\varphi_k)_{\alpha}} \; \d x \\
 &\qquad - \int_{\Omega} f \cdot \overline{\varphi_k} \; \d x - \bigg( \lambda \int_{\Omega_k} u_k \cdot \overline{\varphi_k} \; \d x + \int_{\Omega_k} a^{\alpha \beta}_{j l} (\mu) \partial_l (u_k)_{\beta} \overline{\partial_j (\varphi_k)_{\alpha}} \; \d x \bigg) + \int_{\Omega_k} f \cdot \overline{\varphi_k} \; \d x \\
 &= - \int_{\Omega \setminus \Omega_k} \phi \, \overline{\divergence \varphi_k} \; \d x + \lambda \int_{\Omega \setminus \Omega_k} u \cdot \overline{\varphi_k} \; \d x + \int_{\Omega \setminus \Omega_k} a^{\alpha \beta}_{j l} (\mu) \partial_l u_{\beta} \overline{\partial_j (\varphi_k)_{\alpha}} \; \d x \\
 &\qquad - \int_{\Omega \setminus \Omega_k} f \cdot \overline{\varphi_k} \; \d x + \lambda \int_{\Omega_k} (u - u_k) \cdot \overline{\varphi_k} \; \d x + \int_{\Omega_k} a^{\alpha \beta}_{j l} (\mu) \partial_l (u_{\beta} - (u_k)_{\beta}) \overline{\partial_j (\varphi_k)_{\alpha}} \; \d x.
\end{align*}
Since $\Omega$ is convex, it holds $\| \nabla \varphi_k \|_{\L^2 (\Omega ; \IC^{d^2})} \leq \| \phi - \phi_k \|_{\L^2 (\Omega_k)}$. This implies by Poincar\'e's inequality and $\Omega_k \subset \Omega$ that $\| \varphi_k \|_{\L^2 (\Omega ; \IC^d)} \leq C \diam(\Omega) \| \phi - \phi_k \|_{\L^2 (\Omega_k)}$, where $C > 0$ depends only on $d$. Thus, by virtue of Young's inequality, one can absorb $\| \phi - \phi_k \|_{\L^2 (\Omega_k)}$ to the left-hand side of the inequality above so that the convergences proven above together with the facts that $\phi$, $u$, and $f$ are $\L^2$-integrable on $\Omega$ yield that $\phi - \phi_k \to 0$ as $k \to \infty$ in $\L^2 (\Omega)$, where $\phi - \phi_k$ is defined to be zero in $\Omega \setminus \Omega_k$. Finally, since each $\phi_k$ lies in $\H^1 (\Omega_k)$ and respects the estimate from the formulation of the theorem, we find that $\phi \in \H^1 (\Omega)$ and that
\begin{align*}
 \| \nabla \phi \|_{\L^2 (\Omega)} \leq \liminf_{k \to \infty} \| \nabla \phi_k \|_{k \to \infty}.
\end{align*}
This proves the desired estimate for $u$ and $\phi$.
\end{proof}

\begin{remark}
For a similar approximation scheme in the case of no-slip boundary conditions see~\cite{Monniaux}.
\end{remark}

Notice that the sectoriality of $B_{\mu}$ (by Proposition~\ref{Prop: Resolvent}) implies the validity of the algebraic and topological decomposition $\cL^2_{\sigma} (\Omega)$
\begin{align*}
 \cL^2_{\sigma} (\Omega) = \ker(B_{\mu}) \oplus \overline{\Rg(B_{\mu})},
\end{align*}
where $\ker(B_{\mu})$ denotes the kernel of $B_{\mu}$ and $\Rg(B_{\mu})$ the range of $B_{\mu}$. See~\cite[Prop.~2.2.1]{Haase} for the corresponding statement on the decomposition.

\begin{corollary}
Let $\Omega \subset \IR^d$, $d \geq 2$, be a bounded convex domain and $\mu \in (-1 , \sqrt{2} - 1)$. Then for all $u \in \dom(B_{\mu}) \cap \overline{\Rg(B_{\mu})}$ and the associated pressure $\phi$ one finds that $u \in \H^2(\Omega ; \IC^d)$ and $\phi \in \H^1(\Omega)$. Moreover, there exists $C > 0$ depending only on $d$ and $\mu$ such that
\begin{align*}
 \int_{\Omega} \lvert \nabla^2 u \rvert^2 \; \d x + \int_{\Omega} \lvert \nabla \phi \rvert^2 \; \d x \leq C \int_{\Omega} \lvert B_{\mu} u \rvert^2 \; \d x.
\end{align*}
\end{corollary}

\begin{proof}
First of all, notice that the statement below concerning the strong convergence of resolvent operators follow from~\cite[Prop.~2.2.1]{Haase}. Define $f := B_{\mu} u$. The solution $u$ is approximated by $u_{\lambda} := (\lambda + B_{\mu})^{-1} f$ as $\lambda \in \S_{\pi / 2}$ tends to zero. Indeed, since $f = B_{\mu} u$ and since $u \in \overline{\Rg (B_{\mu})}$ by assumption, one has due to the sectoriality of $B_{\mu}$, see Proposition~\ref{Prop: Resolvent}, that
\begin{align*}
 u_{\lambda} = B_{\mu} (\lambda + B_{\mu})^{-1} u \to u \quad \text{in} \quad \cL^2_{\sigma} (\Omega) \quad \text{as} \quad \lambda \to 0.
\end{align*}
Furthermore, the sectoriality implies that $B_{\mu} u_{\lambda}$ tends to $f$ in $\cL^2_{\sigma} (\Omega)$ and as well that $\lambda u_{\lambda} \to 0$ tends to zero in $\cL^2_{\sigma} (\Omega)$ as $\lambda \in \S_{\pi / 2}$. The convergence of the associated pressures $\phi_{\lambda}$ in $\L^2 (\Omega)$ is proven as before by invoking Bogovski\u{\i}'s operator. Finally, the convergence in the $\H^2 (\Omega ; \IC^d)$- and $\H^1 (\Omega)$-norms of the respective sequences follows by employing the inequality proven in Theorem~\ref{Thm: H2 regularity on convex domains} and the fact that the ``right-hand side'' $B_{\mu} u_{\lambda}$ of the equations for $u_{\lambda}$ and $\phi_{\lambda}$ tend to $f$ in $\cL^2_{\sigma} (\Omega)$. The desired inequality follows from Theorem~\ref{Thm: H2 regularity on convex domains} by taking limits.
\end{proof}

\begin{problem}
Prove or disprove Theorem~\ref{Thm: H2 regularity on convex domains} for $\mu \in [\sqrt{2} - 1 , 1]$.
\end{problem}

In the case of no-slip boundary conditions, the $\H^2$-regularity is known in two and three dimensions if convex polygonal/polyhedral domains are considered, see~\cite{Kellogg_Osborn, Dauge, Mazya_Rossmann}. It would be interesting to know if this property holds on arbitrary convex domains.

\begin{problem}
Prove or disprove Theorem~\ref{Thm: H2 regularity on convex domains} in the case of no-slip boundary conditions.
\end{problem}

In the following, we start by working with cubes in $\IR^d$. By this we mean a non-degenerate cube of the form $(a , b)^d$, i.e., its Lebesgue measure is non-zero and its sides are parallel to the axes. Sometimes we will use the notation $Q (x_0 , r)$ to denote a cube with center $x_0$ and diameter $r$. We continue by deriving local $\H^2$-estimates and start with a technical lemma.

\begin{lemma}
\label{Lem: Piecewise regular}
Let $\Omega \subset \IR^d$ be a bounded convex domain with $\C^2$-boundary and let $Q$ be a cube. Then $Q \cap \Omega$ is piecewise $\C^2$-regular, i.e, there exist sets $\Gamma_0$ and $\Gamma_1$ such that $\partial [Q \cap \Omega] = \Gamma_0 \cup \Gamma_1$, where $\Gamma_0$ has surface measure zero and where for each $x \in \Gamma_1$ the boundary part of $Q \cap \Omega$ is $\C^2$-regular in a neighborhood of $x$. In particular, $\Gamma_1$ satisfies $\Gamma_1 \cap \Omega \subset \partial Q \cap \Omega$.
\end{lemma}

\begin{proof}
First of all, notice that $Q \cap \Omega$ is a bounded convex domain and thus in particular a bounded Lipschitz domain, see~\cite[Cor.~1.2.2.3]{Grisvard}. Notice that due to the Lipschitz boundary of $Q \cap \Omega$ the surface measure is equivalent to the $(d - 1)$-dimensional Hausdorff measure in $\IR^d$. To decompose the boundary of $Q \cap \Omega$, notice that elementary set theoretic manipulations yield
\begin{align*}
 \partial (Q \cap \Omega) \subset ((\partial Q) \cap \overline{\Omega}) \cup (\overline{Q} \cap (\partial \Omega)) &=((\partial Q) \cap \overline{\Omega}) \cup (Q \cap (\partial \Omega)).
\end{align*}
Notice that any point in $Q \cap (\partial \Omega)$ has a neighborhood with an at least $\C^2$-regular boundary. Thus, we consider $(\partial Q) \cap \overline{\Omega}$ more closely. \par
Let $\cN \subset \partial Q$ denote the edges of the cube $Q$. Clearly, its $(d - 1)$-dimensional Hausdorff measure is zero. Let $F \subset \partial Q$ be a face of $Q$ (we consider $F$ to be closed). Since $F$ and $\overline{\Omega}$ are convex, also $F \cap \overline{\Omega}$ is convex. Notice that $F \cap \overline{\Omega}$ is congruent to a convex set in $\IR^{d - 1}$. As convex sets are Lipschitz regular, the boundary of $F \cap \overline{\Omega}$ (with respect to the subspace topology of $F$) has zero $(d - 1)$-dimensional Hausdorff measure. If $x$ is in the interior of $F \cap \overline{\Omega}$ (with respect to the subspace topology of $F$) and if $x \notin \cN$, then there is $\eps > 0$ such that $F \cap \overline{\Omega} \cap B(x , \eps) = F \cap B(x , \eps)$. Thus, in this neighborhood, $F \cap \overline{\Omega}$ can be represented as the graph of a smooth function. Denote the boundary of $F \cap \overline{\Omega}$ taken with respect to the subspace topology by $\partial_F (F \cap \overline{\Omega})$ and the interior by $\mathrm{int}_F (F \cap \overline{\Omega})$ and define
\begin{align*}
 \Gamma_0 := \Big(\cN \cup \bigcup_{F \text{ face of } Q} \partial_F (F \cap \overline{\Omega}) \Big) \cap \partial (Q \cap \Omega)
\end{align*}
and
\begin{align*}
 \Gamma_1 := \bigg\{\bigcup_{F \text{ face of } Q} (\mathrm{int}_F (F \cap \overline{\Omega}) \setminus \cN) \cup (Q \cap (\partial \Omega)) \bigg\} \cap \partial (Q \cap \Omega).
\end{align*}
Notice that $\Gamma_1 \cap \Omega \subset \partial Q \cap \Omega$ holds by construction.
\end{proof}

\begin{lemma}
\label{Lem: Localized regularity estiamtes}
Let $\Omega \subset \IR^d$, $d \geq 2$, be a bounded, convex, and smooth domain, $\mu \in (-1 , \sqrt{2} - 1)$, and $\theta \in (0 , \pi)$. Then there exists $C > 0$ depending only on $d$, $\mu$, and $\theta$ such that smooth functions (smooth up to the boundary) $u : Q \cap \Omega \to \IC^d$ and $\phi : Q \cap \Omega \to \IC$ solving $\lambda u - \Delta u + \nabla \phi = 0$ and $\divergence(u) = 0$ in $Q \cap \Omega$ and which satisfy $\{ D u + \mu [D u]^{\top} \} n - \phi n = 0$ on $Q \cap \partial \Omega$ satisfy
\begin{align*}
 &\lvert \lambda \rvert \int_{Q \cap \Omega} \lvert \nabla u \rvert^2 \; \d x + \int_{Q \cap \Omega} \lvert \nabla^2 u \rvert^2 \; \d x + \int_{Q \cap \Omega} \lvert \nabla \phi \rvert^2 \; \d x \\
 &\qquad\leq C \bigg( \lvert \lambda \rvert^2 \int_{Q \cap \Omega} \lvert u \rvert^2 \; \d x + \int_{(\partial Q) \cap \Omega}\big(  \lvert \nabla^2 u \rvert \lvert \nabla u \rvert + \lvert \nabla^2 u \rvert \lvert \phi \rvert + \lvert \nabla \phi \rvert \lvert \nabla u \rvert + \lvert \nabla \phi \rvert \lvert \phi \rvert \big) \; \d \sigma \bigg).
\end{align*}
\end{lemma}

\begin{proof}
By Lemma~\ref{Lem: Piecewise regular}, $Q \cap \Omega$ is piecewise $\C^2$-regular with corresponding set $\Gamma_1$ satisfying $\Gamma_1 \cap \Omega \subset (\partial Q) \cap \Omega$. Thus, we are in the situation to apply Theorem~\ref{Thm: Piecewise Grisvard} on the underlying domain $Q \cap \Omega$ and $v := v_{\beta}$ defined by~\eqref{Eq: Grisvard testfunction}. The same calculation as in the first part of the proof of Theorem~\ref{Thm: H2 regularity on convex domains} (but with an application of Theorem~\ref{Thm: Piecewise Grisvard} instead of Theorem~\ref{Thm: Grisvard}) yields the existence of a constant $C > 0$ depending only on $d$, $\mu$, and $\theta$ such that
\begin{align*}
 \lvert \lambda \rvert \int_{Q \cap \Omega} \lvert \nabla u \rvert^2 \; \d x &+ \int_{Q \cap \Omega} \lvert \nabla^2 u \rvert^2 \; \d x + \int_{Q \cap \Omega} \lvert \nabla \phi \rvert^2 \; \d x \\
 &\qquad\leq C \bigg(\lvert \lambda \rvert^2 \int_{Q \cap \Omega} \lvert u \rvert^2 \; \d x + \int_{(\partial Q) \cap \Omega} \lvert \nabla v_{\beta} \rvert \lvert v_{\beta} \rvert \; \d \sigma \bigg).
\end{align*}
By definition of $v_{\beta}$ this readily concludes the proof.
\end{proof}

In the previous proposition we saw that a local $\H^2$-estimate can be achieved with the drawback that highest-order terms appear in boundary integrals on the right-hand side of the inequality. The following lemma (the so-called $\eps$-lemma) will help us to absorb these terms to the left-hand side and can be found in~\cite[Lem.~0.5]{Giaquinta_Modica}. Notice that the notation of cubes $Q(x_0 , r)$ used here differs from the one used in~\cite{Giaquinta_Modica}, so that our formulation is slightly different.

\begin{lemma}
\label{Lem: Eps-lemma}
Let $f$, $g$, and $h$ be non-negative functions in $\L^1 (\cQ)$, where $\cQ$ is a cube in $\IR^d$ and let $\alpha > 0$. There exists $\eps_0 > 0$, depending only on $d$ and $\alpha$, such that if for some $0 \leq \eps \leq \eps_0$ and some $C_1 = C_1 (\eps) > 0$ the estimate
\begin{align*}
\int_{Q(x_0 , r)} f \; \d x \leq C_1 \bigg\{ \frac{1}{r^{\alpha}} \int_{Q(x_0 , 2 r)} g \; \d x + \int_{Q (x_0 , 2 r)} h \; \d x \bigg\} + \eps \int_{Q (x_0 , 2 r)} f \; \d x
\end{align*}
holds for all $x_0 \in \cQ$ and $0 < r < \sqrt{d} \dist(x_0 , \partial \cQ)$, then there exists a constant $C > 0$, depending only on $d$, $\alpha$, and $C_1$, such that
\begin{align*}
 \int_{Q (x_0 , r)} f \; \d x \leq C \bigg\{ \frac{1}{r^{\alpha}} \int_{Q (x_0 , 2 r)} g \; \d x + \int_{Q (x_0 , 2 r)} h \; \d x \bigg\}.
\end{align*}
\end{lemma}

The following proposition finally provides us with a local higher-order estimate.

\begin{proposition}
\label{Prop: Reverse regularity estimates}
Let $\Omega \subset \IR^d$, $d \geq 2$, be a bounded, convex, and smooth domain, $\theta \in (0 , \pi)$, and $\mu \in (-1 , \sqrt{2} - 1)$ and let $\cQ$ be a cube with $\cQ \cap \Omega \neq \emptyset$ and diameter $R > 0$. Then there exists $C > 0$ depending only on $d$, $\mu$, and $\theta$ such that smooth functions (smooth up to the boundary) $u : (2 \cQ) \cap \Omega \to \IC^d$ and $\phi : (2 \cQ) \cap \Omega \to \IC$ solving $\lambda u - \Delta u + \nabla \phi = 0$ and $\divergence(u) = 0$ in $(2 \cQ) \cap \Omega$ and which satisfy $\{ D u + \mu [D u]^{\top} \} n - \phi n = 0$ on $(2 \cQ) \cap \partial \Omega$ satisfy
\begin{align*}
 &\lvert \lambda \rvert \int_{\cQ \cap \Omega} \lvert \nabla u \rvert^2 \; \d x + \int_{\cQ \cap \Omega} \lvert \nabla^2 u \rvert^2 \; \d x + \int_{\cQ \cap \Omega} \lvert \nabla \phi \rvert^2 \; \d x \\
 &\qquad \leq C \bigg(\lvert \lambda \rvert^2 \int_{(2 \cQ) \cap \Omega} \lvert u \rvert^2 \; \d x + \frac{1}{R^2} \int_{(2 \cQ) \cap \Omega} \big( \lvert \nabla u \rvert^2 + \lvert \phi \rvert^2 \big) \; \d x \bigg).
\end{align*}
\end{proposition}

\begin{proof}
Fix a cube $\cQ \subset \IR^d$ with $\cQ \cap \Omega \neq \emptyset$. Let $Q := Q(x_0 , r) \subset \IR^d$ be a cube with center $x_0 \in \cQ$ and $\diam(Q) = r$ that satisfies $0 < r < \sqrt{d} \dist(x_0 , \partial \cQ)$. Let $1 < s < 2$. By Lemma~\ref{Lem: Localized regularity estiamtes} one finds
\begin{align*}
 &\lvert \lambda \rvert \int_{Q \cap \Omega} \lvert \nabla u \rvert^2 \; \d x + \int_{Q \cap \Omega} \lvert \nabla^2 u \rvert^2 \; \d x + \int_{Q \cap \Omega} \lvert \nabla \phi \rvert^2 \; \d x \\
 &\leq \lvert \lambda \rvert \int_{(s Q) \cap \Omega} \lvert \nabla u \rvert^2 \; \d x + \int_{(s Q) \cap \Omega} \lvert \nabla^2 u \rvert^2 \; \d x + \int_{s Q \cap \Omega} \lvert \nabla \phi \rvert^2 \; \d x \\
 &\leq C \bigg(\lvert \lambda \rvert^2 \int_{(2 Q) \cap \Omega} \lvert u \rvert^2 \; \d x + \int_{(\partial s Q) \cap \Omega} \big( \lvert \nabla^2 u \rvert \lvert \nabla u \rvert + \lvert \nabla^2 u \rvert \lvert \phi \rvert + \lvert \nabla \phi \rvert \lvert \nabla u \rvert + \lvert \nabla \phi \rvert \lvert \phi \rvert  \big) \; \d \sigma \bigg),
\end{align*}
where the constant $C > 0$ depends only on $d$, $\mu$, and $\theta$. An application of Young's inequality (this produces the factors $r \eps$ and $(r \eps)^{-1}$ for some $\eps > 0$) followed by an integration over $s$ yields
\begin{align*}
 &\lvert \lambda \rvert \int_{Q \cap \Omega} \lvert \nabla u \rvert^2 \; \d x + \int_{Q \cap \Omega} \lvert \nabla^2 u \rvert^2 \; \d x + \int_{Q \cap \Omega} \lvert \nabla \phi \rvert^2 \; \d x \\
 &\qquad\leq C \lvert \lambda \rvert^2 \int_{(2 Q) \cap \Omega} \lvert u \rvert^2 \; \d x + \frac{C}{r \eps} \int_1^2 \int_{(\partial s Q) \cap \Omega} \big(\lvert \nabla u \rvert^2 + \lvert \phi \rvert^2\big) \; \d \sigma \; \d s \\
 &\qquad\qquad + r \eps \int_1^2 \int_{(\partial s Q) \cap \Omega} \big( \lvert \nabla^2 u \rvert^2 + \lvert \nabla \phi \rvert^2 \big) \; \d \sigma \; \d s.
\end{align*}
Now, notice that the co-area formula, see~\cite[Thm.~3.2.12]{Federer}, implies that
\begin{align*}
 \int_1^2 \int_{\partial s Q} \fg \; \d \sigma \; \d s \leq \frac{C_{\text{co-area}}}{r} \int_{2 Q} \fg \; \d x
\end{align*}
for all representatives $\fg$ of a function $g \in \L^1 (\IR^d)$, where $C_{\text{co-area}} > 0$ is an absolute constant. Choosing $g$ in the first integral as $E_0 (\lvert \nabla u \rvert^2 + \lvert \phi \rvert^2)$ and in the second integral as $E_0 (\lvert \nabla^2 u \rvert^2 + \lvert \nabla \phi \rvert^2)$, where $E_0$ extends functions outside of $2 Q \cap \Omega$ by zero delivers
\begin{align*}
 &\lvert \lambda \rvert \int_{Q \cap \Omega} \lvert \nabla u \rvert^2 \; \d x + \int_{Q \cap \Omega} \lvert \nabla^2 u \rvert^2 \; \d x + \int_{Q \cap \Omega} \lvert \nabla \phi \rvert^2 \; \d x \\
 &\qquad\leq C \lvert \lambda \rvert^2 \int_{(2 Q) \cap \Omega} \lvert u \rvert^2 \; \d x + \frac{C C_{\text{co-area}}}{r^2 \eps} \int_{(2 Q) \cap \Omega} \big(\lvert \nabla u \rvert^2 + \lvert \phi \rvert^2\big) \; \d x \\
 &\qquad\qquad + \eps C_{\text{co-area}} \int_{(2 Q) \cap \Omega} \big( \lvert \nabla^2 u \rvert^2 + \lvert \nabla \phi \rvert^2 \big) \; \d x.
\end{align*}
The proof is concluded for $\eps$ small enough by an application of Lemma~\ref{Lem: Eps-lemma}.
\end{proof}

\section{An $\L^p$-extrapolation theorem suitable for subspaces of $\L^p$}
\label{An Lp-extrapolation theorem suitable for subspaces of Lp}

\noindent In classical Calder\'on--Zygmund theory, operators $T$ associated to an integral kernel $K ( \cdot , \cdot )$ give rise to an $\L^p$-bounded operator for all $1 < p < \infty$ if $T$ is bounded on $\L^2$ and if the kernel $K$ is a so-called standard kernel. The standard kernel property is some kind of cancellation property of $K$, see~\cite[Def.~5.11]{Duoandikoetxea}. If the operator $T$ is not associated to a kernel or if one is only interested in whether $T$ is bounded on $\L^p$ but for $p$ being merely in an interval $I \subset (1 , \infty)$, then one can replace the property that $T$ is associated to a standard kernel by weaker cancellation properties. \par
In this context, there are for example the $\L^p$-extrapolation theorems of Shen~\cite{Shen-Riesz_transform} (if one is interested to conclude the $\L^p$-boundedness on an interval $(2 , q)$ with $q > 2$) or of Blunck and Kunstmann~\cite{Blunck_Kunstmann} (if one is interested to conclude the $\L^p$-boundedness on an interval $(q , 2)$ with $q < 2$). In the following, we will consider Shen's theorem more closely and begin with a formulation of his theorem which can be found in~\cite{Tolksdorf_higher_order, Tolksdorf_Dissertation}.

\begin{theorem}
\label{Thm: Shen}
Let $\Omega \subset \IR^d$ be Lebesgue-measurable, $\cM > 0$, and let $T \in \Lop(\L^2(\Omega ; \IC^m) , \L^2(\Omega ; \IC^n))$ with $\| T \|_{\Lop(\L^2(\Omega ; \IC^m) , \L^2(\Omega ; \IC^n))} \leq \cM$. \par
Suppose that there exist constants $q > 2$, $R_0 > 0$, $\alpha_2 > \alpha_1 > 1$, and $\cC > 0$, where $R_0 = \infty$ if $\diam(\Omega) = \infty$, such that the following holds. Namely, for all $B = B(x_0 , r)$ with $0 < r < R_0$, whose center $x_0$ is either such that $x_0 \in \partial \Omega$ or $\alpha_2 B \subset \Omega$, and all compactly supported $f \in \L^{\infty}(\Omega ; \IC^m)$ with $f = 0$ on $\Omega \cap \alpha_2 B$ the estimate
\begin{align}
\label{Eq: Weak reverse Hoelder inequality}
 \begin{aligned}
 \bigg( \frac{1}{r^d} \int_{\Omega \cap B} \lvert T f \rvert^q \; \d x \bigg)^{\frac{1}{q}} &\leq \cC \bigg\{ \bigg( \frac{1}{r^d} \int_{\Omega \cap \alpha_1 B} \lvert T f \rvert^2 \; \d x \bigg)^{\frac{1}{2}} + \sup_{B^{\prime} \supset B} \bigg( \frac{1}{\lvert B^{\prime} \rvert} \int_{\Omega \cap B^{\prime}} \lvert f \rvert^2 \; \d x \bigg)^{\frac{1}{2}} \bigg\}
 \end{aligned}
\end{align}
holds. Here the supremum runs over all balls $B^{\prime}$ containing $B$. \par
Then for each $2 < p < q$ the restriction of $T$ onto $\L^2(\Omega ; \IC^m) \cap \L^p(\Omega ; \IC^m)$ extends to a bounded linear operator from $\L^p (\Omega ; \IC^m)$ into $\L^p(\Omega ; \IC^n)$, with operator norm bounded by a constant depending on $d$, $p$, $q$, $\alpha_1$, $\alpha_2$, $\cC$, and $\cM$, and additionally on $R_0$ and $\diam(\Omega)$ if $\Omega$ is bounded.
\end{theorem}

In this theorem, the standard kernel property is replaced by the validity of~\eqref{Eq: Weak reverse Hoelder inequality}. If $\Omega = \IR^d$, then the proof builds on a good-$\lambda$ argument. If $\Omega$ is not $\IR^d$, one can define an appropriate operator on the whole space given by $S f := E_0 T R_{\Omega} f$, where $E_0$ extends functions on $\Omega$ by zero and $R_{\Omega}$ restricts functions on the whole space to $\Omega$. One can show, that if $T$ satisfies the assumptions of Theorem~\ref{Thm: Shen} on $\Omega$, then $S$ satisfies the assumptions of the same theorem with $\Omega$ set to $\IR^d$, cf.~\cite[p.~78f]{Tolksdorf_Dissertation}. If $\Omega = \IR^d$, an analysis of the good-$\lambda$ argument reveals that~\eqref{Eq: Weak reverse Hoelder inequality} enters the game only once, namely, in order to deduce an inequality of the form
\begin{align}
\label{Eq: Good lambda key}
\begin{aligned}
 &\lvert \{ x \in Q : M_{2 Q^*} (\lvert T f \rvert^2) (x) > \iota \} \rvert \\
 &\quad\leq \frac{C}{\iota} \int_{2 \alpha_2 Q^*} \lvert f \rvert^2 \; \d x + \frac{C \lvert Q \rvert}{\iota^{q / 2}} \bigg\{ \bigg( \frac{1}{\lvert Q \rvert} \int_{2 \alpha_2 Q^*} \lvert T f \rvert^2 \; \d x \bigg)^{\frac{1}{2}} + \sup_{Q^{\prime} \supset 2 Q^*} \bigg( \frac{1}{\lvert Q^{\prime} \rvert} \int_{Q^{\prime}} \lvert f \rvert^2 \; \d x \bigg)^{\frac{1}{2}} \bigg\}^q,
\end{aligned}
\end{align}
cf.~\cite[p.~76f]{Tolksdorf_Dissertation}. Here, $\iota > 0$ is arbitrary, $Q$ is a cube in $\IR^d$, $Q^*$ is its ``parent'', i.e., $Q$ arises from $Q^*$ by bisecting its sides, and $M_{2 Q^*}$ is the localized maximal operator
\begin{align*}
 M_{2 Q^*} g (x) := \sup_{\substack{x \in R \\ R \subset 2 Q^*}}\frac{1}{\lvert R \rvert} \int_R \lvert g \rvert \; \d y \qquad (x \in 2 Q^*),
\end{align*}
where in the supremum $R$ denotes a cube in $\IR^d$. To derive~\eqref{Eq: Good lambda key} from~\eqref{Eq: Weak reverse Hoelder inequality} and the $\L^2$-boundedness of $T$, notice that~\eqref{Eq: Weak reverse Hoelder inequality} can equivalently be formulated with cubes instead of balls. Then, $f$ is decomposed as $f = f \chi_{2 \alpha_2 Q^*} + f \chi_{\IR^d \setminus 2 \alpha_2 Q^*}$, where $\chi$ denotes the characteristic function of a set. This decomposition is used on the left-hand side of~\eqref{Eq: Good lambda key} to estimate
\begin{align}
\label{Eq: Decomposition by sharp cutoff}
\begin{aligned}
 \lvert \{ x \in Q : M_{2 Q^*} (\lvert T f \rvert^2) (x) > \iota \} \rvert &\leq \lvert \{ x \in Q : M_{2 Q^*} (\lvert T f \chi_{2 \alpha_2 Q^*} \rvert^2) (x) > \iota / 4 \} \rvert \\
 &\qquad + \lvert \{ x \in Q : M_{2 Q^*} (\lvert T f \chi_{\IR^d \setminus 2 \alpha_2 Q^*} \rvert^2) (x) > \iota / 4 \} \rvert.
\end{aligned}
\end{align}
The first term on the right-hand side is controlled by the weak type-$(1 , 1)$ estimate of the localized maximal operator and the $\L^2$-boundedness of $T$, yielding the first term on the right-hand side of~\eqref{Eq: Good lambda key}. The second term on the right-hand side is controlled by the embedding $\L^{q / 2} \hookrightarrow \L^{q / 2 , \infty}$ and the $\L^{q / 2}$-boundedness of the localized maximal operator followed by~\eqref{Eq: Weak reverse Hoelder inequality} and the $\L^2$-boundedness of $T$ yielding the remaining terms on the right-hand side of~\eqref{Eq: Good lambda key}, cf.~\cite[p.~76f]{Tolksdorf_Dissertation}. \par
Essentially, the only thing that happened in~\eqref{Eq: Decomposition by sharp cutoff} was that $T f$ was decomposed by means of
\begin{align}
\label{Eq: Decomposition of Tf}
 T f =  T f \chi_{2 \alpha_2 Q^*} +  T f \chi_{\IR^d \setminus 2 \alpha_2 Q^*}.
\end{align}
We would like to stress here, that this decomposition of $T f$ was induced by the linearity of $T$ and a decomposition of $f$. Clearly, one could imagine that other suitable decompositions of $T f$ into a sum of two functions exist and that these might not have anything to do with a decomposition of $f$. Taking this into account in the formulation of the $\L^p$-extrapolation theorem might yield a more flexible result. This could be an advantage if a certain structure of $f$ (such as solenoidality) is eminent and which is destroyed by multiplication by characteristic functions. This happens for example if one considers the map $T : f \mapsto \phi$, where $f$ is mapped to the pressure function corresponding to the Stokes resolvent problem~\eqref{Res} and~\eqref{Neu}. If $f$ is for example divergence-free, then $T f$ enjoys the decay estimate presented in Proposition~\ref{Prop: Pressure estimate Neumann} while $T f \chi_{2 \alpha_2 Q^*}$ and $T f \chi_{\IR^d \setminus 2 \alpha_2 Q^*}$ enjoy no decay estimates at all by Remark~\ref{Rem: No pressure decay}. This indicates the need of a formulation of Shen's $\L^p$-extrapolation theorem that does not rely on a particular decomposition of $T f$ and is presented below. \par
In the rest of this section, we discuss an adapted version of Theorem~\ref{Thm: Shen}, where~\eqref{Eq: Weak reverse Hoelder inequality} is replaced essentially by the validity of~\eqref{Eq: Good lambda key} (which has to be modified if $\Omega \neq \IR^d$). To this end, we say that $Q^*$ is the parent of a cube $Q \subset \IR^d$ if $Q$ arises from $Q^*$ by bisecting its sides. Moreover, for $x_0 \in \IR^d$ and $r > 0$ let $Q(x_0 , r)$ denote the non-degenerated cube in $\IR^d$ with center $x_0$ and $\diam(Q (x_0 , r)) = r$. Finally, for a number $\alpha > 0$ denote by $\alpha Q$ the cube $Q (x_0 , \alpha r)$. The discussion above together with an analysis of the proof of~\cite[Thm.~3.1]{Shen-Riesz_transform} readily shows the validity of the following theorem.

\begin{theorem}
\label{Thm: Modified Shen whole space}
Let $2 < p < q$, $f \in \L^2 (\IR^d ; \IC^m) \cap \L^p (\IR^d ; \IC^m)$, and let $T$ be an operator (not necessarily linear) such that $T(f)$ is defined and contained in $\L^2 (\IR^d ; \IC^n)$. \par
Suppose that there exist constants $\alpha_2 > \alpha_1 > 1$ and $\cC > 0$ such that all $\iota > 0$, all $Q = Q(x_0 , r)$ with $r > 0$ and $x_0 \in \IR^d$, and all parents $Q^*$ of $Q$ the estimate
\begin{align}
\label{Eq: Generalized weak reverse Hoelder inequality IR^d}
 \begin{aligned}
 &\lvert \{ x \in Q : M_{2 Q^*} (\lvert T(f) \rvert^2) (x) > \iota \} \rvert \\
 &\leq \frac{\cC}{\iota} \int_{2 \alpha_2 Q^*} \lvert f \rvert^2 \; \d x + \frac{\cC \lvert Q \rvert}{\iota^{q / 2}} \bigg\{ \bigg( \frac{1}{\lvert Q \rvert} \int_{2 \alpha_2 Q^*} \lvert T(f) \rvert^2 \; \d x \bigg)^{\frac{1}{2}} + \sup_{Q^{\prime} \supset 2 Q^*} \bigg( \frac{1}{\lvert Q^{\prime} \rvert} \int_{Q^{\prime}} \lvert f \rvert^2 \; \d x \bigg)^{\frac{1}{2}} \bigg\}^q,
 \end{aligned}
\end{align}
holds. Here the supremum runs over all cubes $Q^{\prime}$ containing $2 Q^*$. \par
Then there exists a constant $C > 0$ depending on $d$, $p$, $q$, $\alpha_1$, $\alpha_2$, and $\cC$ such that
\begin{align*}
 \| T(f) \|_{\L^p (\IR^d ; \IC^n)} \leq C \| f \|_{\L^p (\IR^d ; \IC^m)}.
\end{align*}
\end{theorem}

Let $T$ be an operator acting on functions defined on $\Omega$ for some Lebesgue-measurable set $\Omega \subset \IR^d$ and let $f \in \L^2 (\Omega ; \IC^m) \cap \L^q (\Omega ; \IC^m)$. The following theorem is a direct consequence of Theorem~\ref{Thm: Modified Shen whole space} when applied to the operator $S := E_0 T R_{\Omega}$ and the function $E_0 f \in \L^2 (\IR^d ; \IC^m) \cap \L^q (\IR^d ; \IC^m)$.

\begin{theorem}
\label{Thm: Modified Shen domain}
Let $\Omega \subset \IR^d$ be Lebesgue-measurable,  $2 < p < q$, $f \in \L^2 (\Omega ; \IC^m) \cap \L^p (\Omega ; \IC^m)$, and let $T$ be an operator (not necessarily linear) such that $T(f)$ is defined and contained in $\L^2 (\IR^d ; \IC^n)$. \par
Suppose that there exist constants $\alpha_2 > \alpha_1 > 1$ and $\cC > 0$ such that for all $\iota > 0$, all $Q = Q(x_0 , r)$ with $r > 0$ and $x_0 \in \IR^d$, and all parents $Q^*$ of $Q$ with $(2 Q^*) \cap \Omega \neq \emptyset$ the estimate
\begin{align}
\label{Eq: Generalized weak reverse Hoelder inequality domain}
 \begin{aligned}
 &\lvert \{ x \in Q : M_{2 Q^*} (\lvert E_0 T (f) \rvert^2) (x) > \iota \} \rvert \leq \frac{\cC}{\iota} \int_{(2 \alpha_2 Q^*) \cap \Omega} \lvert f \rvert^2 \; \d x \\
 &\qquad + \frac{\cC \lvert Q \rvert}{\iota^{q / 2}} \bigg\{ \bigg( \frac{1}{\lvert Q \rvert} \int_{(2 \alpha_2 Q^*) \cap \Omega} \lvert T (f) \rvert^2 \; \d x \bigg)^{\frac{1}{2}} + \sup_{Q^{\prime} \supset 2 Q^*} \bigg( \frac{1}{\lvert Q^{\prime} \rvert} \int_{Q^{\prime} \cap \Omega} \lvert f \rvert^2 \; \d x \bigg)^{\frac{1}{2}} \bigg\}^q,
 \end{aligned}
\end{align}
holds. Here the supremum runs over all cubes $Q^{\prime}$ containing $2 Q^*$. \par
Then there exists a constant $C > 0$ depending on $d$, $p$, $q$, $\alpha_1$, $\alpha_2$, and $\cC$ such that
\begin{align*}
 \| T(f) \|_{\L^p (\Omega ; \IC^n)} \leq C \| f \|_{\L^p (\Omega ; \IC^m)}.
\end{align*}
\end{theorem}

\section{Estimates on the resolvent on convex domains}
\label{Sec: Resolvent estimates on convex domains}

\noindent In this section we verify the assumptions of Theorem~\ref{Thm: Modified Shen domain} for a particular choice of operators $T$. In the case of elliptic operators, a common way to do so is to establish the validity of a Caccioppoli type estimate, which we establish now for the Stokes resolvent problem, see also~\cite[Prop.~5.3.2]{Tolksdorf_Dissertation},~\cite[Lem.~3.8]{Choe_Kozono}, and~\cite[Thm.~1.1]{Giaquinta_Modica}.

\begin{lemma}
\label{Lem: Caccioppoli}
Let $\theta \in [0 , \pi)$, $\lambda \in \S_{\theta}$, $x_0 \in \IR^d$, $r > 0$, and $\mu \in [-1 , 1)$. Let $u \in \cH^1_{\sigma} (Q (x_0 , 2 r) \cap \Omega)$ and $\phi \in \L^2 (Q (x_0 , 2 r) \cap \Omega)$ solve
\begin{align*}
 \lambda \int_{Q (x_0 , 2 r) \cap \Omega} u \cdot \overline{\varphi} \; \d x + \int_{Q (x_0 , 2 r) \cap \Omega} \nabla u \cdot \overline{\nabla \varphi} \; \d x - \int_{Q (x_0 , 2 r) \cap \Omega} \phi \, \overline{\divergence(\varphi)} \; \d x = 0
\end{align*}
for all $\varphi \in \H^1 (Q (x_0 , 2 r) \cap \Omega ; \IC^d)$ with $\varphi = 0$ on $(\partial Q (x_0 , 2 r)) \cap \Omega$. Then there exists a constant $C > 0$ depending only on $\theta$ and $d$ such that
\begin{align*}
 &\lvert \lambda \rvert \int_{Q(x_0 , r) \cap \Omega} \lvert u \rvert^2 \; \d x + \int_{Q(x_0 , r) \cap \Omega} \lvert \nabla u \rvert^2 \; \d x \\
 &\qquad \leq \frac{C}{r^2} \bigg( \frac{1}{\lvert \lambda \rvert} \int_{Q(x_0 , 2 r) \cap \Omega} \lvert \phi \rvert^2 \; \d x + \int_{Q(x_0 , 2 r) \cap \Omega} \lvert u \rvert^2 \; \d x \bigg).
\end{align*}
\end{lemma}

\begin{proof}
The proof follows literally the lines of~\cite[Prop.~5.3.2]{Tolksdorf_Dissertation} (which proves this inequality in the case of homogeneous Dirichlet boundary conditions on $\partial \Omega$).
\end{proof}

Another ingredient that is needed in the verification of the assumptions of Theorem~\ref{Thm: Modified Shen domain} is Sobolev's inequality on convex domains. This is obtained by combining~\cite[Lem.~7.16]{Gilbarg_Trudinger} with either~\cite[Lem.~7.12]{Gilbarg_Trudinger} (in the case $\lvert 1 / p - 1 / q \rvert < 1 / d$) or~\cite[Thm.~3.1.4]{Adams_Hedberg} (in the case $\lvert 1 / p - 1 / q \rvert = 1 / d$).

\begin{proposition}
\label{Prop: Sobolev}
Let $\Xi \subset \IR^d$ be bounded and convex and $1 \leq p < q < \infty$ with $\lvert 1 / p - 1 / q \rvert \leq 1 / d$. Then there exists a constant $C > 0$ depending only on $d$, $p$, and $q$ such that for all $u \in \W^{1 , p} (\Xi)$
\begin{align*}
 \bigg( \int_{\Xi} \lvert u \rvert^q \; \d x \bigg)^{\frac{1}{q}} \leq \lvert \Xi \rvert^{\frac{1}{q} - \frac{1}{p}} \bigg( \int_{\Xi} \lvert u \rvert^p \; \d x \bigg)^{\frac{1}{p}} + C \lvert \Xi \rvert^{\frac{1}{d} - ( \frac{1}{p} - \frac{1}{q}) - 1} \diam(\Xi)^d \bigg( \int_{\Xi} \lvert \nabla u \rvert^p \; \d x \bigg)^{\frac{1}{p}}.
\end{align*}
\end{proposition}

Now, we are in the position to present the proof of Theorem~\ref{Thm: Main}.

\begin{proof}[Proof of Theorem~\ref{Thm: Main}]
We distinguish the cases $p = 2$, $2 < p < 2d / (d - 2)$, and $2d / (d + 2) < p < 2$. Notice that the case $p = 2$ readily follows by Propositions~\ref{Prop: Resolvent} and~\ref{Prop: Pressure estimate Neumann}.

\subsection*{Case 1: It holds $2 < p < 2d / (d - 2)$}

Let $\Omega$ be a bounded and convex domain and let $\Omega_k$ be a bounded, convex, and smooth domain introduced in Remark~\ref{Rem: Approximating sequence}. Let $f \in \C_{\sigma}^{\infty} (\overline{\Omega_k})$ and let $u$ be given by $u := (\lambda + B_{\mu , k})^{-1} f$ and let $\phi$ denote the associated pressure. Here $B_{\mu , k}$ denotes the Stokes operator subject to Neumann-type boundary conditions on $\Omega_k$. Notice that $u$ and $\phi$ are smooth up to the boundary by Remark~\ref{Rem: Higher regularity}. We show that
\begin{align*}
 T_{\lambda} f := \begin{pmatrix} \lvert \lambda \rvert u \\ \lvert \lambda \rvert^{1 / 2} \nabla u \\ \lvert \lambda \rvert^{1 / 2} \phi \end{pmatrix}
\end{align*}
is uniformly bounded with respect to $\lambda$ from $\cL^p_{\sigma} (\Omega_k)$ to $\L^p (\Omega_k ; \IC^{d + d^2 + 1}))$. To this end, we show in the following that $T_{\lambda} f$ satisfies~\eqref{Eq: Generalized weak reverse Hoelder inequality domain} with $q := 2d / (d - 2)$ in the case $d \geq 3$ and $q > 2$ arbitrary in the case $d = 2$. To obtain the uniform boundedness with respect to $\lambda$, we need to verify~\eqref{Eq: Generalized weak reverse Hoelder inequality domain} with involved constants independent of $\lambda$. Let $Q = Q(x_0 , r) \subset \IR^d$ be a cube with center $x_0$ and $\diam(Q) = r$ that satisfies $(2 Q^*) \cap \Omega \neq \emptyset$. Then, we consider the following three cases.

\subsection*{Case 1.1: It holds $2 r > \sqrt{d} \diam(\Omega)$}
The conditions imposed on $Q^*$ and $r$ imply that for all $k \in \IN$ we have $\Omega_k \subset 4 Q^*$. In this case, use the weak-type $(1 , 1)$ estimate of the localized maximal operator and the $\L^2$-boundedness of $T_{\lambda}$ (cf.\@ Propositions~\ref{Prop: Resolvent} and~\ref{Prop: Pressure estimate Neumann}, notice that the constants only depend on $d$, $\theta$, and $\mu$) to obtain
\begin{align*}
 \lvert \{ x \in Q &: M_{2 Q^*} (\lvert E_0 T_{\lambda} f \rvert^2) (x) > \iota \} \rvert \leq \frac{C}{\iota} \int_{\Omega_k} \lvert T_{\lambda} f \rvert^2 \; \d x \leq \frac{C}{\iota} \int_{(4 Q^*) \cap \Omega} \lvert f \rvert^2 \; \d x.
\end{align*}

\subsection*{Case 1.2: It holds $2 r \leq \sqrt{d} \diam(\Omega)$ and $(2 Q^*) \cap \partial \Omega_k \neq \emptyset$}
Let $y \in (2 Q^*) \cap \partial \Omega_k$ and let $\cQ := Q (y , 4 r) \subset \IR^d$ be the cube with center $y$ and $\diam(\cQ) = 4 r$. In this case, it holds $2 Q^* \subset \cQ$. Define functions $v$ and $w$ as follows. Let $\widetilde{B}_{\mu , k}$ denote the Stokes operator subject to Neumann-type boundary conditions on $(8 \cQ) \cap \Omega_k$. Notice that the restriction of $f$ to $(8 \cQ) \cap \Omega_k$ is still in $\C^{\infty}_{\sigma} (\overline{(8 \cQ) \cap \Omega_k})$ and thus define
\begin{align*}
 v := (\lambda + \widetilde{B_{\mu}})^{-1} R_{(8\cQ) \cap \Omega_k} f \qquad \text{and} \qquad w := u - v,
\end{align*}
where $R_{(8 \cQ) \cap \Omega_k}$ denotes the restriction operator to $(8 \cQ) \cap \Omega_k$. Analogously, define the pressures $\vartheta$ associated to $v$ and $R_{(8 \cQ) \cap \Omega_k} f$ and $\psi := \phi - \vartheta$. Thus, in the sense of distributions it holds
\begin{align*}
\left\{ \begin{aligned}
 \lambda v - \Delta v + \nabla \vartheta &= R_{(8 \cQ) \cap \Omega_k} f && \text{in } (8 \cQ) \cap \Omega_k \\
 \divergence(v) &= 0 && \text{in } (8 \cQ) \cap \Omega_k \\
 \{ D v + \mu [D v]^{\top} \} n - \vartheta n &= 0 && \text{on } \partial [(8 \cQ) \cap \Omega_k]
\end{aligned} \right.
\end{align*}
and
\begin{align*}
\left\{ \begin{aligned}
 \lambda w - \Delta w + \nabla \psi &= 0 && \text{in } (8 \cQ) \cap \Omega_k \\
 \divergence(w) &= 0 && \text{in } (8 \cQ) \cap \Omega_k \\
 \{ D w + \mu [D w]^{\top} \} n - \psi n &= 0 && \text{on } (8 \cQ) \cap \partial \Omega_k.
\end{aligned} \right.
\end{align*}
Here, $n$ denotes the outward unit normal vector corresponding to the set $(8 \cQ) \cap \Omega_k$. Notice that in $(8 \cQ) \cap \Omega_k$ the identities $u = v + w$ and $\phi = \vartheta + \psi$ hold and that $w$ and $\vartheta$ are in general non-zero as there is no boundary condition on the remaining boundary part $\partial [(8 \cQ) \cap \Omega_k] \setminus [(8 \cQ) \cap \partial \Omega_k]$ imposed. Let $\widetilde{v}, \widetilde{\nabla v}, \widetilde{\vartheta}, \widetilde{w}, \widetilde{\nabla w}$, and $\widetilde{\psi}$ denote the extensions by zero to $\IR^d$. Then for $\iota > 0$, we have
\begin{align*}
 \lvert \{ x \in Q &: M_{2 Q^*} (\lvert E_0 T_{\lambda} f \rvert^2) (x) > \iota \} \rvert \\
 &\qquad\leq \lvert \{ x \in Q : M_{2 Q^*} (\lvert \lvert \lambda \rvert \widetilde{v} \rvert^2 + \lvert \lvert \lambda \rvert^{1 / 2} \widetilde{\nabla v} \rvert^2 + \lvert \lvert \lambda \rvert^{1 / 2} \widetilde{\vartheta} \rvert^2) (x) > \iota / 4 \} \rvert \\
 &\qquad\qquad + \lvert \{ x \in Q : M_{2 Q^*} (\lvert \lvert \lambda \rvert \widetilde{w} \rvert^2 + \lvert \lvert \lambda \rvert^{1 / 2} \widetilde{\nabla w} \rvert^2 + \lvert \lvert \lambda \rvert^{1 / 2} \widetilde{\psi} \rvert^2) (x) > \iota / 4 \} \rvert \\
 &\qquad =: \mathrm{I} + \mathrm{II}.
\end{align*}
The first term is controlled by the weak-type $(1 , 1)$ estimate of the localized maximal operator followed by Proposition~\ref{Prop: Resolvent}~\eqref{Prop: Neumann} and Proposition~\ref{Prop: Pressure estimate Neumann} yielding
\begin{align*}
 \mathrm{I} \leq \frac{C}{\iota} \int_{(2 Q^*) \cap \Omega_k} \big( \lvert \lvert \lambda \rvert v \rvert^2 + \lvert \lvert \lambda \rvert^{1 / 2} \nabla v \rvert^2 + \lvert \lvert \lambda \rvert^{1 / 2} \vartheta \rvert^2 \big) \; \d x \leq \frac{C}{\iota} \int_{(32 Q^*) \cap \Omega_k} \lvert f \rvert^2 \; \d x,
\end{align*}
where $C > 0$ depends only on $d$, $\theta$, and $\mu$. \par
The second term, $\mathrm{II}$, is controlled by the embedding $\L^{q / 2 , \infty} (2 Q^*) \hookrightarrow \L^{q / 2} (2 Q^*)$, the $\L^{q / 2}$-boundedness of the localized maximal operator, and the fact $2 Q^* \subset \cQ$. Notice that the constants in these estimates depend only on $d$ and $q$ so that
\begin{align*}
 \mathrm{II} \leq \frac{C}{\iota^{q / 2}} \int_{\cQ \cap \Omega_k} \big( \lvert \lvert \lambda \rvert w \rvert^q + \lvert \lvert \lambda \rvert^{1 / 2} \nabla w \rvert^q + \lvert \lvert \lambda \rvert^{1 / 2} \psi \rvert^q \big) \; \d x.
\end{align*}
Next, apply Proposition~\ref{Prop: Sobolev} with $\Xi := \cQ \cap \Omega_k$ combined with~\eqref{Eq: d-set property}, to deduce
\begin{align}
\label{Eq: Hard term in reverse Holder}
\begin{aligned}
 \mathrm{II} &\leq \frac{C}{\iota^{q / 2}} r^d \bigg\{ r^{1 - d / 2} \lvert \lambda \rvert^{1 / 2} \bigg( \int_{\cQ \cap \Omega_k} \big( \lvert \lambda \rvert \lvert \nabla w \rvert^2 + \lvert \nabla^2 w \rvert^2 + \lvert \nabla \psi \rvert^2 \big) \; \d x \bigg)^{\frac{1}{2}} \\
 &\qquad + r^{- d / 2} \bigg( \int_{\cQ \cap \Omega_k} \big( \lvert \lvert \lambda \rvert w \rvert^2 + \lvert \lvert \lambda \rvert^{1 / 2} \nabla w \rvert^2 + \lvert \lvert \lambda \rvert^{1 / 2} \psi \rvert^2 \big) \; \d x \bigg)^{\frac{1}{2}} \bigg\}^q.
\end{aligned}
\end{align}
Due to~\eqref{Eq: d-set property} the constant $C > 0$ also depends on $\diam(\Omega)$ and on $r_0 > 0$, where $r_0$ is such that $B(0 , r_0) \subset \Omega - \{ x \}$ for some $x \in \Omega$. The second term on the right-hand side is estimated by virtue of $u = v + w$ and $\phi = \vartheta + \psi$, Propositions~\ref{Prop: Resolvent}~\eqref{Prop: Neumann} and~\ref{Prop: Pressure estimate Neumann}, and $8 \cQ \subset 32 Q^*$ as
\begin{align}
\label{Eq: Cheat estimate}
\begin{aligned}
 &\bigg( \int_{\cQ \cap \Omega_k} \big( \lvert \lvert \lambda \rvert w \rvert^2 + \lvert \lvert \lambda \rvert^{1 / 2} \nabla w \rvert^2 + \lvert \lvert \lambda \rvert^{1 / 2} \psi \rvert^2 \big) \; \d x \bigg)^{\frac{1}{2}} \\
 &\qquad \leq \bigg( \int_{\cQ \cap \Omega_k} \lvert T_{\lambda} f \rvert^2 \; \d x \bigg)^{\frac{1}{2}} + \bigg( \int_{\cQ \cap \Omega_k} \big( \lvert \lvert \lambda \rvert v \rvert^2 + \lvert \lvert \lambda \rvert^{1 / 2} \nabla v \rvert^2 + \lvert \lvert \lambda \rvert^{1 / 2} \vartheta \rvert^2 \big) \; \d x \bigg)^{\frac{1}{2}} \\
 &\qquad \leq C \bigg\{\bigg( \int_{(32 Q^*) \cap \Omega_k} \lvert T_{\lambda} f \rvert^2 \; \d x \bigg)^{\frac{1}{2}} + \bigg(\int_{(32 Q^*) \cap \Omega_k} \lvert f \rvert^2 \; \d x \bigg)^{\frac{1}{2}} \bigg\}.
\end{aligned}
\end{align}
The first term on the right-hand side in~\eqref{Eq: Hard term in reverse Holder} is estimated by virtue of Proposition~\ref{Prop: Reverse regularity estimates} as
\begin{align*}
 &\int_{\cQ \cap \Omega_k} \big( \lvert \lambda \rvert \lvert \nabla w \rvert^2 + \lvert \nabla^2 w \rvert^2 + \lvert \nabla \psi \rvert^2 \big) \; \d x \\
 &\qquad \leq C \bigg( \lvert \lambda \rvert^2 \int_{(2 \cQ) \cap \Omega_k} \lvert w \rvert^2 \; \d x + \frac{1}{r^2} \int_{(2 \cQ) \cap \Omega_k} \big( \lvert \nabla w \rvert^2 + \lvert \psi \rvert^2 \big) \; \d x \bigg).
\end{align*}
Employing Caccioppoli's inequality, Lemma~\ref{Lem: Caccioppoli}, to the first term on the right-hand side finally delivers
\begin{align}
\label{Eq: Application of Caccioppoli}
\begin{aligned}
 &\int_{\cQ \cap \Omega_k} \big( \lvert \lambda \rvert \lvert \nabla w \rvert^2 + \lvert \nabla^2 w \rvert^2 + \lvert \nabla \psi \rvert^2 \big) \; \d x \\
 &\qquad \leq C \bigg( \frac{\lvert \lambda \rvert}{r^2} \int_{(4 \cQ) \cap \Omega_k} \lvert w \rvert^2 \; \d x + \frac{1}{r^2} \int_{(4 \cQ) \cap \Omega_k} \big( \lvert \nabla w \rvert^2 + \lvert \psi \rvert^2 \big) \; \d x \bigg).
\end{aligned}
\end{align}
Combining~\eqref{Eq: Hard term in reverse Holder} and~\eqref{Eq: Application of Caccioppoli} one finds analogously to~\eqref{Eq: Cheat estimate} that
\begin{align*}
 \mathrm{II} \leq \frac{C}{\iota^{q / 2}} \bigg\{\bigg( \int_{(32 Q^*) \cap \Omega_k} \lvert T_{\lambda} f \rvert^2 \; \d x \bigg)^{\frac{1}{2}} + \bigg(\int_{(32 Q^*) \cap \Omega_k} \lvert f \rvert^2 \; \d x \bigg)^{\frac{1}{2}} \bigg\}^q.
\end{align*}

\subsection*{Case 1.3: It holds $2 r \leq \sqrt{d} \diam(\Omega)$ and $2 Q^* \cap \partial \Omega_k = \emptyset$}
This case is treated similar as the previous case. The only difference is that there is no need to introduce the cube $\cQ$, thus, by setting $\cQ := 2 Q^*$ in Case~1.2, the proof is literally the same.

\subsection*{Conclusion of the proof of Case 1}
Notice that the family $(T_{\lambda})_{\lambda \in \S_{\theta}}$ is uniformly bounded from $\cL^2_{\sigma} (\Omega_k)$ into $\L^2 (\Omega_k ; \IC^{d + d^2 + 1})$ and that all estimates proven in Case~1 are uniform with respect to $\lambda$. Thus we conclude by Theorem~\ref{Thm: Modified Shen domain} that for all $2 < p < 2 d / (d - 2)$ the family $(T_{\lambda})_{\lambda \in \S_{\theta}}$ satisfies a uniform boundedness estimate from $\cL^p_{\sigma} (\Omega_k)$ into $\L^p (\Omega_k ; \IC^{d + d^2 + 1})$ for all $f \in \C_{\sigma}^{\infty} (\overline{\Omega_k})$ and by density for all $f \in \cL^p_{\sigma} (\Omega_k)$. In particular, this holds true for each of the mappings
\begin{align*}
 T_{\lambda}^1 : f \mapsto \lvert \lambda \rvert u, \qquad T_{\lambda}^2 : f \mapsto \lvert \lambda \rvert^{1 / 2} \nabla u, \qquad \text{and} \qquad T_{\lambda}^3 : f \mapsto \lvert \lambda \rvert^{1 / 2} \phi.
\end{align*}
Now, by the approximation argument carried out in the proof of Theorem~\ref{Thm: H2 regularity on convex domains}, the uniform boundedness of these mappings also follows on the domain $\Omega$.

\subsection*{Case 2: It holds $2d / (d + 2) < p < 2$}
To deduce the second case we argue by duality. Thus, let $q := 2d / (d - 2)$ if $d \geq 3$ and let $q > 2$ if $d = 2$. Let again $\Omega_k$ be a bounded, convex, and smooth domain introduced in Remark~\ref{Rem: Approximating sequence}. Let $F \in \C_c^{\infty} (\Omega_k ; \IC^{d \times d})$ and let $u$ be given by $u := (\lambda + \cB_{\mu , k})^{-1} \divergence(F)$ and let $\phi$ denote the associated pressure. Consider the operator
\begin{align*}
 S_{\lambda} F := \begin{pmatrix} \lvert \lambda \rvert^{1 / 2} u \\ \nabla u \\ \phi \end{pmatrix}.
\end{align*}
Notice that $S_{\lambda}$ extends to a bounded operator from $\L^2 (\Omega_k ; \IC^{d \times d})$ to $\L^2 (\Omega_k ; \IC^{d + d^2 + 1})$ by Propositions~\ref{Prop: Resolvent} and~\ref{Prop: Pressure estimate Neumann} and that its operator norm is bounded by a constant depending merely on $d$, $\mu$, and $\theta$. For such a smooth $F$, the assumptions of Theorem~\ref{Thm: Modified Shen domain} are verified analogously to Case~1. It follows that each of the mappings
\begin{align*}
 S_{\lambda}^1 : F \mapsto \lvert \lambda \rvert^{1 / 2} u, \qquad S_{\lambda}^2 : F \mapsto \nabla u, \qquad \text{and} \qquad S_{\lambda}^3 : F \mapsto \phi
\end{align*}
gives rise to a uniformly bounded family of operators on $\L^r (\Omega_k)$ for each $2 < r < q$. The approximation argument carried out in the proof of Theorem~\ref{Thm: H2 regularity on convex domains}, implies the uniform boundedness of these mappings on the domain $\Omega$. By duality, we conclude from the boundedness properties of the mapping $T_{\lambda}^1$ from Case~1 and from the boundedness properties of $S_{\lambda}^1$ that there exists $C > 0$ such that for all $\lambda \in \S_{\theta}$ and all $f \in \cL^p_{\sigma} (\Omega)$ it holds
\begin{align}
\label{Eq: Estimate for small p}
 \| \lambda (\lambda + B_{\mu})^{-1} f \|_{\cL^p_{\sigma} (\Omega)} + \lvert \lambda \rvert^{1 / 2} \| \nabla (\lambda + B_{\mu})^{-1} f \|_{\L^p (\Omega ; \IC^{d^2})} \leq C \| f \|_{\cL^p_{\sigma} (\Omega)}.
\end{align}
The estimate on $\nabla (\lambda + \cB_{\mu})^{-1} \divergence$ follows from the boundedness of $S_{\lambda}^2$ and duality.
\end{proof}

\begin{remark}
To control the pressure in $\L^p$ for $2d / (d + 2) < p < 2$ is difficult. Intuitively, one would employ~\eqref{Eq: Variational Neumann} to write
\begin{align*}
 \| \phi \|_{\L^p (\Omega)} = \sup_{\substack{g \in \L^{p^{\prime}} (\Omega) \\ \| g \|_{\L^{p^{\prime}}(\Omega) \leq 1}}} \Big\lvert \int_{\Omega} \phi \, \overline{\divergence \nabla \Delta^{-1}_D g} \; \d x \Big\rvert = \sup_{\substack{g \in \L^{p^{\prime}} (\Omega) \\ \| g \|_{\L^{p^{\prime}}(\Omega) \leq 1}}} \Big\lvert \int_{\Omega} \nabla u \cdot \overline{\nabla \nabla \Delta^{-1}_D g} \; \d x \Big\rvert.
\end{align*}
However, one cannot control $\nabla \nabla \Delta^{-1}_D g$ in $\L^{p^{\prime}}$ due to the counterexample in~\cite[Prop.~2]{Fromm}.
\end{remark}


\end{document}